\documentclass[english,leqno,10pt]{article}
\usepackage{geometry}
\usepackage[T1]{fontenc}
\usepackage[latin9]{inputenc}
\usepackage{babel}
\usepackage{float}
\usepackage{amsmath}
\usepackage{stmaryrd}
\usepackage{amsthm}
\usepackage{amssymb}
\usepackage{cancel}
\usepackage{graphicx,xcolor}
\usepackage{esint}
\usepackage[toc,page]{appendix}
\usepackage{chngcntr}
\usepackage{apptools}
\usepackage{algorithm,algpseudocode}
\newcommand{\R}{{\mathbb R}}
\newcommand{\E}{{\mathbb E}}\makeatletter
\newcommand{\N}{{\mathbb N}}

\floatstyle{ruled}
\newfloat{algorithm}{tbp}{loa}
\providecommand{\algorithmname}{Algorithm}
\floatname{algorithm}{\protect\algorithmname}
\theoremstyle{definition}
\newtheorem*{condition*}{\protect\conditionname}
\theoremstyle{plain}
\newtheorem{thm}{\protect\theoremname}[section]
\theoremstyle{plain}
\newtheorem{prop}[thm]{\protect\propositionname}
\theoremstyle{definition}

\theoremstyle{plain}
\newtheorem{lem}[thm]{\protect\lemmaname}
\theoremstyle{plain}
\newtheorem{cor}[thm]{\protect\corollaryname}
\theoremstyle{plain}
\newtheorem{remark}[thm]{\protect\remarkname}
\theoremstyle{plain}

\theoremstyle{plain}
\newtheorem*{assumption*}{\protect\assumptionname}
\numberwithin{equation}{section}
\makeatother
\providecommand{\assumptionname}{Assumption}
\providecommand{\conditionname}{Condition}
\providecommand{\corollaryname}{Corollary}
\providecommand{\definitionname}{Definition}
\providecommand{\lemmaname}{Lemma}
\providecommand{\propositionname}{Proposition}
\providecommand{\theoremname}{Theorem}
\providecommand{\remarkname}{Remark}
\providecommand{\examplename}{Example}

\begin{document}
\title{Weak and strong error analysis for mean-field rank-based particle approximations of one dimensional viscous scalar conservation laws}
\author{O. Bencheikh and B. Jourdain\thanks{Cermics, \'Ecole des Ponts, INRIA, Marne-la-Vall\'ee, France. E-mails : benjamin.jourdain@enpc.fr, oumaima.bencheikh@enpc.fr. The authors would like to acknowledge financial support from Universit\'e Mohammed VI Polytechnique.}}
\maketitle

\begin{abstract}
In this paper, we analyse the rate of convergence of a system of $N$ interacting particles with mean-field rank-based interaction in the drift coefficient and constant diffusion coefficient. We first adapt arguments by Kolli and Shkolnikhov \cite{ShkoKol} to check trajectorial propagation of chaos with optimal rate $N^{-1/2}$ to the associated stochastic differential equations nonlinear in the sense of McKean. We next relax the assumptions needed by Bossy \cite{MBoss} to check the convergence in $L^1\left(\R\right)$ with rate ${\mathcal O}\left(\frac{1}{\sqrt N} + h\right)$ of the empirical cumulative distribution function of the Euler discretization with step $h$ of the particle system to the solution of a one dimensional viscous scalar conservation law. Last, we prove that the bias of this stochastic particle method behaves as ${\mathcal O}\left(\frac{1}{N} + h\right)$. We provide numerical results which confirm our theoretical estimates.
\end{abstract}


\section{Introduction}

The order of weak convergence in terms of the number $N$ of particles for the approximation of diffusions nonlinear in the sense of McKean solving
\begin{equation}
   X_t=X_0+\int_0^t\varsigma(s,X_s,\mu_s)\,dW_s+\int_0^t\vartheta(s,X_s,\mu_s)\,ds\mbox{ with $\mu_s$ denoting the probability distribution of $X_s$},\label{edsnlg}
\end{equation}
by the systems of $N$ interacting particles
\begin{equation}
   \check X^{i,N}_t=\check X^i_0+\int_0^t\varsigma(s,\check X^{i,N}_s,\check \mu^N_s)\,dW^i_s+\int_0^t\vartheta(s,\check X^{i,N}_s,\check \mu^N_s)\,ds,\;i\in\{1,\hdots,N\}\mbox{ with }\check\mu^N_s=\frac{1}{N}\sum_{i=1}^N\delta_{\check X^{i,N}_s},\label{syspart}
\end{equation}
has been recently investigated in several papers \cite{Mischler,BenJou,Chassa,ChaudruFrikha}.
Here $\left(W_t\right)_{t\ge 0}$ is a $d$-dimensional Brownian motion independent of the initial $\R^n$-valued random vector $X_0$, $\left(W^i,\check X^i_0\right)_{i\ge 1}$ are i.i.d. copies of $\left(W,X_0\right)$, $\varsigma:[0,T]\times\R^n\times{\mathcal P}(\R^n)\rightarrow \R^{n\times d}$ and $\vartheta:[0,T]\times\R^n\times{\mathcal P}(\R^n)\rightarrow \R^{n}$ with ${\mathcal P}(\R^n)$ denoting the space of Borel probability distributions on $\R^n$.
Typically, under some regularity assumptions, the weak error (or bias) $\displaystyle \E\left[\varphi\left(\check X^{1,N}_T\right)-\varphi\left(X_T\right)\right]=\E\left[\int_{\R^n}\varphi(x)\check\mu^N_T(dx)\right]-\int_{\R^n}\varphi(x)\mu_T(dx)$ for test functions $\varphi:\R^n\to\R$ is of order $N^{-1}$. On the other hand, it is well known since \cite{Szn91} that the strong error $\max\limits_{1\le i\le N}\E\left[\sup_{t\in[0,T]}\left|\check X^{i,N}_t-X^i_t\right|\right]$ where $(X^i_t)_{t\in[0,T]}$ solves \eqref{edsnlg} with $(X_0,W)$ replaced by $(X^i_0,W^i)$ is of order $N^{-1/2}$. From a numerical perspective, this implies that simulating $N$ independent copies of the system with $N$ particles leads to a bias and a statistical error both of order $N^{-1}$ which is also the order of the global error resulting from one single simulation of the system with $N^2$ particles. When the computation time of the interaction is quadratic, then the cost of these $N$ copies is of order $N^3$ compared to the order $N^4$ of the computation cost of the system with $N^2$ particles.

In Theorem 6.1 \cite{Mischler}, Mischler, Mouhot and Wennberg prove that for $\varsigma$ uniformly elliptic and not depending on the time and measure arguments, $\displaystyle \sup_{t\in[0,T]}\left|\E\left[\varphi\left(\check X^{1,N}_t\right)\right]-\int_{\R^n}\varphi(x)\mu_t(dx)\right|\le \frac{C}{N}$ when $\varphi$ is Lispchitz and has some Sobolev regularity and $\vartheta(t,x,\mu)=Ax+\int U(x-y)\mu(dy)$ for some constant matrix $A$ and some function $U$ with Sobolev regularity.
In \cite{BenJou}, we consider the case of interaction through moments: 
$\left(\begin{array}{c}\varsigma \\\vartheta \end{array}\right)(s,x,\mu)=\left(\begin{array}{c}\sigma\\b \end{array}\right)\left(s,\displaystyle \int_{\R^n}\alpha(x)\mu(dx),x\right)$. When $\alpha:\R^n\to\R^p$, $\sigma:[0,T]\times\R^p\times \R^n\to\R^{n\times d}$, $b:[0,T]\times\R^p\times \R^n\to\R^{n}$ and $\varphi:\R^n\to\R$ are twice continuously differentiable with bounded derivatives and Lipschitz second order derivatives and $\sigma\sigma^*$ is globally Lipschitz, we obtain: $$\exists \, C<\infty,\;\forall h\in[0,T],\;\forall N\in\N^*,\quad \sup_{t\in[0,T]}\left|\E\left[\varphi\left(\check X^{1,N,h}_t\right)\right]-\int_{\R^n}\varphi(x)\mu_t(dx)\right|\le C\left(\frac{1}{N}+h\right)$$ where $\check X^{i,N,0}_t$ denotes the particle system \eqref{syspart} and $\check X^{i,N,h}_t$ its Euler discretization with step $h$ when $h>0$.

When $n=d$, in Theorem 2.17 \cite{Chassa}, Chassagneux, Szpruch and Tse  prove the expansion of the bias $$\displaystyle \E\left[\Phi\left(\check \mu^N_T\right)\right]-\Phi\left(\mu_T\right)=\sum_{j=1}^{k-1}\frac{C_j}{N^j}+{\mathcal O}\left(\frac{1}{N^k}\right),$$ for time-homogeneous coefficients $\varsigma$ and $\vartheta$, $(2k+1)$-times differentiable with respect to both the spatial coordinates and the probability measure argument (for the notion of lifted differentiability introduced by Lions in his lectures at the Coll\`ege de France \cite{cardaliaguet}) with $\varsigma$ bounded and $X_0$ admitting a finite moment of order $(2k+1)$. They assume the same regularity on the test function $\Phi$ on the space of probability measures on $\R^d$ which is possibly nonlinear: $\Phi(\mu)$ is not necessarily of the form $\int_{\R^d}\varphi(x)\mu(dx)$.
In Theorem 3.6 \cite{ChaudruFrikha}, under uniform ellipticity, Chaudru de Raynal and Frikha prove $\left|\E\left[\Phi\left(\check \mu^N_T\right)\right]-\Phi\left(\mu_T\right)\right|\le \frac{C}{N}$ when $\Phi$ has bounded and H\"older continuous first and second order linear functional derivatives and $\varsigma\varsigma^*$ and $\vartheta$ are bounded and globally H\"older continuous with respect to the spatial variables and have bounded and H\"older continuous first and second order linear functional derivatives with respect to the measure argument. Notice that the existence of a linear functional derivative requires less regularity than the Fr\'echet differentiability of the lift since the lifted derivative is the gradient of the linear functional derivative with respect to the spatial variables.\\

Our aim in the present paper is to check that the ${\mathcal O}\left(\frac{1}{N}+h\right)$ behaviour of the weak error for the Euler discretization with step $h$ of the system with $N$ particles generalizes to a stochastic differential equation with an even discontinuous drift coefficient. This SDE is one-dimensional ($n=d=1$) and has a constant diffusion coefficient $\varsigma(s,x,\mu)=\sigma$ for $\sigma>0$. The drift coefficient writes $\vartheta(s,x,\mu)=\lambda(\mu((-\infty,x]))$ where $\R\times{\mathcal P}(\R)\ni(x,\mu)\mapsto \mu\left((-\infty,x]\right)$ is not even continuous and $\lambda$ is the derivative of a $C^1$ function $\Lambda:[0,1] \to \R$:
\begin{equation}\label{edsnl}
  \left\{
      \begin{aligned}
        \displaystyle &X_t = X_0 + \sigma W_t + \int_0^t \lambda\left(F(s,X_s)\right)\,ds,\quad t\in[0,T]\\
        &F(s,x) = \mathbb{P}\left(X_s \le x\right),\;\forall (s,x)\in[0,T]\times\R.\\
      \end{aligned}
    \right.
\end{equation}
We denote by $m$ the probability distribution of $X_0$ and by $F_0$ its cumulative distribution function. According to Section \ref{secmainres} in the paper \cite{BossTal} specialized to the case $\Lambda(u) = u^2/2$ and Proposition 1.2 and Theorem 2.1 \cite{JDP} for a general function $\Lambda$, weak existence and uniqueness hold for the SDE \eqref{edsnl}. By \cite{Veret}, it actually admits a unique strong solution. For $t>0$, by the Girsanov theorem, the law $\mu_t$ of $X_t$ admits a density  $p(t,x)$ with respect to the Lebesgue measure (see Lemma \ref{ctyF} below). The function $p(t,x)$ is a weak solution to the Fokker-Planck equation $ \partial_t p(t,x) + \partial_x \left(\lambda(F(t,x)) p(t,x)\right)  = \frac{\sigma^2}{2} \partial_{xx}p(t,x)$. By integration with respect to the spatial variable $x$, we deduce that $F(t,x)$ is a weak solution to the following viscous conservation law:
\begin{equation}\label{vpde}
  \left\{
      \begin{aligned}
        &\partial_t F(t,x) + \partial_x\Big(\Lambda(F(t,x))\Big) = \frac{\sigma^2}{2}\partial_{xx} F(t,x),\\
        &F_0(x) = m \left((-\infty,x]\right).\\
      \end{aligned}
    \right.
  \end{equation}
  
The corresponding particle dynamics is 
\begin{align}\label{PSDEbis}
  \displaystyle  \breve X^{i,N}_t = X^{i}_0 + \sigma W^{i}_t + \int_0^t \lambda\left(\frac 1 N\sum \limits_{j=1}^{N} \mathbf{1}_{ \left\{  \breve X^{j,N}_s \leq \breve X^{i,N}_s \right\} } \right)\,ds, \quad 1\le i\le N,\;t\in[0,T].
\end{align}
As for the initial positions $\left(X^{i}_0\right)_{i \ge 1}$, we will consider both cases of the random initialization ($\left(X^i_0\right)_{i\ge 1}$ i.i.d. according to $m$) and an optimal deterministic initialization which will be made precise in Section \ref{secmainres}.

In fact, for $1\le i\le N$, the coefficient $\lambda(i/N)$ is close to
\begin{align}\label{driftPart}
    \lambda^{N}(i) = N \left( \Lambda\left(\frac{i}{N}\right) - \Lambda\left(\frac{i-1}{N}\right) \right)
\end{align}
so that the dynamics is close (see Corollary \ref{corec2systpart} for a precise statement) to the one introduced in \cite{JMal} : 
\begin{align}\label{PSDE}
  \displaystyle  X^{i,N}_t = X^{i}_0 + \sigma W^{i}_t + \int_0^t \lambda^N\left(\sum \limits_{j=1}^{N} \mathbf{1}_{ \left\{  X^{j,N}_s \leq X^{i,N}_s \right\} } \right)\,ds, \quad 1\le i\le N,\;t\in[0,T].
\end{align}
We denote by $\mu^{N}_t = \frac{1}{N}\sum \limits_{i=1}^N \delta_{X^{i,N}_t}$ the empirical measure and by $F^N(t,x)=\frac 1N\sum\limits_{i=1}^N{\mathbf 1}_{\{X^{i,N}_t\le x\}}$ the empirical cumulative distribution function at time $t$ of this second particle system.
Both dynamics are called rank-based models since the drift coefficient only depends on the rank of the $i$-th particle in the system. We call them mean-field rank-based since the interaction between the particles is also of mean-field type. The ability of rank-based models to reproduce stylized empirical properties observed on stock markets \cite{Fern}, has motivated their mathematical study \cite{BanFerKar}. By the Girsanov theorem, the stochastic differential equations \eqref{PSDEbis} and \eqref{PSDE} admit a unique weak solution and, according to  \cite{Veret}, they actually admit a unique strong solution. Under concavity of $\Lambda$, Jourdain and Malrieu \cite{JMal} prove the propagation of chaos with optimal rate $N^{-1/2}$ and study the long-time behaviour of the particle system \eqref{PSDE} and its mean-field limit \eqref{edsnl}. For the particle system \eqref{PSDEbis}, this study is extended by Jourdain and Reygner \cite{jrProp} when the diffusion coefficient is no longer constant but also of mean-field rank-based type. For this more general model and without the concavity assumption, Kolli and Shkolnikhov \cite{ShkoKol} recently proved the propagation of chaos with optimal rate $N^{-1/2}$ and convergence of the associated fluctuations when the initial probability measure $m$ admits a bounded density with respect to the Lebesgue measure. 
We choose to focus on the modified dynamics \eqref{PSDE} because when $y^1<y^2<\hdots<y^N$, then the distribution derivative of $x\mapsto \Lambda\left(\frac 1N\sum\limits_{i=1}^N{\mathbf 1}_{\{y^i\le x\}}\right)$ is $x\mapsto \frac 1N\sum\limits_{i=1}^N\lambda^N(i){\mathbf 1}_{\{y^i\le x\}}$ and not (when $\Lambda$ is not affine) $x\mapsto \frac 1N\sum\limits_{i=1}^N\lambda(i/N){\mathbf 1}_{\{y^i\le x\}}$.
For this reason, it is more closely connected to the PDE \eqref{vpde}. As our error analysis is based on a comparison of the mild formulation of the PDE \eqref{vpde} and the perturbed mild formulation satisfied by empirical cumulative distribution function of the Euler discretization of the particle system, we concentrate on \eqref{PSDE}, for which no extra error term appears in this perturbed version but we will also explain how our results extend to \eqref{PSDEbis}. This approach, also called integral representation technique, was used in \cite{TangTsai} to study certain system of particles and its limit.
Let us also introduce the Euler discretization with time-step $h \in (0,T]$ of \eqref{PSDE} : 
\begin{align}\label{edsnh}
    \displaystyle X^{i,N,h}_t &= X^i_0+ \sigma W^{i}_t +  \int_0^t \lambda^N\left(\sum \limits_{j=1}^{N} \mathbf{1}_{ \left\{ X^{j,N,h}_{\tau^{h}_{s}} \le X^{i,N,h}_{\tau^{h}_{s}} \right\}} \right)\,ds,\;1\le i\le N,\;t\in[0,T] \mbox{ where }\tau^h_s=\lfloor s/h\rfloor h.
\end{align}
The empirical cumulative distribution function of $\mu^{N,h}_t = \frac{1}{N}\sum \limits_{i=1}^N \delta_{X^{i,N,h}_t}$ is $ F^{N,h}(t,x) := \frac{1}{N} \sum \limits_{j=1}^N \mathbf{1}_{\left\{X^{j,N,h}_t \le x \right\}} $. It is natural and convenient to consider that $\tau^0_s=s$ and $\left(X^{i,N,0}_t\right)_{t\in[0,T],1\le i\le N}=\left(X^{i,N}_t\right)_{t\in[0,T],1\le i\le N}$. Using these notations, we then have by convention that $F^{N,0}(t,x) = F^{N}(t,x)$. Moreover, we will refer to  the empirical cumulative distribution function $F^{N,h}_0$ at initialization by $\hat F^N_0$ when choosing positions that are i.i.d. according to $m$ and by $\tilde F^N_0$ when choosing optimal deterministic initial positions. Finally let us define $\left(\breve X^{i,N,h}_t\right)_{t\in[0,T],1\le i\le N}$ like $\left(\breve X^{i,N}_t\right)_{t\in[0,T],1\le i\le N}$ by replacing $\lambda^N(k)$ by $\lambda(k/N)$ in \eqref{edsnh} and set $\breve\mu^{N,h}_t=\frac 1N \sum\limits_{i=1}^N\delta_{X^{i,N,h}_t}$ and $\breve F^{N,h}(t,x) := \frac{1}{N} \sum \limits_{j=1}^N \mathbf{1}_{\left\{\breve X^{j,N,h}_t \le x \right\}} $.

The paper is organized as follows. In Section $2$, we state our results. Taking advantage of the constant diffusion coefficient, we adapt the arguments in \cite{ShkoKol} to obtain propagation of chaos with optimal rate $N^{-1/2}$ for the particle systems \eqref{PSDE} and \eqref{PSDEbis} without any assumption on the initial probability measure $m$. Then we state that the strong rate of convergence of $\mu^{N,h}_t$ to $\mu_t$ for the Wasserstein distance with index one (or equivalently of $F^{N,h}(t,.)$ to $F(t,.)$ for the $L^1$ norm) is ${\mathcal O}\left(\frac{1}{\sqrt{N}}+h\right)$, a result already obtained long ago by Bossy \cite{MBoss} under more regularity assumptions on the initial probability measure $m$ and the function $\Lambda$. Our main result is that the weak rate of convergence is ${\mathcal O}(\frac{1}{N}+h)$. In Section $3$, we introduce the reordered particle system and establish the mild formulation of the PDE \eqref{vpde} satisfied by $F(t,x)$ and the perturbed version satisfied by $F^{N,h}(t,x)$. Section $4$ is dedicated to the proofs of the results in Section $2$. We finally provide numerical experiments in Section $5$ to illustrate our results. Beforehand, we introduce some additional notation.

\subsection*{Notation:}
\begin{itemize}
  \item  We denote by $L_{\Lambda} = \sup_{u \in [0,1]} \left| \lambda(u)\right|$ the Lipchitz constant of $\Lambda$. When $\lambda$ is also assumed to be Lipschitz continuous, we denote similarly its Lipschitz constant by $L_\lambda$.
\item For $ 1 \le p < \infty$, we denote by $L^p(\R)$ the space of measurable real valued functions which are $L^p$-integrable for the Lebesgue measure i.e. $f\in L^p$ if $\displaystyle \left\Vert f\right\Vert_{L^p} = \left( \int_{\R} |f(x)|^p \; dx \right)^{\frac{1}{p}}<\infty$.
\item The space $L^{\infty}(\R)$ refers to the space of almost everywhere bounded measurable real valued functions endowed with the norm $\left\Vert f\right\Vert_{L^{\infty}} = \inf\{ C \ge 0: |f(x)|\le C \text{ for almost every } x \in \R \} $.
\item The $*$ in $f*g$ stands for the standard convolution operator.
\item The set of natural numbers starting at one is denoted by $\N^*$; and the interval notation $\llbracket a,b\rrbracket$ for $a,b\in\N^*$ with $a\le b$ stands for the set of natural numbers between $a$ and $b$. 
\item For $x,y\in\R$, we set $x\wedge y=\min(x,y)$, $x\vee y=\max(x,y)$ and denote the positive part of $y$ by $y^{+} = \max(y,0)$.
\item We denote by $\Gamma$ the Gamma function defined by $\displaystyle \Gamma(x) = \int_{0}^{+ \infty}y^{x-1}\exp(-y)\,dy$ for $x\in (0,+\infty)$. 
\item To simplify the notations, when a function $g$ defined on $[0,T] \times \R$ and $x \in \R$, we may use sometimes the notation $g_0(x) := g(0,x)$.
\end{itemize}

\section{Main results}\label{secmainres}

We will state the propagation of chaos result with optimal rate $N^{-1/2}$ before giving the main results concerning the convergence of the empirical cumulative distribution function $F^{N,h}$ of the Euler discretization with time-step $h$ of the system with $N$ interacting particles towards its limit $F$. We will make an intensive use of the interpretation of the $L^1$-norm of their difference as the Wasserstein distance with index $1$ between $\mu^{N,h}_t=\frac 1N\sum\limits_{i=1}^n\delta_{X^{i,N,h}_t}$ and the law $\mu_t$ of $X_t$.

The Wasserstein distance of index $\rho\ge 1$ between two probability measures $\mu$ and $\nu$ on $\R^d$ is defined by 
$$\mathcal{W}_\rho^\rho(\mu, \nu)= \inf \left\{ \E\left[|X-Y|^\rho\right]; \text{Law}(X) = \mu, \text{Law}(Y) = \nu \right\}.$$
In dimension $d=1$, the Hoeffding-Fr\'echet or comonotone coupling given by the inverse transform sampling is optimal:
\begin{align}\label{Wasserstein}
  \displaystyle \mathcal{W}_\rho^\rho(\mu, \nu) = \int_0^1 \left|F^{-1}_{\mu}(u) - F^{-1}_{\nu}(u) \right|^\rho\,du
\end{align}
where $\R\ni x\mapsto F_{\eta}(x) = \eta\left(\left(-\infty;x\right] \right)$ and $(0,1)\ni u\mapsto F^{-1}_{\eta}(u) = \inf\left\{ x \in \R: F_{\eta}(x) \geq u\right\}$ respectively denote the cumulative distribution function and the quantile function of a probability measure $\eta$ on $\R$. Since $\int_0^1 \left|F^{-1}_{\mu}(u) - F^{-1}_{\nu}(u)\right|\,du= \int_{\R}\left|F_{\mu}(x) - F_{\nu}(x) \right|\,dx$, the $\mathcal{W}_1$ distance between two probability measures $\mu$ and $\nu$ on the real line is equal to the $L^1$-norm of the difference between the cumulative distribution functions of $\mu$ and $\nu$:

 \begin{equation}\label{w1fdr}
   \mathcal{W}_1(\mu,\nu)= \int_{\R}\left|F_{\mu}(x) - F_{\nu}(x) \right|\,dx.
\end{equation}

We will also take advantage of the dual formulation of the $\mathcal{W}_1$ distance which holds whatever $d\in\N^*$:
\begin{align}
    \mathcal{W}_1(\mu,\nu)= \sup_{\varphi \in \mathcal{L}} \left(\int_{\R^d}\varphi(x)\mu(dx)-\int_{\R^d}\varphi(x)\nu(dx) \right)\label{dualwass}
\end{align}
where $\mathcal{L}$ denotes the set of all $1$-Lipschitz function $\varphi:\R^d \to \R$. 

\subsection{Propagation of chaos}

Kolli and Shkolnikov \cite{ShkoKol} prove a quantitative propagation of chaos result at optimal rate $N^{-1/2}$ and convergence of the associated fluctuations for the particle system without time-discretization in the much more general and difficult case when the diffusion coefficient is also mean-field rank-based. Taking advantage of the constant diffusion coefficient, we are going to relax their assumptions on $\lambda$ and $m$ to prove the following result.
\begin{thm}\label{thmpropchaos}
  Let the initial positions $X^i_0$ be i.i.d. according to $m$ and $(X^i_t)_{t\ge 0}$ denote the solution to the stochastic differential equation nonlinear in the sense of McKean \eqref{edsnl} starting from $X^i_0$ and driven by $(W^i_t)_{t\ge 0}$. If $\lambda$ is Lispchitz continuous, then
   $$\forall \rho,T>0,\;\exists \, C<\infty,\;\forall N\in\N^*,\;\max_{1\le i\le N}\E\left[\sup_{t\in[0,T]}\left|X^i_t-X^{i,N}_t\right|^\rho+\sup_{t\in[0,T]}\left|X^i_t-\breve X^{i,N}_t\right|^\rho\right]\le C N^{-\rho/2}.$$
 \end{thm}
 The estimation $\E\left[\sup_{t\in[0,T]}\left|X^i_t-\breve X^{i,N}_t\right|^\rho\right]\le C N^{-\rho/2}$ follows from Theorem 1.6 \cite{ShkoKol} when $\lambda$ is differentiable with an H\"older continuous derivative and $m$ has a bounded density w.r.t. the Lebesgue measure and a finite moment of order $2+\varepsilon$ for some $\varepsilon>0$.
An immediate consequence of Theorem \ref{thmpropchaos} is to quantify the proximity of the two particles dynamics \eqref{PSDEbis} and \eqref{PSDE}.
\begin{cor}\label{corec2systpart}
  Assume that the initial positions $X^i_0$ are i.i.d. according to $m$ and that $\lambda$ is Lispchitz continuous. Then:
   $$\forall \rho,T>0,\;\exists C<\infty,\;\forall N\in\N^*,\;\max_{1\le i\le N}\E\left[\sup_{t\in[0,T]}\left|X^{i,N}_t-\breve X^{i,N}_t\right|^\rho\right]\le C N^{-\rho/2}.$$
\end{cor}
\begin{remark}
By Theorem 2.1 in \cite{Jabir}, propagation of chaos also holds at optimal rate $N^{-1/2}$ in total variation distance.
\end{remark}

\subsection{Initialization error}

In addition to the random initialization of the particles which permits to obtain propagation of chaos, we will also consider deterministic initial positions. 
\begin{itemize}
\item When choosing a random initialization, we denote by $\hat F^N_0(x)=\frac{1}{N}\sum \limits_{i=1}^N \mathbf{1}_{\left\{X^i_0 \le x \right\}}$ and $\hat \mu^{N}_0 = \frac{1}{N}\sum \limits_{i=1}^N \delta_{X^{i}_0}$ the empirical cumulative distribution function and the empirical measure of the $N$ first random variables in the sequence $(X^i_0)_{i\ge 1}$ i.i.d. according to $m$.
\item When choosing a deterministic initialization, we seek to construct a family $x^{N}_{1}\le x_2^N\le \hdots\le x^N_N$ of initial positions minimizing the $L^1$ norm of the difference between the piecewise constant function $\tilde F^N_0(x) = \frac{1}{N}\sum \limits_{i=1}^N \mathbf{1}_{\left\{x^N_{i} \le x \right\}} $ and $F_0$. According to \eqref{Wasserstein},
  $$\int_\R\left|\tilde F^N_0(x)-F_0(x)\right|\,dx=\sum \limits_{i=1}^N \int_{\frac{i-1}{N}}^{\frac{i}{N}} \left|x^N_{i} - F^{-1}_0(u) \right|\,du.$$
Since, as remarked in \cite{jrdcds}, for $i\in\llbracket1,N\rrbracket$, $\displaystyle y\mapsto N\int_{\frac{i-1}{N}}^{\frac{i}{N}} \left|y -F^{-1}_0(u) \right|\,du$ is minimal for $y$ equal to the median $F_0^{-1}\left(\frac{2i-1}{2N}\right)$ of the image of the uniform law on $\left[\frac{i-1}{N},\frac i N \right]$ by $F_0^{-1}$, we choose $X^i_0=x_i^N=F_0^{-1}\left(\frac{2i-1}{2N}\right)$. We denote  by $\tilde \mu^{N}_0 = \frac{1}{N}\sum \limits_{i=1}^N \delta_{F_0^{-1}\left(\frac{2i-1}{2N}\right)}$ the associated empirical measure.
\end{itemize}

The next proposition, discusses assumptions under which the $L^1$-norm of the difference between $F_0$ and $\hat F_0^N$ or $\tilde F_0^N$ is of order $N^{-1/2}$. 
\begin{prop}\label{propvitsqn}
 We denote for simplicity $\displaystyle \int_{\R}|x|^{2+}m(dx)<\infty$ the existence of $\varepsilon>0$ such that $\displaystyle \int_{\R}|x|^{2+\varepsilon}m(dx)<\infty$ and $\displaystyle \int_{\R}|x|^{2-}m(dx)<\infty$ the fact that $\displaystyle \int_{\R}|x|^{2-\varepsilon}m(dx)<\infty$ for each $\varepsilon\in (0,2]$. We have the following results concerning the ${\mathcal O}(N^{-1/2})$ behaviour of the errors:
\begin{equation}
  \begin{array}{ccccc}
      & & \displaystyle{\sup_{N\ge 1} }\sqrt{N}\E \left[\mathcal{W}_1\left(\hat \mu^N_0,m \right) \right]<\infty& &\\
      & & \Updownarrow & &\\
      \displaystyle \int_{\R}|x|^{2+}m(dx)<\infty & \Rightarrow &\displaystyle \int_{\R}\sqrt{F_0(x)(1-F_0(x))}\,dx<\infty & \Rightarrow &\displaystyle \int_{\R}|x|^{2}m(dx)<\infty
  \end{array} \label{stiniid}
\end{equation}
\begin{equation}
  \begin{array}{ccccc}
      \displaystyle \int_{\R}|x|^{2}m(dx)<\infty &\Rightarrow & \displaystyle{\sup_{x\ge 0}}\,x^2\left(F_0(-x)+1-F_0(x)\right)<\infty&\Rightarrow &\displaystyle \int_{\R}|x|^{2-}m(dx)<\infty\\
      & &\Updownarrow &  &  \\
      & & \displaystyle{\sup_{N \ge 1} }\sqrt{N} \, \mathcal{W}_1 \left(\tilde \mu^N_0 ,m\right)<\infty& .
  \end{array}\label{stindet}
\end{equation}
\end{prop}
For the statements \eqref{stiniid} which concern the random initialization, we refer to \cite{bobkovledoux} Section 3.1. The statements \eqref{stindet} concerning the optimal deterministic initialization are proved in Section \ref{proofinitdet}.

Concerning the weak error, since the empirical cumulative distribution function of i.i.d. samples is unbiased, $\E\left[\hat F_0^N(x)\right]=F_0(x)$ for all $N\ge 1$ and $x\in\R$ and then
\begin{equation}
   \int_\R \left|\E\left[\hat F^N_0(x)\right]-F_0(x)\right|\,dx=0.
   \label{iidsansbiais}
\end{equation} 
As for the deterministic initialization, we have that:
\begin{align}\label{w1opt}
   \displaystyle \int_\R \left|\tilde F_0^N(x)-F_0(x)\right|\,dx = &\int_{-\infty}^{F_0^{-1}\left(\frac{1}{2N}\right)} F_0(x)\,dx + \sum_{i=1}^{N-1} \int_{F_0^{-1}\left(\frac{2i-1}{2N}\right)}^{F_0^{-1}\left(\frac{2i+1}{2N}\right)}\left|F_0(x)- \frac{i}{N} \right|\,dx +\int_{F_0^{-1}\left(\frac{2N-1}{2N}\right)}^{+\infty}\left(1-F_0(x)\right)\,dx
\end{align}
where the integrand is not greater than $1/(2N)$. When $m([c,d])=1$ with $-\infty<c\le d<\infty$, the integrand vanishes outside the interval $[c,d]$. One then easily deduces the next proposition proved in \cite{jrdcds} by using the alternative formulation:
\begin{equation}
   \displaystyle \int_\R \left|\tilde F_0^N(x)-F_0(x)\right|\,dx = \int_0^1\left|\left(\tilde F_0^N\right)^{-1}(u)-F_0^{-1}(u)\right|\,du = \sum_{i=1}^N \int_{\frac{2i-1}{2N}}^{\frac{i}{N}}\left(F_0^{-1}(u)-F_0^{-1}\left(u-\frac{1}{2N}\right)\right)\,du.\label{alterw1}
\end{equation}
\begin{prop}\label{detL1init}
  When $m$ is compactly supported i.e. there exists $ -\infty<c\le d<\infty$ such that $m([c,d])=1$, then $$\mathcal{W}_1 \left(\tilde \mu^N_0 ,m\right) \le \frac{d-c}{2N}.$$
\end{prop}

\subsection{Strong and weak ${\cal W}_1$-errors}
We recall that $\lambda$ is the derivative of the $C^1$ function $\Lambda:[0,1]\to\R$.

Let us state our estimation of the strong error which is proved in Section \ref{secproofthm1}.

\begin{thm}\label{thm1}
 Assume that for some $\rho >1$, $\displaystyle \int_{\R}|x|^{\rho}m(dx) < \infty$ and assume either that the initial positions are optimal deterministic or the initial positions are i.i.d. according to $m$. Then
	$$\exists C< \infty, \forall N \in \N^{*}, \quad \sup_{t \leq T}\E\left[ \mathcal{W}_1\left(\mu^{N,0}_t,\mu_t \right)\right] \leq C \left(\E\left[ \mathcal{W}_1\left(\mu^{N}_0,m\right)\right] + \frac{1}{\sqrt{N}}\right). $$
	Moreover, if $\lambda$ is Lipschitz continuous then:
	$$\exists C< \infty, \forall N \in \N^{*}, \forall h \in (0,T], \quad \sup_{t \leq T}\E\left[ \mathcal{W}_1\left(\mu^{N,h}_t,\mu_t \right)\right] \le C \left(\E\left[ \mathcal{W}_1\left(\mu^{N}_0,m \right)\right] + \frac{1}{\sqrt{N}} + h\right). $$
\end{thm}

Combining the theorem with Proposition \ref{propvitsqn}, we have the following corollary:
\begin{cor}\label{cor1}
Assume that the initial positions are \begin{itemize}
\item either i.i.d. according to $m$ and $\displaystyle \int_{\R}\sqrt{F_0(x)(1 - F_0(x))}\,dx < \infty$,
  \item or optimal deterministic and $\displaystyle \sup_{x \ge 1} x\int_{x}^{+\infty} \left(F_0(-y) + 1 - F_0(y) \right)\,dy<\infty$.
\end{itemize} Then:
  $$\exists C< \infty, \forall N \in \N^{*}, \quad \sup_{t \leq T}\E\left[ \mathcal{W}_1\left(\mu^{N,0}_t,\mu_t \right)\right] \le \frac{C}{\sqrt{N}}. $$
  Moreover, if $\lambda$ is Lipschitz continuous then:
  $$\exists C< \infty, \forall N \in \N^{*}, \forall h \in (0,T], \quad \sup_{t \leq T}\E\left[ \mathcal{W}_1\left(\mu^{N,h}_t,\mu_t \right)\right] \le C\left(\frac{1}{\sqrt{N}} + h\right). $$
\end{cor}

Let us now state our main result, proved in Section \ref{secproofthm2}, concerning the weak error : the $L^1$-weak error between the empirical cumulative distribution function $F^{N,h}$ of the Euler discretization with time-step $h$ of the system with $N$ interacting particles and its limit $F$ is $\mathcal O\left(\frac{1}{N} +h \right)$. We denote by $\E\left[\mu^{N,h}_t\right]$ the probability measure on $\R$ defined by
$$\int_\R\varphi(x) \E\left[\mu^{N,h}_t\right](dx)= \E\left[\int_\R\varphi(x)\mu^{N,h}_t(dx)\right]=\E \left[\frac{1}{N}\sum \limits_{i=1}^N \varphi\left( X^{i,N,h}_t\right) \right]$$for each $\varphi:\R\to\R$ measurable and bounded. The cumulative distribution function of $\E\left[\mu^{N,h}_t\right]$ is equal to $\E\left[F^{N,h}(t,x)\right]$ and $\displaystyle \mathcal{W}_1 \left(\E\left[\mu^{N,h}_t\right],\mu_t\right) = \int_\R \left|\E\left[F^{N,h}(t,x)\right]-F(t,x)\right|\,dx$.
\begin{thm}\label{thmBias}
Assume that $\lambda$ is Lipschitz continuous and the initial positions are \begin{itemize}
\item either i.i.d. according to $m$ and $\displaystyle \int_{\R}\sqrt{F_0(x)(1 - F_0(x))}\,dx < \infty$,
  \item  or optimal deterministic and $\displaystyle \sup_{x \ge 1} x\int_{x}^{+\infty} \left(F_0(-y) + 1 - F_0(y) \right)\,dy<\infty$.
\end{itemize} Then:
	$$\exists C_b< \infty, \forall N \in \N^{*}, \forall h \in [0,T],\quad \sup_{t \leq T}\mathcal{W}_1 \left(\E\left[\mu^{N,h}_t\right],\mu_t\right) \le C_b \left(\mathcal{W}_1 \left(\E\left[\mu^{N}_0\right],m\right) +  \left(\frac{1}{N} + h \right)\right). $$
\end{thm}

Combining the theorem with \eqref{iidsansbiais} and Proposition \ref{detL1init}, we obtain the following corollary:
\begin{cor}\label{corBiais}
  Assume that $\lambda$ is Lipschitz continuous and the initial positions are\begin{itemize}
    \item  either i.i.d. according to $m$ and $\displaystyle \int_\R\sqrt{F_0(x)(1-F_0(x))}\,dx<\infty$,
   \item or optimal deterministic with $m$ compactly supported.
\end{itemize}Then:
  $$\exists C_b< \infty, \forall N \in \N^{*}, \forall h \in [0,T],\quad \sup_{t \leq T}\mathcal{W}_1 \left(\E\left[\mu^{N,h}_t\right],\mu_t\right) \le C_b\left(\frac{1}{N} + h \right). $$
\end{cor}
Using the dual formulation \eqref{dualwass} of the Wasserstein distance, we deduce that if $\varphi:\R\to\R$ is Lipschitz continuous with constant ${\rm Lip}(\varphi)$ then  $$\forall N \in \N^{*}, \forall t,h \in [0,T],\quad \left| \E \left[\frac{1}{N}\sum \limits_{i=1}^N \varphi\left( X^{i,N,h}_t\right) \right] - \E \left[\varphi(X_t)\right]\right| \le C_b{\rm Lip}(\varphi) \left(\frac{1}{N} + h\right).$$

\begin{remark}\label{remcheck}
  For the dynamics \eqref{PSDEbis} with initial positions deterministic and given by $x^N_{i} = F^{-1}_0\left(\frac{i}{N}\right)$ when $i=1,\ldots,N-1$ and $x^N_N = F^{-1}_0\left( 1 - \frac{1}{2N}\right)$,  Bossy \cite{MBoss} proved an estimation also dealing with the supremum of the expected error between $\breve F^{N,h}(t,x)$ and $F(t,x)$  similar to the last statement in Corollary \ref{cor1}: $$\exists C< \infty, \forall N \in \N^{*}, \forall h \in (0,T],\;\sup_{t \leq T}\E\left[ \mathcal{W}_1\left(\breve \mu^{N,h}_t,\mu_t \right)\right]+ \sup_{(t,x) \in [0,T]\times \R}\E\left[ \left|\breve F^{N,h}(t,x) - F(t,x)\right| \right]\le C\left(\frac{1}{\sqrt{N}} + h\right).$$
  She assumes additional regularity on the coefficient $\Lambda$, namely that $\Lambda$ is $C^3$, and on the initial measure $m$, namely that $F_0$ is $C^2$ bounded with bounded first and second order derivatives in $x$ and that $\exists \;M, \beta>0, \; \alpha \geq 0 $ such that $|\partial_x F_0(x)| \leq \alpha \exp\left(-\beta x^2/2 \right)$ when $|x| >M$. Her proof is based on the regularity of the backward Kolmogorov PDE associated with the generator of the diffusion \eqref{edsnl}. By contrast, our approach is based on a comparison of the mild formulation of the forward in time PDE \eqref{vpde} satisfied by $F(t,x)$ and the pertubed mild formulation satisfied by $F^{N,h}(t,x)$.

In fact, all the above results hold with $\mu^{N,h}_t$ replaced by $\breve \mu^{N,h}_t$. For those concerning $\displaystyle \sup_{t \leq T}\E\left[ \mathcal{W}_1\left(\mu^{N}_t,\mu_t \right)\right]$, we just need to add the assumption that $\lambda$ is H\"older continuous with exponent $1/2$ to ensure that $\sup\limits_{1\le i\le N}\sqrt{N}\left|\lambda^N(i)-\lambda(i/N)\right|$ $<\infty$. Notice that under Lipschitz continuity of $\lambda$, we even get $\sup\limits_{1\le i\le N}N\left|\lambda^N(i)-\lambda(i/N)\right|<\infty$. See Remark \ref{mildEDSbis} below, where we outline how to adapt the proofs.
\end{remark}

\section{Dynamics of the reordered particle system and mild formulations}
The reordering of mean-field rank-based particle systems without time discretization has been first introduced in \cite{JSPA} and has proved to be a very useful tool in the study of the limit $N\to\infty$ with vanishing viscosity (the parameter $\sigma$ depends on $N$ and tends to $0$ as $N\to\infty$) \cite{JAppPro,JMelWoy} (the latter when the driving Brownian motions are replaced by symmetric $\alpha$-stable L\'evy processes with $\alpha>1$), the long time behaviour of both the particle system and its mean-field limit \cite{jrProp} and the small noise limit $\sigma\to 0$ of the particle system \cite{jrSmall}. Before deriving the dynamics of the reordering of the Euler discretization \eqref{edsnh}, let us check the existence of the density $p(t,x)$ of $X_t$ for $t>0$, which guarantees that, in the sense of distributions, $\partial_x \Lambda(F(t,x)) = \lambda(F(t,x))p(t,x)$ so that $F(t,x)$ is a weak solution of the viscous scalar conservation law \eqref{vpde}. Since the drift coefficient in \eqref{edsnl} is bounded, the Girsanov theorem easily implies the next statement.
\begin{lem}\label{ctyF}
  For $t>0$, $X_t$ admits a density $p(t,x)$ with respect to the Lebesgue measure.
\end{lem}

Let for each $t\ge 0$, $\eta_t$ be a permutation of $\{1,\hdots,N\}$ such that $X^{\eta_t(1),N,h}_t\le X^{\eta_t(2),N,h}_t\le\hdots\le X^{\eta_t(N),N,h}_t$ and $\left(Y^{i,N,h}_t=X^{\eta_t(i),N,h}_t\right)_{i \in \llbracket1,N\rrbracket}$ denote the increasing reordering also called order statistics of $\left(X^{i,N,h}_t\right)_{i \in \llbracket1,N\rrbracket}$. Even if the empirical measures of the reordered and original positions do not coincide in general at the level of sample-paths, one has, for each $t\ge 0$, $ \frac{1}{N}\sum\limits_{i=1}^N\delta_{X^{i,N,h}_t}=\frac{1}{N}\sum\limits_{i=1}^N\delta_{Y^{i,N,h}_t}$ and therefore $\displaystyle F^{N,h}(t,x)=\frac{1}{N}\sum_{i=1}^N\mathbf{1}_{\left\{Y^{i,N,h}_t \le x \right\}}$. By the Girsanov Theorem, reasoning like in the proof of Lemma \ref{ctyF}, we show that for $t>0$, the vector $\left(X^{1,N,h}_t, X^{2,N,h}_t, \dots, X^{N,N,h}_t \right)$ admits a density with respect to the Lebesgue measure on $\R^N$ and therefore
\begin{align}\label{exchanG}
    \forall t>0, a.s., \text{ the original (resp. reordered) particles have distinct positions.}
\end{align}
The collisions are possible along time but have zero probability to occur at a fixed time $t>0$ (see e.g. \cite{Sarantsev}).
We are going to check that the function $F(t,x)$ solves a mild formulation of the PDE \eqref{vpde} and $F^{N,h}(t,x)$ solves a perturbed version of this mild formulation. To do so, it is convenient to obtain the dynamics of the reordered positions $Y^{i,N,h}_t$. 
Let $\overline{\tau}^h_s = \lceil s/h \rceil h $ denote the discretization time right after $s$. We recall that $\tau^h_s=\lfloor s/h\rfloor h$ denotes the discretization time right before $s$. We set $t_k = kh$ for $k \in \N$. For $s \in [t_k,t_{k+1})$, $\tau^h_s = t_k$ and, for $s \in (t_k, t_{k+1}]$, $\overline{\tau}^h_s = t_{k+1}$.  For $t>0$, let $\eta_t^{-1}$ denote the inverse of the permutation $\eta_t$. By \eqref{exchanG}, $a.s.$, for each $k\in\N^*$, the positions $\left(X^{i,N,h}_{t_k}\right)_{1\le i\le N}$ are distinct and for $t$ in the time-interval $[t_k,t_{k+1})$, $X^{i,N,h}_t$ evolves with the drift coefficient $\lambda^N\left(\eta_{t_k}^{-1}(i)\right)=\lambda^N\left(\eta_{\tau^h_t}^{-1}(i)\right)$. To obtain the same expression of the drift coefficient on the first time interval $[0,t_1)$ we will use from now on the convention 
\begin{equation}
   \eta_0^{-1}(i)=\sum_{j=1}^{N}{\mathbf 1}_{\left\{X^{j}_0\le X^{i}_0 \right\}} \quad \mbox{ for }1\le i\le N.\label{convinvinit}
\end{equation} 
With this convention, which is consistant with the usual definition of the inverse of a permutation only if the initial positions are distinct, we have 
$$dX^{i,N,h}_t=\sigma dW^i_t+\lambda^N\left(\eta_{\tau^h_t}^{-1}(i)\right)\,dt,\;1\le i\le N.$$
By the Girsanov theorem, we may define a new probability measure equivalent to the original one on each finite time horizon under which the processes $\left(X^{i,N,h}_t-X^i_0=W^i_t+\int_0^t\lambda^N\left(\eta_{\tau_s}^{-1}(i)\right)\,ds\right)_{t\ge 0,1\le i\le N}$ are independent Brownian motions. Applying Lemma 3.7 \cite{Szn91}, which states that under this probability measure, the reordered positions evolve as a $N$-dimensional Brownian motion normally reflected at the boundary of the simplex (also called Weyl chamber), we deduce that 
\begin{equation*}
    dY^{i,N,h}_t = \sum \limits_{j=1}^{N} \mathbf{1}_{\left \{ Y^{i,N,h}_t = X^{j,N,h}_t \right \}}\left(\sigma dW^{j}_t+\lambda^N\left(\eta_{\tau^h_t}^{-1}(j)\right)\,dt\right)+  \left(\gamma^{i}_t - \gamma^{i+1}_t \right)\,d|K|_t,\;1\le i\le N
  \end{equation*}
where the local-time process $K$ with coordinates $K^i_t=\int_0^t \left( \gamma^{i}_s - \gamma^{i+1}_s\right)\,d|K|_s$ is an $\R^N$-valued continuous process with finite variation $|K|$ such that:
\begin{align}\label{procVarFin}
  d|K|_t \; \text{a.e.} \;, \gamma^{1}_t = \gamma^{N+1}_t = 0\mbox{ and for }2\le i\le N,\; \gamma^{i}_t \geq 0 \; \text{and} \; \gamma^{i}_t\left(Y^{i,N,h}_t - Y^{i-1,N,h}_t \right) = 0.
\end{align}
Defining a $N$-dimensional Brownian motion $\left(\beta^1,\hdots,\beta^N\right)$ by $\beta^{i}_t = \sum \limits_{j=1}^{N} \displaystyle \int_0^t \mathbf{1}_{\left \{ Y^{i,N,h}_s = X^{j,N,h}_s \right \}}\,dW^{j}_s$ and using the definition of $\eta_t$ and \eqref{exchanG}, we have
\begin{equation}\label{dynamicY}
  dY^{i,N,h}_t = \sigma d\beta^{i}_t+ \lambda^N\left( \eta^{-1}_{\tau^h_t}\left(\eta_t(i)\right) \right)\,dt +  \left(\gamma^{i}_t - \gamma^{i+1}_t \right)\,d|K|_t,\;1\le i\le N.
\end{equation}
Denoting by  $G_t(x) = \exp(-\frac{x^2}{2\sigma^2t}) \Big/ \sqrt{2\pi\sigma^2t} $ the probability density function of the normal law $\mathcal{N}(0,\sigma^2t)$, we are now ready to state the mild formulation of the PDE \eqref{vpde} satisfied by $F(t,x)$ and the perturbed version satisfied by $F^{N,h}(t,x)$.

\begin{prop}\label{propspde}
For each $t\ge 0$ and each $h \in [0,T]$ , we have $dx$ a.e.:
\begin{align}
 \displaystyle F(t,x) =& \quad G_t*F_0(x) - \int_0^t \partial_xG_{t-s}*\Lambda(F(s,.))(x)\,ds,\label{milf}\\
 \text{a.s.} \quad F^{N,h}&(t,x) = \quad G_t*F^{N,h}_0(x) - \int_0^t \partial_xG_{t-s}*\Lambda(F^{N,h}(s,.))(x)\,ds - \frac{\sigma}{N}\sum \limits_{i=1}^N \int_0^t G_{t-s}(X^{i,N,h}_s -x)\,dW^{i}_s \notag \\
 &+ \frac{1}{N}\sum \limits_{i=1}^N \int_0^t G_{t-s}(Y^{i,N,h}_s -x) \left[ \lambda^N\left(i\right) - \lambda^N\left(\eta_{\tau^h_s}^{-1}(\eta_s(i))\right) \right]\,ds.\label{milfnh}
\end{align}
\end{prop}
\begin{remark}
When $h=0$, one should notice that the fourth term in $F^{N,h}(t,x)$ is null so that: 
$$\displaystyle F^{N,0}(t,x) = G_t*F^{N,0}_0(x) - \int_0^t \partial_xG_{t-s}*\Lambda(F^{N,0}(s,.))(x)\,ds - \frac{\sigma}{N}\sum \limits_{i=1}^N \int_0^t G_{t-s}(X^{i,N,0}_s -x)\,dW^{i}_s.$$ 
\end{remark}
\begin{remark}\label{mildEDSbis}
  Let similarly for each $t\ge 0$, $\breve \eta_t$ be a permutation of $\{1,\hdots,N\}$ such that $\breve X^{\breve \eta_t(1),N,h}_t\le \breve X^{\breve \eta_t(2),N,h}_t\le\hdots\le \breve X^{\breve \eta_t(N),N,h}_t$ and $\left(\breve Y^{i,N,h}_t=\breve X^{\breve \eta_t(i),N,h}_t\right)_{i \in \llbracket1,N\rrbracket}$. Let also $\breve \eta_t^{-1}$ denote the inverse of the permutation $\breve \eta_t$ for $t>0$ and $\breve\eta_0^{-1}=\eta_0^{-1}$. Reasoning like in the proof of Proposition \ref{propspde}, we may derive the perturbed mild equation satisfied by the associated empirical cumulative distribution function $\breve F^{N,h}(t,x)$:
\begin{align*}
\breve F^{N,h}(t,x) = &\quad G_t*F^{N,h}_0(x) - \int_0^t \partial_xG_{t-s}*\Lambda(\breve F^{N,h}(s,.))(x)\,ds - \frac{\sigma}{N}\sum \limits_{i=1}^N \int_0^t G_{t-s}(\breve X^{i,N,h}_s -x)\,dW^{i}_s \notag \\
 &+ \frac{1}{N}\sum \limits_{i=1}^N \int_0^t G_{t-s}(\breve Y^{i,N,h}_s -x) \left[ \lambda^N\left(i\right) - \lambda\left(\breve \eta_{\tau^h_s}^{-1}(\breve \eta_s(i))/N\right) \right]\,ds.   
\end{align*}
Using the estimation $$\int_\R\left|\frac{1}{N}\sum \limits_{i=1}^N\int_0^t G_{t-s}\left(\breve Y^{i,N,h}_s -x\right) \left[ \lambda^N\left(\breve \eta_{\tau^h_s}^{-1}(\breve \eta_s(i))\right)-\lambda\left(\breve \eta_{\tau^h_s}^{-1}(\breve \eta_s(i))/N\right) \right]\,ds\right|\,dx\le t\max_{1\le j\le N}|\lambda^N(j)-\lambda(j/N)|$$ of the additional error term in comparison with \eqref{milfnh}, we may adapt all proofs to check the statements at the end of Remark \ref{remcheck}.
\end{remark}
\begin{proof}
  Let $t>0$, $f$ be a $C^1$ and compactly supported function on $\R$ and $\displaystyle \varphi(s,x)=\int_\R {\mathbf 1}_{\{x\le y\}}G_{t-s}*f(y)\,dy$ be the convolution of $G_{t-s}$ with $\displaystyle x\mapsto\int_x^{+\infty}f(y)\,dy$ for $(s,x)\in[0,t)\times\R$ and $\displaystyle \varphi(t,x)=\int_\R {\mathbf 1}_{\{x\le y\}}f(y)\,dy$. The function $\varphi(s,x)$ is continuously differentiable w.r.t. to $s$ and twice continuously differentiable w.r.t. to $x$ on $[0,t]\times\R$ and solves
  \begin{equation}\label{edpvarphi}
   \partial_s \varphi(s,x)+\frac{\sigma^2}{2}\partial_{xx}\varphi(s,x)=0\mbox{ for }(s,x)\in[0,t]\times\R.
 \end{equation}
 Computing $\varphi(t,X_t)$ where $(X_s)_{s\ge 0}$ solves \eqref{edsnl} and using \eqref{edpvarphi}, we obtain that:
 $$\varphi(t,X_t)=\varphi(0,X_0)-\sigma\int_0^tG_{t-s}*f(X_s)\,dW_s-\int_0^t\lambda(F(s,X_s))G_{t-s}*f(X_s)\,ds.$$
 Since, on $[0,t)\times\R$, $G_{t-s}*f(x)$ is bounded by the supremum of $|f|$, the expectation of the stochastic integral is zero. By Fubini's theorem and since $G_t$ is even, the expectations of $\varphi(t,X_t)$ and $\varphi(0,X_0)$ are respectively equal to $\displaystyle \int_\R \int_\R{\mathbf 1}_{\{x\le y\}}\mu_t(dx)f(y)\,dy=\int_\R F(t,y)f(y)\,dy$ and $\displaystyle \int_\R F_0(y)G_t*f(y)\,dy=\int_\R G_t*F_0(y)f(y)\,dy$. Using Fubini's theorem, the equality $\displaystyle G_{t-s}*f(x)=-\int_\R{\mathbf 1}_{\{x\le y\}}\partial_yG_{t-s}*f(y)\,dy$, the fact that, by the chain rule for continuous functions with finite variation, $\displaystyle \int_\R{\mathbf 1}_{\{x\le y\}}\lambda(F(s,x))p(s,x)\,dx=\Lambda(F(s,y))-\Lambda(0)$, the equality $\displaystyle \int_\R\partial_yG_{t-s}*f(y)\,dy=0$ and the oddness of $\partial_y G_{t-s}$, we obtain that the expectation of the last term on the right-hand side is equal to
\begin{align*}
  \displaystyle \int_0^t\int_\R\int_\R{\mathbf 1}_{\{x\le y\}}\lambda(F(s,x))p(s,x)\,dx\,\partial_yG_{t-s}*f(y)\,dy\,ds
   &=\int_0^t\int_\R(\Lambda(F(s,y))-\Lambda(0))\partial_yG_{t-s}*f(y)\,dy\,ds\\
   &=-\int_0^t\int_\R\partial_yG_{t-s}*\Lambda(F(s,.))(y)f(y)\,dy\,ds.
\end{align*}
Exchanging the time and space integrals by Fubini's theorem, we  deduce that
$$\int_\R F(t,x)f(x)dx=\int_\R G_t*F_0(x)f(x)\,dx-\int_\R f(x)\int_0^t\partial_x G_{t-s}*\Lambda(F(s,.))(x)\,ds\,dx.$$
 Since $f$ is arbitrary, we conclude that $F$ satisfies the mild formulation \eqref{milf}.

 Let us now establish that $F^{N,h}$ satisfies a perturbed version of this equation given by \eqref{milfnh}.
  By computing $\varphi(t,Y^{i,N,h}_t)$ by It\^o's formula, using \eqref{edpvarphi} and summing over $i\in\{1,\hdots,N\}$, we obtain
  \begin{align}\label{itoy}
   \int_{\R}\sum_{i=1}^N{\mathbf 1}_{\{Y^{i,N,h}_t\le y\}}f(y)\,dy=&\int_{\R}\sum_{i=1}^N{\mathbf 1}_{\{Y^{i,N,h}_0\le y\}}G_t*f(y)\,dy-\sum_{i=1}^N\int_0^tG_{t-s}*f(Y^{i,N,h}_s)\left(\sigma\,d\beta^i_s+\lambda^N\left( \eta^{-1}_{\tau^h_s}(\eta_s(i) \right)\,ds\right)\\
   &+\sum_{i=1}^N\int_0^t\partial_x\varphi(s,Y^{i,N,h}_s)(\gamma^i_s-\gamma^{i+1}_s)\,d|K|_s.\notag
  \end{align}

  By summation by parts and \eqref{procVarFin},
$$\sum_{i=1}^N\int_0^t\partial_x\varphi(s,Y^{i,N,h}_s)(\gamma^i_s-\gamma^{i+1}_s)\,d|K|_s =\sum_{i=2}^N\int_0^t(\partial_x\varphi(s,Y^{i,N,h}_s)-\partial_x\varphi(s,Y^{i-1,N,h}_s))\gamma^i_s\,d|K|_s=0.$$
 Since the empirical cumulative distribution functions of the original and the reordered systems at time $t$ coincide and the function $G_t$ is even, the left-hand side and the first term in the right-hand side are respectively equal to $\displaystyle N\int_\R  F^{N,h}(t,y)f(y)\,dy$ and $ \displaystyle N\int_\R  G_t*F^{N,h}(0,y)f(y)\,dy$. The definition of the Brownian motion $\beta$ and \eqref {exchanG} imply that
  $$\sum_{i=1}^N\int_0^tG_{t-s}*f(Y^{i,N,h}_s)\,d\beta^i_s=\sum_{j=1}^N\int_0^tG_{t-s}*f(X^{j,N,h}_s)\,dW^j_s.$$
 Dividing \eqref{itoy} by $N$, we deduce that
  \begin{align*}
  \displaystyle \int_\R  F^{N,h}(t,y)f(y)\,dy=&\int_\R  G_t*F^{N,h}(0,y)f(y)\,dy-\frac{\sigma}{N}\sum_{i=1}^N\int_0^tG_{t-s}*f\left(X^{i,N,h}_s\right)\,dW^i_s\\
  &-\frac{1}{N}\sum_{i=1}^N\int_0^tG_{t-s}*f(Y^{i,N,h}_s)\lambda^N\left( \eta^{-1}_{\tau^h_s}(\eta_s(i) \right)\,ds.
  \end{align*}
 We are going to add and substract
\begin{align*}
  \displaystyle \frac 1 N\sum_{i=1}^NG_{t-s}*f(Y^{i,N,h}_s)\lambda^N(i)&=-\int_{\R}\partial_yG_{t-s}*f(y)\sum_{i=1}^N{\mathbf 1}_{\{Y^{i,N,h}_s\le y\}}\left(\Lambda(i/N)-\Lambda((i-1)/N)\right)\,dy\\
  &=-\int_\R\partial_yG_{t-s}*f(y)\left(\Lambda\left(F^{N,h}(s,y)\right)-\Lambda(0)\right)\,dy=\int_\R f(y)\partial_yG_{t-s}*\Lambda\left(F^{N,h}(s,.)\right)(y)\,dy.
\end{align*}
 On the other hand, since $f$ is square integrable and with the use of Young's inequality for the product and the estimate \eqref{GSquareEsp} from Lemma \ref{EstimHeatEq}, we have that: 
 \begin{align*}
  \displaystyle \int_{\R} \left\{\int_0^t \left| G_{t-s}(X^{i,N,h}_s -x)f(x)\right|^2\,ds \right\}^{1/2}dx = &\int_{\R} |f(x)|\left\{\int_0^t| G_{t-s}(X^{i,N,h}_s -x)|^2\,ds \right\}^{1/2}\,dx \\
    &\leq \frac{1}{2}\int_{\R} f^2(x)\,dx + \frac{1}{2}\int_{\R} \int_0^t G^2_{t-s}(X^{i,N,h}_s -x)\,ds\,dx \\
    &= \frac{1}{2}\int_{\R} f^2(x)\,dx + \frac{1}{2\sigma} \sqrt{\frac{t}{\pi}} < \infty.
 \end{align*}
 For that reason, we can use a stochastic Fubini theorem stated by Veraar \cite{VERAA} and recalled in Lemma \ref{StochFub} to deduce that $\displaystyle \frac{\sigma}{N}\sum_{i=1}^N\int_0^tG_{t-s}*f\left(X^{i,N,h}_s\right)\,dW^i_s=\frac{\sigma}{N}\sum \limits_{i=1}^N  \int_{\R}f(x)\left\{ \int_0^t  G_{t-s}\left(X^{i,N,h}_s -x\right)\,dW^{i}_s\right\}\,dx$.
Therefore
\begin{align*}
   \int_\R  F^{N,h}(t,x)f(x)\,dx=&\int_\R  G_t*F^{N,h}(0,x)f(x)\,dx-\frac{\sigma}{N}\sum \limits_{i=1}^N  \int_{\R}f(x)\left\{ \int_0^t  G_{t-s}(X^{i,N,h}_s -x)\,dW^{i}_s\right\}\,dx\\
   &-\int_0^t\int_\R f(x)\partial_xG_{t-s}*\Lambda\left(F^{N,h}(s,.)\right)(x)\,dx\,ds\\
   &+\frac{1}{N}\sum_{i=1}^N\int_0^t\int_\R G_{t-s}\left(Y^{i,N,h}_s-x\right)f(x)\,dx\left\{\lambda^N(i)-\lambda^N\left( \eta^{-1}_{\tau^h_s}(\eta_s(i) \right)\right\}\,ds.
\end{align*}
Since $f$ is bounded and $\Lambda$ is bounded on the interval $[0,1]$, using \eqref{FirstDerivG}, we check that we can apply Fubini's theorem to interchange the space and time integrals in the two last terms of the right-hand side. Since $f$ is arbitrary, we conclude that \eqref{milfnh} holds a.s. $dx$ a.e..
\end{proof}


\section{Proofs of the results in Section 2}
\subsection{Quantitative propagation of chaos result}
The proof of Theorem \ref{thmpropchaos} relies on the following Lemma which estimates for $t>0$ the $L^\infty$-norm of the density $p(t,x)$ of $X_t$ solution to \eqref{edsnl} which is guaranteed to exist according to Lemma \ref{ctyF}.
\begin{lem}\label{estidens}
 $$\forall \; T\in(0,+\infty),\exists \; C_{\infty,T}<\infty,\forall \; t\in(0,T],\quad \|p(t,.)\|_{L^\infty}\le C_{\infty,T} t^{-1/2}.$$
\end{lem}
\begin{proof}
  Reasoning like at the beginning of the proof of Proposition \ref{propspde} but with the function $\varphi(s,x)={\mathbf 1}_{[0,t[}(s)G_{t-s}*f(x)+{\mathbf 1}_{\{s=t\}}f(x)$ (in place of its spatial antiderivative) for $f$ $C^2$ and compactly supported on $\R$, we easily check that $p(t,x)$ satisfies the mild formulation :
   \begin{equation}\label{mildp} 
  \displaystyle \forall t>0,\;dx\;a.e.,\quad p(t,x)=G_t*m(x)-\int_0^t\partial_x G_{t-s}*\left(\lambda(F(s,.))p(s,.)\right)(x)\,ds.
 \end{equation}
 Since for $t>0$, $\|G_t*m\|_{L^\infty}\le \|G_t\|_{L^\infty}=(2\pi \sigma^2t)^{-1/2}$, it is enough to check that the estimation holds for the time integral in the mild formulation. By Jensen's inequality, Minkowski's inequality, Young's inequality then  then \eqref{GPrimSqu} and \eqref{GSquare}, \eqref{mildp} implies that, for $t>0$,
 \begin{align*}
  \displaystyle \|p(t,.)\|_{L^2}&\le \|G_t\|_{L^2}+\int_0^t\|\partial_x G_{t-s}\|_{L^2}\|\lambda(F(s,.))p(s,.)\|_{L^1}\,ds\\
   &\le \frac{1}{\sqrt{2\sigma} (\pi t)^{1/4}}+L_\Lambda \int_0^t\frac{ds}{2\sigma^{3/2}\pi^{1/4}(t-s)^{3/4}}=\frac{1}{\sqrt{2\sigma} (\pi t)^{1/4}}+\frac{2L_\Lambda t^{1/4}}{\sigma^{3/2}\pi^{1/4}}
 \end{align*}
where the right-hand side is not greater than  $\left(\frac{1}{\sqrt{2\sigma} (\pi)^{1/4}}+\frac{2L_\Lambda T^{1/2}}{\sigma^{3/2}\pi^{1/4}}\right) t^{-1/4}$ for $t\in(0,T]$.
 With the boundedness of $\lambda$ and Young's inequality for convolutions, we deduce that for $t\in(0,T]$,
 \begin{align*}
   \|p(t,.)\|_{L^\infty}&\le \|G_t\|_{L^\infty}+L_\Lambda\int_0^t\|\partial_x G_{t-s}\|_{L^2}\|\lambda(F(s,.))p(s,.)\|_{L^2}\,ds\\
   &\le (2\pi \sigma^2t)^{-1/2}+\frac{L_\Lambda}{2\sigma^{3/2}\pi^{1/4}}\left(\frac{1}{\sqrt{2\sigma} (\pi)^{1/4}}+\frac{2L_\Lambda T^{1/2}}{\sigma^{3/2}\pi^{1/4}}\right)\int_0^t\frac{ds}{(t-s)^{3/4}s^{1/4}}.
 \end{align*}
 Since $\displaystyle \int_0^t\frac{ds}{(t-s)^{3/4}s^{1/4}}ds=\int_0^1\frac{du}{(1-u)^{3/4}u^{1/4}}\le \int_0^1\frac{du}{(1-u)^{3/4}u^{1/4}}T^{1/2}t^{-1/2}$, we easily conclude.
\end{proof}
We are now ready to prove Theorem \ref{thmpropchaos} by adapting the proof of Theorem 1.6 \cite{ShkoKol}. Since, by Jensen's inequality, the conclusion with $\rho=1$ implies the conclusion with $\rho\in (0,1)$, we suppose without loss of generality that $\rho\ge 1$. 
Lemma \ref{estidens} implies the following estimation of the Lipschitz constant of $x\mapsto \lambda(F(t,x))$:
\begin{equation}
   \label{liplamft}\forall t\in(0,T],\;L_{\lambda(F(t,.))}\le C_{\infty,T}L_\lambda t^{-1/2}. 
\end{equation}
We deduce that for a finite constant $C$ which may change from line to line and depends on $T$ but not on $N$:
\begin{align*}
 \displaystyle &\sup_{s\in[0,t]}\left|X^i_s-\breve X^{i,N}_s\right|^\rho\le \left(\int_0^t\left|\lambda\left(F\left(u,X^i_u\right)\right)-\lambda\left(\breve F^N\left(u,\breve X^{i,N}_u\right)\right)\right|\,du\right)^\rho \\
 &= \left(\int_0^t u^{-\frac{\rho-1}{2\rho}} \times u^{\frac{\rho-1}{2\rho}}\left|\lambda\left(F\left(u,X^i_u\right)\right)-\lambda\left(\breve F^N\left(u,\breve X^{i,N}_u\right)\right)\right|\,du\right)^\rho\\
  &\le \left(\int_0^t u^{-1/2}\,du\right)^{\rho-1} \int_0^tu^{(\rho-1)/2}\left(\left|\lambda\left(F\left(u,X^i_u\right)\right)-\lambda\left(F\left(u,\breve X^{i,N}_u\right)\right)\right|+\left|\lambda\left(F\left(u,\breve X^{i,N}_u\right)\right)-\lambda\left(\breve F^N\left(u,\breve X^{i,N}_u\right)\right)\right|\right)^\rho\,du\\
  &\le C\int_0^t \left(u^{-1/2}\left|X^i_u-\breve X^{i,N}_u\right|^\rho +u^{(\rho-1)/2}\left|\lambda\left(F\left(u,\breve X^{i,N}_u\right)\right)-\lambda\left(\breve F^N\left(u,\breve X^{i,N}_u\right)\right)\right|^\rho \right)\,du,
\end{align*}
where we used H\"older's inequality for the second inequality.
Using exchangeability of $(\breve X^{1,N},\hdots,\breve X^{N,N})$, denoting by $\breve Y^{1,N}_u\le \breve Y^{2,N}_u\le \hdots\le \breve Y^{N,N}_u$ (resp. $Y^1_u\le Y^2_u\le \hdots\le Y^N_u$) the increasing reordering of $\left(\breve X^{1,N}_u,\hdots,\breve X^{N,N}_u\right)$ (resp. $\left(\breve X^{1}_u,\hdots,\breve X^{N}_u\right)$) and using that \eqref{exchanG} and its proof generalizes to the particle system \eqref{PSDEbis} then \eqref{liplamft}, we obtain that
\begin{align*}
   \E&\left[\left|\lambda\left(F\left(u,\breve X^{i,N}_u\right)\right)-\lambda\left(\breve F^N\left(u,\breve X^{i,N}_u\right)\right)\right|^\rho\right]=\E\left[\frac{1}{N}\sum_{j=1}^N\left|\lambda\left(F\left(u,\breve X^{j,N}_u\right)\right)-\lambda\left(\breve F^N\left(u,\breve X^{j,N}_u\right)\right)\right|^\rho\right]\\
   &=\E\left[\frac{1}{N}\sum_{j=1}^N\left|\lambda\left(F\left(u,\breve Y^{j,N}_u\right)\right)-\lambda\left(\breve F^N\left(u,\breve Y^{j,N}_u\right)\right)\right|^\rho\right]\\
   &=\E\left[\frac{1}{N}\sum_{j=1}^N\left|\lambda(F(u,\breve Y^{j,N}_u))-\lambda(F(u,Y^{j}_u))+\lambda(F(u,Y^{j}_u))-\lambda\left(\frac{j}{N}\right)\right|^\rho\right]\\
   &\le C \left(u^{-\rho/2}\E\left[\frac{1}{N}\sum_{j=1}^N\left|\breve Y^{j,N}_u-Y^j_u\right|^\rho\right]+\E\left[\frac{1}{N}\sum_{j=1}^N\left|F(u,Y^{j}_u)-\frac{j}{N}\right|^\rho\right]\right).
\end{align*}
Since $F(u,Y^1_u)\le F(u,Y^2_u)\le \hdots\le F(u,Y^N_u)$ is the increasing reordering of the random variables $(F(u,X^i_u))_{1\le i\le N}$ which are i.i.d. according to the uniform law on $[0,1]$, according to the proof of Theorem 1.6 \cite{ShkoKol}, the second expectation in the right-hand side is bounded from above by $CN^{-\rho/2}$. On the other hand, by \eqref{Wasserstein}, $\frac{1}{N}\sum \limits_{j=1}^N\left|\breve Y^{j,N}_u-Y^j_u\right|^\rho={\cal W}_\rho^\rho\left(\frac 1N\sum \limits_{i=1}^N\delta_{\breve Y^i_u},\frac 1N\sum \limits_{i=1}^N\delta_{X^i_u}\right)\le \frac{1}{N}\sum \limits_{j=1}^N\left|\breve X^{j,N}_u-X^j_u\right|^\rho$. We deduce that for all $t\in[0,T]$,
$$\displaystyle \E\left[\sup_{s\in[0,t]}\left|X^i_s-\breve X^{i,N}_s\right|^\rho\right]\le C\left(N^{-\rho/2}+\int_0^t u^{-1/2}\E\left[\left|X^i_u-\breve X^{i,N}_u\right|^\rho\right]\,du\right).$$
Performing the change of variable $v=\sqrt{u}$ in the integral and setting $f(t)=\E\left[\sup_{s\in[0,t^2]}|X^i_s-\breve X^{i,N}_s|^\rho\right]$, we deduce that $\forall t\in\left[0,\sqrt{T}\right]$, $f(t)\le C\left(N^{-\rho/2}+\int_0^t f(v)dv\right)$. Since, by boundedness of $\lambda$, the function $f$ is locally bounded, we conclude using Gr\"onwall's lemma that $\E\left[\sup_{s\in[0,T]}|X^i_s-\breve X^{i,N}_s|^\rho\right]\le CN^{-\rho/2}$. 

Remarking that $\displaystyle \left|X^{i,N}_t-X^i_0-\sigma W^i_t-\int_0^t\lambda\left(F^{N}\left(s,X^{i,N}_s\right)\right)\,ds\right|\le\frac{L_\lambda t}{2N}$, we may adapt the arguments to deal with the particle system \eqref{PSDE}.

\subsection{Rate of convergence of the strong $L^1$-error}\label{secproofthm1}

To prove Theorem \ref{thm1}, we need the following lemmas.
\begin{lem}\label{XandY}
 $$\forall 0 \le s \le t \le T,\;\forall \rho \ge 1, h \in [0,T],\; N \in \N^{*} \quad \sum \limits_{j=1}^N \left| Y^{j,N,h}_t - Y^{j,N,h}_s \right|^{\rho} \leq \sum \limits_{j=1}^N \left| X^{j,N,h}_t - X^{j,N,h}_s \right|^{\rho}.$$
\end{lem}
\begin{proof}
 By \eqref{Wasserstein}, we have $\displaystyle \mathcal{W}_{\rho}^{\rho}\left( \mu^{N,h}_s,\mu^{N,h}_t\right)=\frac{1}{N}\sum \limits_{j=1}^N \left|Y^{j,N,h}_t - Y^{j,N,h}_s \right|^{\rho}$.
Since $\displaystyle \frac{1}{N} \sum \limits_{j=1}^N \delta_{\left(X^{j,N,h}_s,X^{j,N,h}_t\right)}$ is a coupling measure on $\R^2$ with first marginal $\mu^{N,h}_s$ and second marginal $\mu^{N,h}_t$, we conclude that $$\displaystyle  \frac{1}{N}\sum \limits_{j=1}^N \left|X^{j,N,h}_t - X^{j,N,h}_s \right|^{\rho} \ge \mathcal{W}_{\rho}^{\rho}\left( \mu^{N,h}_s,\mu^{N,h}_t\right) = \frac{1}{N}\sum \limits_{j=1}^N \left|Y^{j,N,h}_t - Y^{j,N,h}_s \right|^{\rho}.$$
 \end{proof}

This second lemma ensures the local integrability of $t \mapsto \E\left[ \mathcal{W}_1\left(\mu^{N,h}_t,\mu_t \right)\right]$.
\begin{lem}\label{integrF}
$$\forall t,h \in [0,T],\forall N \in \N^{*}, \quad \displaystyle \E\left[\mathcal{W}_1\left(\mu^{N,h}_t,\mu_t \right)\right] \leq 2 \, \sigma \, \sqrt{\frac{2 t}{\pi}} + 2 \, L_{\Lambda}t + \E\left[\mathcal{W}_1(\mu^{N,h}_0, m)\right]. $$
 If $\displaystyle \int_{\R} |x|m(dx) < \infty$, then for each $N\in\N^*$ and each $h\in[0,T]$, $t\mapsto \E[{\cal W}_1(\mu^{N,h}_t,\mu_t)]$ is locally integrable on $\R_+$. \end{lem}
    
\begin{proof}
	Using the triangle inequality, we have:
\begin{align}\label{inegTriang}
	\E\left[\mathcal{W}_1(\mu^{N,h}_t, \mu_t)\right] \leq \E\left[\mathcal{W}_1(\mu^{N,h}_t, \mu^{N,h}_0)\right] + \E\left[\mathcal{W}_1(\mu^{N,h}_0, m)\right] + \mathcal{W}_1(m, \mu_t).
\end{align}
Since 
\begin{equation}\label{estimW1}
  \left\{
      \begin{aligned}
        & \mathcal{W}_1(m,\mu_t) \leq \E\left[\left| X_t - X_0 \right|\right] \leq \E \left[ \left| \sigma W_t\right|\right] + \E\left[\left| \displaystyle \int_0^t \lambda(F(s,X_s))\,ds \right|\right] = \sigma \, \sqrt{\frac{2 t}{\pi}} + L_{\Lambda}t, \\
        &\E\left[\mathcal{W}_1(\mu^{N,h}_t, \mu^{N,h}_0)\right] \le  \frac{1}{N} \sum \limits_{i=1}^N \E\left[\left|\sigma W^{i}_t + \displaystyle \int_0^t \lambda^N\left( \mathbf{1}_{\left\{ X^{j,N,h}_s \leq X^{i,N,h}_s \right\}} \right)\,ds \right| \right] \leq \sigma \,\sqrt{\frac{2 t}{\pi}} + L_{\Lambda}t.\\
      \end{aligned}
    \right.
\end{equation}
then by injecting \eqref{estimW1} in \eqref{inegTriang}, we obtain the upper-bound of $\E\left[ \mathcal{W}_1\left(\mu^{N,h}_t,\mu_t \right)\right]$. Since for $i\in\{1,\hdots,N\}$, $x^N_{i} = F^{-1}_0 \left(\frac{2i-1}{2N} \right)$ minimizes $\displaystyle \R\ni y\mapsto \int_{\frac{i-1}{N}}^{\frac{i}{N}}\left|y - F^{-1}_0(u) \right|\,du$, we have that:
\begin{align}
  \displaystyle \mathcal{W}_1\left(\tilde \mu_0^N,m\right) =\sum_{i=1}^N \int_{\frac{i-1}{N}}^{\frac{i}{N}}\left|x^N_{i} - F^{-1}_0(u) \right|\,du &\le \sum_{i=1}^N \int_{\frac{i-1}{N}}^{\frac{i}{N}}\left| F^{-1}_0(u) \right|\,du = \int_{\R}|x|m(dx).\label{majow1det}
\end{align}
For the random initialization, by Theorem 3.5 \cite{bobkovledoux},
\begin{align*}
  \E\left[ \mathcal{W}_1\left(\hat\mu^N_0,m\right)\right]\le 2\int_\R\mathbf{1}_{\{F_0(x)(1-F_0(x))\le \frac{1}{4N}\}}F_0(x)(1-F_0(x))dx+\frac{1}{\sqrt{N}}\int_\R \mathbf{1}_{\{F_0(x)(1-F_0(x))> \frac{1}{4N}\}}\sqrt{F_0(x)(1-F_0(x))}dx.\end{align*}
Since $F_0(x)(1-F_0(x))> \frac{1}{4N}$ implies $\frac{1}{\sqrt{N}}<2\sqrt{F_0(x)(1-F_0(x))}$, we deduce that
$$\E\left[ \mathcal{W}_1\left(\hat\mu^N_0,m\right)\right] \le 2\int_\R F_0(x)(1-F_0(x))dx\le 2\left(\int_{-\infty}^0F_0(x)dx+\int^{+\infty}_0(1-F_0(x))dx\right)=2\int_\R |x|m(dx).$$
Hence for both the random and the optimal deterministic initializations, the finiteness of the first order moment of $m$  implies the finiteness of $\E\left[\mathcal{W}_1(\mu^{N,h}_0, m)\right]$ and, with the first statement, the local integrability of $t \mapsto \E\left[ \mathcal{W}_1\left(\mu^{N,h}_t,\mu_t \right)\right]$.
\end{proof}

The third lemma gives a control of the moments of order $\rho\ge 1$ of $X^{i,N,h}_t$, $\forall i \in \llbracket1,N \rrbracket$.
\begin{lem}\label{controlM}
  If $\displaystyle \int_{\R} |x|^{\rho}m(dx) < \infty$ for some $\rho\ge 1$, then $\forall N \in \N^{*}, \forall h \in [0,T]$, $$\sup_{t \leq T}\E\left[\frac{1}{N}\sum_{i=1}^N\left|X^{i,N,h}_t\right|^{\rho} \right] \leq M := 3^{\rho-1}\left( 2\int_{\R}|x|^{\rho}m(dx)+  \frac{1}{\sqrt \pi}\Gamma\left(\frac{\rho +1}{2} \right)(2\sigma^2T)^{\rho/2}+ (L_{\Lambda}T)^{\rho} \right). $$
\end{lem}
\begin{proof}
We have   $$\E\left[\left|X^{i,N,h}_t\right|^{\rho} \right] \le 3^{\rho-1}\left( \E\left[|X^{i}_0|^{\rho}\right] + \sigma^\rho \E\left[|W^{i}_t|^{\rho}\right] + \E\left[\left|\int_0^t \lambda^N\left(\sum \limits_{j=1}^{N} \mathbf{1}_{ \{  X^{j,N,h}_{\tau^h_s} \le X^{i,N,h}_{\tau^h_s} \}}\right)\,ds\right|^{\rho}\right]  \right).$$
  Since $\E\left[|W^{i}_t|^{\rho}\right]^{1/\rho}=\frac{1}{\sqrt \pi}\Gamma\left(\frac{\rho +1}{2} \right)(2t)^{\rho/2}$, one easily concludes when the initial conditions are i.i.d. according to $m$. When they are optimal deterministic, we sum over $i\in\{1,\hdots,N\}$, divide by $N$ and use the estimation
$\frac{1}{N}\sum_{i=1}^N\left|F_0^{-1}\left(\frac{2i-1}{2N}\right)\right|^\rho\le 2\int_\R|x|^\rho m(dx)$ that we now prove.

Since
   \begin{align*}
     \displaystyle &\forall i\in\{1,\hdots,N-1\},\quad F_0^{-1}\left(\frac{2i-1}{2N}\right)\ge 0\Rightarrow \frac{\left|F_0^{-1}\left(\frac{2i-1}{2N}\right)\right|^\rho}{N}\le \int_{\frac{2i-1}{2N}}^{\frac{2i+1}{2N}}\left|F_0^{-1}(u)\right|^\rho\,du,\\
     &\forall i\in\{2,\hdots,N\},\quad F_0^{-1}\left(\frac{2i-1}{2N}\right)\le 0\Rightarrow \frac{\left|F_0^{-1}\left(\frac{2i-1}{2N}\right)\right|^\rho}{N} \le \int_{\frac{2i-3}{2N}}^{\frac{2i-1}{2N}}\left|F_0^{-1}(u)\right|^\rho\,du,
   \end{align*}
   one has 
   \begin{align*}
    \displaystyle \frac{1}{N}\sum_{i=1}^N\left|F_0^{-1}\left(\frac{2i-1}{2N}\right)\right|^\rho &\le \mathbf{1}_{\left\{F_0^{-1}\left(\frac{1}{2N}\right)\le 0 \right\}} \frac{\left|F_0^{-1}\left(\frac{1}{2N}\right)\right|^\rho}{N}+\int_{\frac{1}{2N}}^{\frac{2N-1}{2N}}\left|F_0^{-1}(u)\right|^\rho\,du  + \mathbf{1}_{\left\{F_0^{-1}\left(\frac{2N-1}{2N}\right)\ge 0 \right \}} \frac{\left|F_0^{-1}\left(\frac{2N-1}{2N}\right)\right|^\rho}{N}\\
    &\le \int_0^{\frac{1}{2N}}\left|F_0^{-1}(u)\right|^\rho\,du +\int_{0}^{1}\left|F_0^{-1}(u)\right|^\rho\,du+\int^1_{\frac{2N-1}{2N}}\left|F_0^{-1}(u)\right|^\rho\,du\\
    &\le 2\int_0^1|F_0^{-1}(u)|^\rho du=2\int_\R|x|^\rho m(dx).\end{align*}\end{proof}

\subsubsection{Proof of Theorem \ref{thm1}}
Defining for all $t,h \in [0,T], N \in \N^{*}$,
\begin{align}
\displaystyle &R^{N,h}(t,x) = - \frac{\sigma}{N}\sum \limits_{i=1}^N \int_0^t G_{t-s}(X^{i,N,h}_s -x)\,dW^{i}_s,\label{defrnh} \\
&E^{N,h}(t,x) = \frac{1}{N}\sum \limits_{i=1}^N  \int_0^t  G_{t-s}(Y^{i,N,h}_s-x) \left[ \lambda^{N}\left(i\right) - \lambda^{N}\left(\eta_{\tau^h_s}^{-1}(\eta_s(i))\right) \right]\,ds,
\end{align}
 we deduce from Proposition \ref{propspde} that:
\begin{align}\label{eqFh}
    \displaystyle F^{N,h}(t,x) - F(t,x) =& \quad G_t*\left(F^{N,h}_0 - F_0\right)(x) - \int_0^t \partial_x G_{t-s} * \left(\Lambda(F^{N,h}(s,.)) - \Lambda(F(s,.))\right)(x)\,ds \\
    &+ R^{N,h}(t,x) + E^{N,h}(t,x). \notag
\end{align}
Using the triangle inequality and taking expectations, we deduce that:
\begin{align*}
     \displaystyle \E\left[\left\| F^{N,h}(t,.) - F(t,.) \right\|_{L^1} \right] &\le \E\left[\left\| G_t \ast \left(F^{N,h}_0 - F_0\right)  \right\|_{L^1} \right] 
     + \E\left[\left\| R^{N,h}(t,.) \right\|_{L^1} \right] + \E\left[\left\| E^{N,h}(t,.) \right\|_{L^1} \right]  \\
     &+  \E\left[\left\|\int_0^t \partial_x G_{t-s} * \left(\Lambda(F^{N,h}(s,.)) - \Lambda(F(s,.))\right)\,ds \right\|_{L^1} \right].
\end{align*}
Using the estimate \eqref{FirstDerivG} from Lemma \ref{EstimHeatEq} and setting $\displaystyle A = \frac{L_{\Lambda}}{\sigma} \sqrt{\frac{2}{\pi}}$, we obtain:
\begin{align*}
    \displaystyle \E\left[\left\|\int_0^t \partial_x G_{t-s} * \left(\Lambda(F^{N,h}(s,.)) - \Lambda(F(s,.))\right)\,ds \right\|_{L^1} \right] &\leq \int_0^t \left\| \partial_x G_{t-s}\right\|_{L^1} L_{\Lambda}\E\left[\left\Vert F^{N,h}(s,.) - F(s,.) \right\Vert_{L^1}\right]\,ds \\
    &= A\int_0^t \frac{1}{\sqrt{t-s}}\E\left[\left\Vert F^{N,h}(s,.) - F(s,.) \right\Vert_{L^1}\right]\,ds.
\end{align*}
Therefore,
\begin{align}\label{controlConvoB}
     \displaystyle \E\left[\left\| F^{N,h}(t,.) - F(t,.) \right\|_{L^1} \right] \le \E\left[\left\| F^{N,h}_0 - F_0 \right\|_{L^1} \right] &+ \E\left[\left\| R^{N,h}(t,.) \right\|_{L^1} \right] + \E\left[\left\| E^{N,h}(t,.) \right\|_{L^1} \right] \\
    &+ A \int_0^t \frac{1}{\sqrt{t-s}}\E\left[\left\Vert F^{N,h}(s,.) - F(s,.) \right\Vert_{L^1}\right]\,ds.\notag
\end{align}

The next lemma states that the random variable $R^{N,h}(t,x)$ is centered and provides an upper-bound for $\E\left[\left\| R^{N,h}(t,.)\right\|_{L^1} \right]$.

 \begin{lem}\label{controlRN} We have $\forall N \in \N^{*}, \forall h,t \in [0,T]$, $ \left\| \E\left[R^{N,h}(t,.)\right]\right\|_{L^1} = 0 $. Moreover, if for some $\rho >1$, $\displaystyle \int_{\R} |x|^{\rho}m(dx) < \infty$, then:
 $$ \exists R< \infty,\;\forall N \in \N^{*},\;\forall h \in [0,T],\quad \sup_{t \leq T}\E\left[\left\| R^{N,h}(t,.) \right\|_{L^1}
\right] \leq \frac{R}{\sqrt{N}}.$$ 
\end{lem} 

\begin{proof} 
We have that $\displaystyle \int_{\R}\E\left[\int_0^t G^2_{t-s}(X^{i,N,h}_s - x)\,ds\right]\,dx = \E\left[\int_{\R}\int_0^t G^2_{t-s}(X^{i,N,h}_s - x)\,ds\,dx \right] \le \frac{1}{\sigma} \sqrt{\frac{t}{\pi}}  < \infty $ according to the estimate \eqref{GSquareEsp} from Lemma \ref{EstimHeatEq}. Therefore, $ \E\left[R^{N,h}(t,x) \right] = 0 \quad  \text{dx a.e.}$. Moreover, denoting $\displaystyle I_{\rho} = \int_{\R} \frac{dx}{1 + |x|^{\rho}}$ and using the It\^o isometry for the first equality then Young's inequality for the second inequality and last the estimate \eqref{GSquareEsp} from Lemma \ref{EstimHeatEq}, we obtain:

\begin{align*}
\displaystyle \E\left[ \left\Vert R^{N,h}(t,.)\right\Vert_{L^1} \right] &\le \int_{\R} \E^{1/2}\left[ \left( \frac{\sigma}{N}\sum \limits_{i=1}^N \int_0^t G_{t-s}(X^{i,N,h}_s -x)\,dW^{i}_s\right)^2\right]\,dx \\      
&=\frac{\sigma}{\sqrt{N}} \int_{\R} \E^{1/2} \left[ \frac{1}{N}\sum
\limits_{i=1}^N \int_0^t G^2_{t-s}(X^{i,N,h}_s -x)\,ds\right]\,dx \\   
&= \frac{\sigma}{\sqrt{N}}  \int_{\R} \E^{1/2} \left[ \frac{1}{N}\sum \limits_{i=1}^N \int_0^t G^2_{t-s}(X^{i,N,h}_s -x)\,ds (1 + |x|^{\rho}) \right] \frac{dx}{\sqrt{1+|x|^{\rho}}} \\      
&\le \frac{\sigma}{2\sqrt{N}} \int_{\R} \left(\frac{1}{1 + |x|^{\rho}} + \E\left[\frac{1}{N}\sum \limits_{i=1}^N \int_0^t G^2_{t-s}(X^{i,N,h}_s -x)\,ds (1 + |x|^{\rho}) \right]\right)\,dx \\
&= \frac{\sigma I_{\rho}}{2\sqrt{N}} + \frac{\sigma}{2\sqrt{N}}\E\left[\frac{1}{N}\sum \limits_{i=1}^N\int_0^t\left(\int_{\R}G^2_{t-s}(X^{i,N,h}_s -x)\,dx + \int_{\R} |x|^{\rho} G^2_{t-s}(X^{i,N,h}_s -x)\,dx \right)\,ds \right] \\     
&= \frac{\sigma I_{\rho}}{2\sqrt{N}} + \frac{\sigma}{2\sqrt{N}} \E\left[ \frac{1}{N}\sum
\limits_{i=1}^N  \frac{1}{\sigma} \, \sqrt{\frac{t}{\pi}} + \int_0^t \int_{\R} \frac{|X^{i,N,h}_s - y|^{\rho}}{2\sigma \, \sqrt{\pi (t-s)}}G_{(t-s)/2}(y)\,dy\,ds\right] \\     
&\le \frac{\sigma I_{\rho}}{2\,\sqrt{N}} +\frac{1}{2}\sqrt{\frac{t}{N\pi}} + \frac{2^{\rho -1}\sigma}{2\sqrt{N}}
\E\left[\frac{1}{N}\sum \limits_{i=1}^N \int_0^t \frac{|X^{i,N,h}_s|^{\rho}}{2\sigma \, \sqrt{\pi(t-s)}}\,ds + \frac{\sigma^{\rho-1}}{2 \pi} \Gamma\left(\frac{\rho +1}{2} \right) \int_0^t (t-s)^{(\rho-1)/2} \,ds \right]\\    
&= \frac{1}{2\sqrt{N}}\left(
\sigma I_{\rho} + \sqrt{\frac{t}{\pi}} \left(1 + \frac{2^{\rho -1}\sigma^{\rho}}{\rho +1}\Gamma\left(\frac{\rho +1}{2}\right)\sqrt{\frac{t^{\rho}}{\pi}}\right) + \frac{2^{\rho -1}}{\sqrt{\pi}}\int_0^t \E\left[ \frac{1}{N}\sum \limits_{i=1}^N |X^{i,N,h}_s|^{\rho} \right] \frac{ds}{2\sqrt{(t-s)}}\right).
\end{align*} 
With the use of Lemma \ref{controlM}, we conclude by setting $\displaystyle R
= \frac{1}{2} \left(\sigma I_{\rho}+ \sqrt{\frac{T}{\pi}}\left(1 +2^{\rho -1}M +\frac{2^{\rho -1}\sigma^{\rho}}{\rho +1}\Gamma\left(\frac{\rho +1}{2}\right)\sqrt{\frac{T^{\rho}}{\pi}}\right)\right)$.
\end{proof} 
Therefore, Inequality \eqref{controlConvoB} becomes: 
\begin{align}\label{controlConvo}
\displaystyle \E\left[\left\| F^{N,h}(t,.) - F(t,.) \right\|_{L^1} \right] \le \E\left[\left\| F^{N,h}_0 - F_0 \right\|_{L^1} \right] &+\frac{R}{\sqrt{N}}  +\E\left[\left\|E^{N,h}(t,.)\right\|_{L^1}\right] \\
&+A \int_0^t \frac{1}{\sqrt{t-s}}\E\left[\left\Vert F^{N,h}(s,.) - F(s,.) \right\Vert_{L^1}\right]\,ds.\notag 
\end{align}
\textbullet $\;$ One should notice that for $h=0$, $E^{N,0}(t,x)=0 \quad \forall t \in [0,T], N \in \N^{*}, x \in \R$. Therefore, to control the term $\E\left[ \mathcal{W}_1\left(\mu^{N,0}_t,\mu_t \right)\right]$, we iterate Inequality \eqref{controlConvo} and obtain:
\begin{align*}
    \displaystyle \E\left[\left\Vert F^{N,0}(t,.) - F(t,.) \right\Vert_{L^{1}}\right] \leq \left(2A \sqrt{t} +1\right) &\left( \E\left[\left\| F^{N,h}_0 - F_0 \right\|_{L^1} \right] + \frac{R}{\sqrt N} \right) \\
    &+ A^2 \displaystyle \int_0^t \E\left[\left\Vert F^{N,0}(r,.) - F(r,.) \right\Vert_{L^1}\right] \int_r^t \frac{ds}{\sqrt{t-s}\, \sqrt{s-r}}\,dr.
\end{align*}
Since $\displaystyle \int_r^t \frac{ds}{\sqrt{t-s}\,\sqrt{s-r}} = \pi$ and with the use of Lemma \ref{integrF}, we can apply Gr\"onwall's lemma to deduce that\\ $\displaystyle \forall N \in \N^{*}, \sup_{t \leq T}\E\left[ \mathcal{W}_1\left(\mu^{N,0}_t,\mu_t \right)\right] \le C\left(\E\left[ \mathcal{W}_1\left(\mu^{N}_0,m \right)\right] + \frac{1}{\sqrt{N}} \right)$ where $ C = \max(1,R)\left(2A \sqrt{T} +1 \right)\exp\left( A^2 \pi T\right)$. This concludes the proof of Theorem \ref{thm1} when $h=0$.\\ \\
\textbullet $\;$ When $h>0$, we need to estimate $\E\left[\left\| E^{N,h}(t,.) \right\|_{L^1} \right], h \in (0,T]$.
\begin{prop}\label{ControlFinalE}
We assume that for some $\rho >1$, $\displaystyle \int_{\R} |x|^{\rho}m(dx) < \infty$ and that the function $\lambda$ is Lipschitz continuous. Then $\exists Z <\infty, \forall N \in \N^{*}, \forall h \in (0,T],\forall t\in[0,T],$
$$\displaystyle \E \left[ \left\| E^{N,h}(t,.)\right\|_{L^1} \right] \leq Z  \left( \frac{1}{\sqrt{N}} + h + \sqrt{h}\E\left[\left\|F^{N,h}(t,.) - F(t,.) \right\|_{L^1} \right] + \int_{0}^{t} \frac{1}{2\sqrt{t-s}}\E\left[\left\|F(s,.) - F^{N,h}(s,.) \right\|_{L^1}\right]\,ds \right).$$
\end{prop}
Proposition \ref{ControlFinalE} will be proved in Section \ref{cetSec}. \\ \\
From Equation\eqref{controlConvo} and Proposition \ref{ControlFinalE}, we have that:
\begin{align*}
    \displaystyle &\left(1 - Z \sqrt{h} \right)\E\left[\left\Vert F^{N,h}(t,.) - F(t,.) \right\Vert_{L^{1}} \right] \\
    &\le \E\left[\left\Vert F^{N,h}_0 - F_0 \right\Vert_{L^{1}} \right]+ \frac{Z+R}{\sqrt{N}} + Z  h  + \left(A+\frac{Z}{2}\right) \int_0^t \frac{1}{\sqrt{t-s}}\E\left[\left\Vert F^{N,h}(s,.) - F(s,.) \right\Vert_{L^1}\right]\,ds.
\end{align*}
Hence, if we denote $\displaystyle J = 2\left(Z + R \right)$ and $K = 2A + Z$ then:                     
\begin{align}\label{EquaIciiii}
    \displaystyle &2\left(1 - Z \sqrt{h} \right) \E\left[\left\Vert F^{N,h}(t,.) - F(t,.) \right\Vert_{L^{1}} \right] \notag\\
    &\le 2\E\left[\left\Vert F^{N,h}_0 - F_0 \right\Vert_{L^{1}} \right] + J \left(\frac{1}{\sqrt{N}} + h \right) + \int_0^t \frac{K}{\sqrt{t-s}}\E\left[\left\Vert F^{N,h}(s,.) - F(s,.) \right\Vert_{L^1}\right]\,ds.
\end{align}
\begin{itemize}
    \item When $h \leq \frac{1}{4Z^2}$, Equation \eqref{EquaIciiii} implies:
\begin{align*}
    &\E\left[\left\|F^{N,h}(t,.) - F(t,.) \right\|_{L^1}\right] \\
    &\le 2\E\left[\left\Vert F^{N,h}_0 - F_0 \right\Vert_{L^{1}} \right]+ J\left(\frac{1}{\sqrt{N}} +h \right) + \int_0^t  \frac{K}{ \sqrt{t-s}}\E\left[ \left\Vert F^{N,h}(s,.) - F^N(s,.) \right\Vert_{L^1}\right]\,ds.
\end{align*}
We iterate this inequality to obtain:
\begin{align*}
    &\E\left[\left\|F^{N,h}(t,.) - F(t,.) \right\|_{L^1}\right] \\
    &\le \left(1 + 2 K \sqrt{t} \right) \left(2\E\left[\left\Vert F^{N,h}_0 - F_0 \right\Vert_{L^{1}} \right] + J\left(\frac{1}{\sqrt{N}}+ h \right) \right) +  K^2 \pi \int_0^t  \E\left[ \left\Vert F^{N,h}(r,.) - F^N(r,.) \right\Vert_{L^1}\right]\,dr.
\end{align*}
With the use of Lemma \ref{integrF}, we can apply Gr\"onwall's Lemma and deduce that:
$$\forall t \in [0,T], \quad \E\left[\left\|F^{N,h}(t,.) - F(t,.) \right\|_{L^1}\right] \le \left(1 + 2 K \sqrt{t} \right)\exp\left(K^2\pi t \right)\left(2\E\left[\left\Vert F^{N,h}_0 - F_0 \right\Vert_{L^{1}} \right] + J\left(\frac{1}{\sqrt{N}}+ h \right) \right). $$
 \item When $h>\frac{1}{4Z^2}$, by Lemma \ref{integrF} and \eqref{w1fdr},
\begin{align*}
	\displaystyle \E\left[\left\|F^{N,h}(t,.) - F(t,.) \right\|_{L^1}\right] &\le 2 \sigma \sqrt{\frac{2t}{\pi}} + 2 L_{\Lambda} t + \E\left[\left\Vert F^{N,h}_0 - F_0 \right\Vert_{L^{1}} \right]\\
	&\le 4Z^2h\left(2\sigma \sqrt{\frac{2t}{\pi}} + 2 L_{\Lambda} t\right)  + \E\left[\left\Vert F^{N,h}_0 - F_0 \right\Vert_{L^{1}} \right].
\end{align*}
\end{itemize}
We choose  $C = \max\left(\max(2,J)\left(1 + 2 K \sqrt{T} \right)\exp\left(K^2\pi T \right), 4Z^2\left(2\sigma \sqrt{\frac{2T}{\pi}} + 2 L_{\Lambda} T\right)\right) $ and conclude that:
$$\forall N \in \N^{*},\forall h \in (0,T],\forall t \in [0,T], \quad \sup_{t \leq T}\E\left[ \mathcal{W}_1\left(\mu^{N,h}_t,\mu_t \right)\right] \le  C\left(\E\left[ \mathcal{W}_1\left(\mu^{N}_0,m \right)\right] +  \left(\frac{1}{\sqrt{N}} + h \right)\right). $$

Let us now prove Proposition \ref{ControlFinalE} in the following section.

\subsubsection{Proof of Proposition \ref{ControlFinalE}}\label{cetSec}

We recall the expression of $E^{N,h}(t,x)$:

$$ \displaystyle E^{N,h}(t,x) = \frac{1}{N}\sum \limits_{i=1}^N  \int_0^t  G_{t-s}(Y^{i,N,h}_s-x) \left[ \lambda^{N}\left(i\right) - \lambda^{N}\left(\eta_{\tau^h_s}^{-1}(\eta_s(i))\right) \right]\,ds. $$
We do not know how to estimate the difference of values of $\lambda^N$ between the brackets. For $s>0$, we are going to take advantage of the permutation $\eta_s^{-1}\circ\eta_{\tau^h_s}$ (because of the convention \eqref{convinvinit}, this is not necessarily a permutation for $s=0$ and $\eta_{\tau^h_s}^{-1}\circ\eta_{\tau^h_s}$ is equal to the identity permutation for $s\ge h$ but not necessarily for $s\in[0,h)$) to change indices and obtain the same value multiplied by a difference of values of the smooth function $G_{t-s}$. Using this permutation for the first equality then that $Y^{\eta_s^{-1}(j),N,h}_s=X^{j,N,h}_s$ for $s>0$ and $1\le j\le N$ for the second one, we obtain that 
\begin{align}
  E^{N,h}(t,x) &= \frac{1}{N}\sum \limits_{i=1}^N  \int_0^t  G_{t-s}(Y^{i,N,h}_s-x)\lambda^{N}\left(i\right)-G_{t-s}\left(Y^{\eta_s^{-1}(\eta_{\tau^h_s}(i)),N,h}_s-x\right)\lambda^{N}\left(\eta_{\tau^h_s}^{-1}(\eta_{\tau^h_s}(i))\right)\,ds\notag\\
    &= \frac{1}{N}\sum \limits_{i=1}^N \int_0^t \left\{\left(G_{t-s}\left(Y^{i,N,h}_s -x\right)- G_{t-s}\left(X^{\eta_{\tau^h_s}(i),N,h}_s -x \right) \right)\right\}\lambda^N(i)\,ds,\notag\\
    &+\frac{1}{N}\sum \limits_{i=1}^N \int_{0}^{t\wedge h}G_{t-s}(X^{\eta_0(i),N,h}_s-x)\left(\lambda^N(i)-\lambda^N\left(\eta_0^{-1}(\eta_0(i))\right)\right)\,ds\label{decompENh}.
  \end{align}

Substracting $G_{t-\tau^h_s}\left(Y^{i,N,h}_{\tau^h_s}-x\right) - G_{t-\tau^h_s} \left(X^{\eta_{\tau^h_s}(i),N,h}_{\tau^h_s}-x\right) =0$  in the brace in the first term of the right-hand side makes apparent that this term is not too large since $\tau^h_s$ is close to $s$. 
Computing $G_{t-s}\left(Y^{i,N,h}_s -x\right)-G_{t-\tau^h_s}\left(Y^{i,N,h}_{\tau^h_s} -x\right)$ and $G_{t-s}\left(X^{\eta_{\tau^h_s}(i),N,h}_s -x \right) - G_{t-{\tau^h_s}}\left(X^{\eta_{\tau^h_s}(i),N,h}_{\tau^h_s} -x \right)$ by It\^o's formula, we obtain the following new expression of $E^{N,h}(t,x)$:
\begin{lem}\label{ControlE} 
The process $\tilde \beta=\left(\tilde \beta^1,\hdots,\tilde \beta^N \right)$ where $\displaystyle \tilde \beta^i_t=\sum_{j=1}^N\int_0^t\mathbf{1}_{\left\{\eta_{\tau^h_s}(i)=j\right\}}dW^j_s$ is a $N$-dimensional Brownian motion and 
we can express $E^{N,h}(t,x)$ as $E^{N,h}(t,x)=\sum \limits_{p=0}^5 e^{N,h}_p(t,x)$ where:
\begin{itemize}
  \item $ \displaystyle e^{N,h}_0(t,x) = \frac{1}{N}\sum \limits_{i=1}^N \int_{0}^{t\wedge h}G_{t-s}\left(X^{\eta_0(i),N,h}_s-x\right)\left(\lambda^N(i)-\lambda^N\left(\eta_0^{-1}(\eta_0(i))\right)\right)\,ds $,  
  \item $\displaystyle e^{N,h}_1(t,x) = \frac{1}{N}\sum \limits_{i=2}^N  \int_{0}^{t}(t \wedge \overline{\tau}^h_s -s)\left(  \lambda^N(i) - \lambda^N(i-1) \right)\partial_x G_{t-s}\left(Y^{i,N,h}_s -x \right)\gamma^{i}_s\,d|K|_s$,
   \item $\displaystyle e^{N,h}_2(t,x) = \frac{\sigma}{N} \sum \limits_{i=1}^N  \int_{0}^{t}(t \wedge \overline{\tau}^h_s -s)\lambda^N(i) \partial_x G_{t-s}\left(Y^{i,N,h}_s -x \right)\,d\beta^{i}_s$,
   \item $\displaystyle e^{N,h}_3(t,x) = - \frac{\sigma}{N}\sum \limits_{i=1}^N  \int_{0}^{t}(t \wedge \overline{\tau}^h_s -s)\lambda^N(i) \partial_x G_{t-s}\left(X^{\eta_{\tau^h_s}(i),N,h}_s -x \right)\,d\widetilde{\beta}^i_s$,
   \item $\displaystyle e^{N,h}_4(t,x) =  \frac{1}{N} \sum \limits_{i=1}^N  \int_{0}^{t}(t \wedge \overline{\tau}^h_s -s)\lambda^N(i) \lambda^N\left( \eta^{-1}_{\tau^h_s}(\eta_s(i))\right)\partial_x G_{t-s}\left(Y^{i,N,h}_s -x \right)\,ds$,
   \item $ \displaystyle e^{N,h}_5(t,x) = - \frac{1}{N}\sum \limits_{i=1}^N \int_{0}^{t}(t \wedge \overline{\tau}^h_s -s) \lambda^N(i)\lambda^N\left( \eta^{-1}_{\tau^h_s}\left(\eta_{\tau^h_s}(i)\right)\right)\partial_x G_{t-s}\left(X^{\eta_{\tau^h_s}(i),N,h}_s -x \right)\,ds $.
\end{itemize}
\end{lem}
Notice that in the definition of $e^{N,h}_5(t,x)$, $\lambda^N\left( \eta^{-1}_{\tau^h_s}(\eta_{\tau^h_s}(i))\right)=\lambda^N(i)$ for $s\ge h$, but because of the convention \eqref{convinvinit}, this equality does not necessarily hold for $s\in[0,h)$.

\begin{proof}
 For $1\le i,k\le N$ and $t\ge 0$, one has 
 $$\langle \tilde \beta^i,\tilde \beta^k\rangle_t=\sum_{j=1}^N\int_0^t\mathbf{1}_{\left\{\eta_{\tau^h_s}(i)=j \right\}}\mathbf{1}_{\left\{\eta_{\tau^h_s}(k)=j \right\}}\,ds=\int_0^t\mathbf{1}_{\left\{\eta_{\tau^h_s}(i)=\eta_{\tau^h_s}(k)\right\}}\,ds=\mathbf{1}_{\{i=k\}}t,$$ 
 since $\eta_{\tau^h_s}$ is a permutation for each $s\ge 0$ . One deduces that $\tilde \beta$ is a Brownian motion by applying L\'evy's characterization. 
  By \eqref{decompENh} and the equality $G_{t-\tau^h_s}\left(Y^{i,N,h}_{\tau^h_s}-x\right) - G_{t-\tau^h_s} \left(X^{\eta_{\tau^h_s}(i),N,h}_{\tau^h_s}-x\right) =0$, it is enough to check that 
  \begin{align*}
   \frac{1}{N}\sum \limits_{i=1}^N \int_0^t &\Bigg\{\left(G_{t-s}\left(Y^{i,N,h}_s -x\right)-G_{t-\tau^h_s}\left(Y^{i,N,h}_{\tau^h_s} -x\right) \right) \\
    &\phantom{\frac{1}{N}\sum \limits_{i=1}^N \int_0^t\{( - } - \left(G_{t-s}\left(X^{\eta_{\tau^h_s}(i),N,h}_s -x \right) - G_{t-{\tau^h_s}}\left(X^{\eta_{\tau^h_s}(i),N,h}_{\tau^h_s} -x \right) \right)\Bigg\} \lambda^N(i)\,ds=\sum_{p=1}^5e^{N,h}_p(t,x).
  \end{align*}
We are going to compute the two differences in the right-hand side by applying It\^{o}'s formula. To do so, let us recall the dynamics of $X^{\eta_{\tau^h_s}(i),N,h}_u$ for $u\in[\tau^h_s,\bar\tau^h_s)$ : 
$$dX^{\eta_{\tau^h_s}(i),N,h}_u = \sigma\,d\tilde \beta^{i}_u + \lambda^N \left(\eta_{\tau^h_s}^{-1}\left(\eta_{\tau^h_s}(i)\right)\right)\,du= \sigma\,d\tilde \beta^{i}_u + \lambda^N\left(\eta_{\tau^h_u}^{-1}\left(\eta_{\tau^h_u}(i)\right)\right)\,du.$$ 
We then have:
\begin{align*}
    \displaystyle &G_{t-s}\left(X^{\eta_{\tau^h_s}(i),N,h}_s -x \right) = G_{t-\tau^h_s}\left(X^{\eta_{\tau^h_s}(i),N,h}_{\tau^h_s} -x \right) + \sigma \int_{\tau^h_s}^s \partial_x G_{t-u}\left(X^{\eta_{\tau^h_s}(i),N,h}_u -x \right)\,d\tilde \beta^{i}_u \\
    &+ \int_{\tau^h_s}^s \left(\partial_u G_{t-u}\left(X^{\eta_{\tau^h_s}(i),N,h}_u -x \right) + \frac{\sigma^2}{2}\partial_{xx}G_{t-u}\left(X^{\eta_{\tau^h_s}(i),N,h}_u -x \right) + \lambda^N\left(\eta_{\tau^h_u}^{-1}\left(\eta_{\tau^h_u}(i)\right)\right)\partial_x G_{t-u}\left(X^{\eta_{\tau^h_s}(i),N,h}_u -x \right)\right)\,du. 
\end{align*}
Since $\partial_u G_{t-u} = - \partial_t G_{t-u}$, by the heat equation \eqref{heateq} from Lemma \ref{EstimHeatEq}, we have:
\begin{align}\label{eqXG}
     \displaystyle \int_{0}^t &\left( G_{t-s}\left(X^{\eta_{\tau^h_s}(i),N,h}_s -x \right)-  G_{t-\tau^h_s}\left(X^{\eta_{\tau^h_s}(i),N,h}_{\tau^h_s} -x \right) \right)\,ds \\
    &=  \sigma  \int_{0}^t \int_{\tau^h_s}^s \partial_x G_{t-u}\left(X^{\eta_{\tau^h_s}(i),N,h}_u -x \right)\,d\tilde \beta^{i}_u\,ds +   \int_{0}^t \int_{\tau^h_s}^s \lambda^N \left(\eta_{\tau^h_u}^{-1}\left(\eta_{\tau^h_u}(i)\right)\right)\partial_x G_{t-u}\left(X^{\eta_{\tau^h_s}(i),N,h}_u -x \right)\,du\,ds. \notag
\end{align}
Let us suppose that $t>0$ and treat each term of the right-hand side of the above equation. For $x \in \R$, The function $u \mapsto \partial_xG_{t-u}\left(X^{\eta_{\tau^h_s}(i),N,h}_u -x \right)$ is continuous on $[0,t)$. Since $\left( X^{1,N,h}_t, \dots,X^{N,N,h}_t \right) $ admits a density, as stated after the proof of Lemma \ref{ctyF}, $\mathbb{P}\left(X^{\eta_{\tau^h_s}(i),N,h}_t = x \right) \le \sum\limits_{j=1}^N \mathbb{P}\left(X^{j,N,h}_t = x \right) = 0$ a.s.. Therefore, a.s. the previous function has a vanishing limit as $u\to t$ and is therefore bounded on the interval $[0,t]$. We can then apply Fubini's theorem to obtain:
\begin{align*}
  \displaystyle \int_{0}^t \int_{\tau^h_s}^s\lambda^N\left(\eta_{\tau^h_u}^{-1}\left(\eta_{\tau^h_u}(i)\right)\right)& \partial_x G_{t-u}\left(X^{\eta_{\tau^h_s}(i),N,h}_u -x \right)\,du\,ds \\&=\int_{0}^{t}\left(t \wedge \overline{\tau}^h_u -u \right) \lambda^N \left(\eta_{\tau^h_u}^{-1}\left(\eta_{\tau^h_u}(i)\right)\right)\partial_x G_{t-u}\left(X^{\eta_{\tau^h_u}(i),N,h}_u -x \right)\,du \quad a.s..
\end{align*} 
Secondly, with the use of Young's inequality and the same arguments of density of $\left( X^{1,N,h}_t, \dots,X^{N,N,h}_t \right) $, we get:
\begin{align*}
    \displaystyle &\int_0^t \left( \int_{\tau^h_s}^s \left|\partial_x G_{t-u}\left(X^{\eta_{\tau^h_s}(i),N,h}_u -x \right) \right|^2\,du\right)^{1/2}\,ds \\
    &= \int_0^t \left( \int_{\tau^h_s}^s \frac{\left(X^{\eta_{\tau^h_s}(i),N,h}_u -x \right)^2}{2 \sigma^5 \sqrt{\pi} (t-u)^{5/2}} G_{(t-u)/2}\left(X^{\eta_{\tau^h_s}(i),N,h}_u -x \right) \,du\right)^{1/2}\,ds \\
    &\leq \frac{t}{2} + \frac{1}{2}\int_0^t \int_{\tau^h_s}^s \frac{\left(X^{\eta_{\tau^h_s}(i),N,h}_u -x \right)^2}{2 \sigma^5 \sqrt{\pi} (t-u)^{5/2}} G_{(t-u)/2}\left(X^{\eta_{\tau^h_s}(i),N,h}_u -x \right)\,du\,ds < \infty \quad a.s..
\end{align*}
Therefore, we can apply the stochastic Fubini Lemma \ref{StochFub} and obtain:
\begin{align*}
    \displaystyle \int_{0}^t \int_{\tau^h_s}^s \partial_x G_{t-u}\left(X^{\eta_{\tau^h_s}(i),N,h}_u -x \right)\,d\tilde \beta^{i}_u\,ds &= \int_{0}^{t}(t \wedge \overline{\tau}^h_s -u)\partial_x G_{t-u}\left(X^{\eta_{\tau^h_u}(i),N,h}_s -x \right)\,d\tilde \beta^{i}_u.
\end{align*}
Equation \eqref{eqXG} becomes:
\begin{align}\label{eqXGnew}
    \displaystyle& \int_{0}^t \left( G_{t-s}\left(X^{\eta_{\tau^h_s}(i),N,h}_s -x \right)-  G_{t-\tau^h_s}\left(X^{\eta_{\tau^h_s}(i),N,h}_{\tau^h_s} -x \right) \right)\,ds \\
    &=     \sigma \int_{0}^{t}\left(t \wedge \overline{\tau}^h_s -s\right)\partial_x G_{t-s}\left(X^{\eta_{\tau^h_s}(i),N,h}_s -x \right)\,d\tilde \beta^{i}_s  +  \int_{0}^{t}\left(t \wedge \overline{\tau}^h_s -s\right)\lambda^N\left(\eta_{\tau^h_s}^{-1}\left(\eta_{\tau^h_s}(i)\right)\right)\partial_x G_{t-s}\left(X^{\eta_{\tau^h_s}(i),N,h}_s -x \right)\,ds . \notag
\end{align}

Now, let us apply It\^{o}'s formula to $G_{t-s}\left(Y^{i,N,h}_s -x\right)$. Using once again the property \eqref{heateq} of the kernel $G_t(x)$ from Lemma \ref{EstimHeatEq} and the dynamics of $Y^{i,N,h}_s$ given by \eqref{dynamicY}, we have:
\begin{align*}
    \displaystyle G_{t-s}&\left(Y^{i,N,h}_s -x\right) - G_{t-\tau^h_s}\left(Y^{i,N,h}_{\tau^h_s} -x\right) =  \sigma \int_{\tau^h_s}^s \partial_x G_{t-u}\left(Y^{i,N,h}_u -x \right)\,d\beta^{i}_u \\
    &+  \int_{\tau^h_s}^s \partial_x G_{t-u}\left(Y^{i,N,h}_u -x \right) \lambda^N\left( \eta^{-1}_{\tau^h_u}(\eta_u(i) \right)\,du + \int_{\tau^h_s}^s \partial_x G_{t-u}\left(Y^{i,N,h}_u -x \right)\left(\gamma^{i}_u - \gamma^{i+1}_u \right)\,d|K|_u.
\end{align*}

We use the same reasoning as for $X^{\eta_{\tau^h_s}(i),N,h}_s$ to treat the integrals from $0$ to $t$ of the first two terms :
\begin{align*}
     \displaystyle \int_{0}^t \int_{\tau^h_s}^s \partial_x G_{t-u}\left(Y^{i,N,h}_u -x \right)\,d\beta^{i}_u\,ds &=  \int_{0}^{t}(t \wedge \overline{\tau}^h_s -s) \partial_x G_{t-s}\left(Y^{i,N,h}_s -x \right)\,d\beta^{i}_s, \\
    \int_{0}^t \int_{\tau^h_s}^s  \partial_x G_{t-u}\left(Y^{i,N,h}_u -x \right) \lambda^N\left( \eta^{-1}_{\tau^h_u}(\eta_u(i) \right)\,du\,ds &=  \int_{0}^{t}(t \wedge \overline{\tau}^h_s -s) \partial_x G_{t-s}\left(Y^{i,N,h}_s -x \right)\lambda^N\left( \eta^{-1}_{\tau^h_s}(\eta_s(i) \right)\,ds.
\end{align*}

As for the last term $ \displaystyle \int_{0}^t \int_{\tau^h_s}^s \partial_x G_{t-u}\left(Y^{i,N,h}_u -x \right)\left(\gamma^{i}_u - \gamma^{i+1}_u \right)\,d|K|_u\,ds$, we sum over $i \in \llbracket1,N\rrbracket$ after multiplying by $\lambda^N(i)$ then we apply Fubini's theorem. Using the property \eqref{procVarFin}, we finally obtain:
\begin{align*}
    \displaystyle \frac{1}{N}\sum \limits_{i=1}^N  \lambda^N(i)&\int_{0}^t \int_{\tau^h_s}^s  \partial_x G_{t-u}\left(Y^{i,N,h}_u -x \right)\left(\gamma^{i}_u - \gamma^{i+1}_u \right)\,d|K|_u\,ds \\
    &= \frac{1}{N}  \sum \limits_{i=2}^N \int_{0}^{t}(t \wedge \overline{\tau}^h_s -s) \left(  \lambda^N(i)\partial_x G_{t-s}\left(Y^{i,N,h}_s -x \right) - \lambda^N(i-1)\partial_x G_{t-s}\left(Y^{i-1,N,h}_s -x \right)  \right)\gamma^{i}_s\,d|K|_s \\
    &= \frac{1}{N}  \sum \limits_{i=2}^N \int_{0}^{t}(t \wedge \overline{\tau}^h_s -s) \left(  \lambda^N(i) - \lambda^N(i-1)  \right)\partial_x G_{t-s}\left(Y^{i,N,h}_s -x \right)\gamma^{i}_s\,d|K|_s.
\end{align*}
Therefore,
\begin{align*}
    \displaystyle &\frac{1}{N} \sum \limits_{i=1}^N \int_{0}^t \lambda^N(i) \left(G_{t-s}\left(Y^{i,N,h}_s -x\right) - G_{t-\tau^h_s}\left(Y^{i,N,h}_{\tau^h_s} -x\right) \right)\,ds \\
    &= \frac{\sigma}{N} \sum \limits_{i=1}^N  \lambda^N(i) \left\{\int_{0}^{t}(t \wedge \overline{\tau}^h_s -s) \partial_x G_{t-s}\left(Y^{i,N,h}_s -x \right)\,d\beta^{i}_s + \int_{0}^{t}(t \wedge \overline{\tau}^h_s -s) \lambda^N\left( \eta^{-1}_{\tau^h_s}(\eta_s(i) \right)\partial_x G_{t-s}\left(Y^{i,N,h}_s -x \right)\,ds \right\} \notag \\
    &+ \frac{1}{N}  \sum \limits_{i=2}^N\int_{0}^{t}(t \wedge \overline{\tau}^h_s -s) \left(  \lambda^N(i) - \lambda^N(i-1) \right)\partial_x G_{t-s}\left(Y^{i,N,h}_s -x \right)\gamma^{i}_s\,d|K|_s. \notag
\end{align*}
We conclude by combining this equality and the sum over $i \in \llbracket1,N\rrbracket$ of \eqref{eqXGnew} multiplied by $\lambda^N(i)/N$.
\end{proof}

Now that we got rid of the difference of $\lambda^N$ in the term $E^{N,h}(t,x)$, we can control the mean of the $L^1$-norm of this term. We present a succession of lemmas that will estimate each $\E\left[ \left\Vert e_p^{N,h}(t,.)\right\Vert_{L^1} \right]$ for $p \in \llbracket0,5\rrbracket$. \\

Since $G_{t-s}$ is a probability density, $\displaystyle \E \left[ \left\| e_{0}^{N,h}(t,.)\right\|_{L^1} \right] \le \frac{1}{N} \sum \limits_{i=1}^N   \int_{0}^{t\wedge h}\left| \lambda^N(i) - \lambda^N\left( \eta^{-1}_{0}(\eta_{0}(i))\right) \right| \,ds $. Therefore, we obtain the following result concerning the term $e^{N,h}_0(t,x)$:
\begin{lem}\label{ControlE0}
$$ \forall N \in \N^{*}, \forall h \in (0,T], \quad \sup_{t \le T} \left\| \E\left[ e^{N,h}_{0}(t,.)\right]\right\|_{L^1} \le \sup_{t \le T}\E \left[ \left\| e^{N,h}_{0}(t,.)\right\|_{L^1} \right] \le 2L_{\Lambda}h.$$\\
\end{lem}

We remark that the terms $e_4^{N,h}(t,x)$ and $e_5^{N,h}(t,x)$ are of the same nature.
\begin{lem}\label{ControlDetE}
For $r \in \{4,5\}$: $$\exists C_{4,5}<\infty, \forall N \in \N^{*}, \forall h \in (0,T], \quad \sup_{t \le T} \left\| \E\left[ e^{N,h}_{r}(t,.)\right]\right\|_{L^1} \le \sup_{t \le T}\E \left[ \left\| e^{N,h}_{r}(t,.)\right\|_{L^1} \right] \le C_{4,5}\, h.$$ 
\end{lem}
\begin{proof}
Let us treat the term $e^{N,h}_5(t,x)$. \\
We have, using the estimate \eqref{FirstDerivG} from Lemma \ref{EstimHeatEq} for the second inequality, then the opposite monotonicities of the functions $s \mapsto \frac{1}{\sqrt{t-s}} $ and $s \mapsto (t \wedge t_{k+1} -s) $ on the time interval $[t\wedge t_k,t\wedge t_{k+1}]$ for the third inequality that: 
\begin{align*}
    \displaystyle \left\| \E\left[ e^{N,h}_{5}(t,.)\right]\right\|_{L^1} &\le \E \left[ \left\| e_{5}^{N,h}(t,.)\right\|_{L^1} \right] \leq \frac{1}{N} \sum \limits_{i=1}^N   \int_{0}^{t}(t \wedge \overline{\tau}^h_s -s) \left| \lambda^N(i) \right| \left| \lambda^N\left( \eta^{-1}_{\tau^h_s}(\eta_{\tau^h_s}(i))\right) \right| \left\Vert \partial_xG_{t-s} \right\Vert_{L^1}\,ds \\
    &\le \frac{L_{\Lambda}^2}{\sigma}\sqrt{\frac{2}{\pi}}\sum_{k\in\N:t_k<t}\int_{t_k}^{t\wedge t_{k+1}}\frac{(t \wedge t_{k+1}-s)}{\sqrt{t-s}}\,ds \\&\le \frac{L_{\Lambda}^2}{\sigma}\sqrt{\frac{2}{\pi}} \left\{ \sum_{k\in\N:t_k<t}\frac{1}{t\wedge t_{k+1}-t_k}\int_{t_k}^{t\wedge t_{k+1}}(t \wedge t_{k+1}-s)ds\int_{t_k}^{t\wedge t_{k+1}}\frac{ds}{\sqrt{t-s}} \right\} \\
    &\le \frac{L_{\Lambda}^2}{\sigma}\sqrt{\frac{2}{\pi}} \left\{ \frac{h}{2}\sum_{k\in\N:t_k<t}\int_{t_k}^{t\wedge t_{k+1}}\frac{ds}{\sqrt{t-s}} \right\}=\frac{L_{\Lambda}^2}{\sigma}\sqrt{\frac{2t}{\pi}}\,h.
\end{align*}
The term $e^{N,h}_4(t,x)$ can be estimated in the same way and the conclusion holds with $\displaystyle C_{4,5} = \frac{L_{\Lambda}^2}{\sigma}\sqrt{\frac{2T}{\pi}}$.
\end{proof}
We remark that the terms $e_2^{N,h}(t,x)$ and $e_3^{N,h}(t,x)$ are of the same nature as well.
\begin{lem}\label{ControlStochE}
For $r \in \{2,3\}$, $\forall N \in \N^{*}, \forall h \in (0,T]$, $\left\| \E\left[ e^{N,h}_{r}(t,.)\right]\right\|_{L^1} = 0$. Moroever, if $\displaystyle \int_{\R} |x|^{\rho}m(dx) < \infty$ for some $\rho >1$, then:
  $$\exists C_{2,3}<\infty, \quad  \sup_{t \le T}\E \left[ \left\| e^{N,h}_{r}(t,.)\right\|_{L^1} \right] \leq \frac{C_{2,3}}{\sqrt{N}}.$$ 
\end{lem}
\begin{proof}
Let us treat the term $e_{3}^{N,h}(t,x)$. \\
Using the estimate \eqref{firstDerivSquar} from Lemma \ref{EstimHeatEq}, we obtain that
\begin{align*}
   \int_{\R} \E\left[ \int_0^t(t \wedge \overline{\tau}^h_s -s)^2 (\lambda^N(i))^2\left(\partial_xG_{t-s}\right)^2\left(X^{\eta_{\tau^h_s}(i),N,h}_s-x \right)\,ds \right]\,dx\le \int_0^t\frac{(t-s)^2L_\Lambda^2}{4\sigma^3(t-s)^{3/2}\sqrt{\pi}}ds<\infty.
\end{align*}
Therefore, $ \E\left[e^{N,h}_3(t,x) \right] = 0 $ dx a.e.. Moreover, using the estimate \eqref{GPrimSqu} from Lemma \ref{EstimHeatEq}, the It\^o isometry for the first equality, Young's inequality for the second inequality and last Fubini's theorem, we obtain:
\begin{align*}
    \displaystyle &\E \left[ \left\| e_{3}^{N,h}(t,.)\right\|_{L^1} \right] \leq \int_{\R} \E^{1/2}\left[ \left(\frac{\sigma}{N}\sum \limits_{i=1}^N \int_{0}^{t}(t \wedge \overline{\tau}^h_s -s) \lambda^N(i)\partial_x G_{t-s}\left(X^{\eta_{\tau^h_s}(i),N,h}_s -x \right)\,d\tilde \beta^{i}_s \right)^2 \right]\,dx \\
    &= \frac{\sigma}{\sqrt{N}} \int_{\R} \E^{1/2}\left[ \frac{1}{N}\sum \limits_{i=1}^N \int_{0}^{t}(t \wedge \overline{\tau}^h_s -s)^2 \left|\lambda^N(i)\right|^2 \left(\partial_x G_{t-s}\right)^2\left(X^{\eta_{\tau^h_s}(i),N,h}_s -x \right)\,ds\right]\,dx \\
    &= \frac{\sigma}{\sqrt{N}} \int_{\R} \E^{1/2}\left[ \frac{1}{N}\sum \limits_{i=1}^N \int_{0}^{t}(t \wedge \overline{\tau}^h_s -s)^2 \left|\lambda^N(i)\right|^2 \left(\partial_x G_{t-s}\right)^2\left(X^{\eta_{\tau^h_s}(i),N,h}_s -x \right)\,ds(1+|x|^{\rho})  \right]\frac{dx}{\sqrt{1+|x|^{\rho}}} \\
    &\leq \frac{\sigma }{2\sqrt{N}} \int_{\R} \left( \frac{1}{1+|x|^{\rho}} + L_{\Lambda}^2\E\left[ \frac{1}{N}\sum \limits_{i=1}^N \int_{0}^{t}(t \wedge \overline{\tau}^h_s -s)^2  \left(\partial_x G_{t-s}\right)^2\left(X^{\eta_{\tau^h_s}(i),N,h}_s -x \right)\,ds(1+|x|^{\rho}) \right]\right)\,dx \\
    &= \frac{\sigma I_{\rho} }{2\sqrt{N}} + \frac{L_{\Lambda}^2}{4 \sigma^4\sqrt{N \pi}}\E\Bigg[ \frac{1}{N}\sum \limits_{i=1}^N \int_{0}^{t} \frac{(t \wedge \overline{\tau}^h_s -s)^2}{(t-s)^{5/2}}\bigg\{ \int_{\R} \left(X^{\eta_{\tau^h_s}(i),N,h}_s -x \right)^2G_{(t-s)/2}\left(X^{\eta_{\tau^h_s}(i),N,h}_s -x \right)\,dx \\
    &\phantom{\frac{\sigma \pi }{2\sqrt{N}} + \frac{L_{\Lambda}^2}{4 \sigma^4\sqrt{N \pi}} \E\Bigg[\frac{1}{N}\sum \limits_{i=1}^N \int_{0}^{t} \frac{(t \wedge \overline{\tau}^h_s -s)^2}{(t-s)^{5/2}} \bigg\{ \;}+ \int_{\R} |x|^{\rho}\left(X^{\eta_{\tau^h_s}(i),N,h}_s -x \right)^2G_{(t-s)/2}\left(X^{\eta_{\tau^h_s}(i),N,h}_s -x \right)\,dx\bigg\}\,ds\Bigg] \\
    &= \frac{\sigma I_{\rho} }{2\sqrt{N}} + \frac{L_{\Lambda}^2}{4 \sigma^4\sqrt{N \pi}}\E\left[ \frac{1}{N}\sum \limits_{i=1}^N \int_{0}^{t} \frac{(t \wedge \overline{\tau}^h_s -s)^2}{(t-s)^{
    5/2}}\left\{ \int_{\R} y^2G_{(t-s)/2}(y)\,dy + \int_{\R} \left|X^{\eta_{\tau^h_s}(i),N,h}_s -y \right|^{\rho}y^2G_{(t-s)/2}(y)\,dy \right\}\,ds\right] \\
    &\le \frac{\sigma I_{\rho} }{2\sqrt{N}} + \frac{L_{\Lambda}^2}{4 \sigma^4\sqrt{N \pi}}\E\Bigg[ \frac{1}{N}\sum \limits_{i=1}^N \int_{0}^{t} \frac{(t \wedge \overline{\tau}^h_s -s)^2}{(t-s)^{5/2}}\Bigg\{ \left(1 + 2^{\rho -1} \left|X^{\eta_{\tau^h_s}(i),N,h}_s \right|^{\rho}\right) \frac{\sigma^2(t-s)}{2}  \\
    &\phantom{\frac{\sigma I_{\rho} }{2\sqrt{N}} + \frac{L_{\Lambda}^2}{4 \sigma^4\sqrt{N \pi}}\E\Bigg[ \frac{1}{N}\sum \limits_{i=1}^N \int_{0}^{t} \frac{(t \wedge \overline{\tau}^h_s -s)^2}{(t-s)^{5/2}}\Bigg\{ \left(1 + 2^{\rho -1} \left|X^{\eta_{\tau^h_s}(i),N,h}_s \right|^{\rho}\right) \Bigg\{\;} + 2^{\rho -1} \frac{\left(\sigma \sqrt{t-s} \right)^{2+\rho}}{\sqrt \pi}\Gamma\left(\frac{\rho +3}{2} \right) \Bigg\}\,ds\Bigg] \\
    &= \frac{1}{2\sqrt{N}}\Bigg( \sigma I_{\rho} + \frac{L_{\Lambda}^2}{4 \sigma^2 \sqrt{\pi}} \Bigg\{ \int_0^t \frac{(t \wedge \overline{\tau}^h_s -s)^2}{(t-s)^{3/2}}\,ds + 2^{\rho-1}\int_0^t \frac{(t \wedge \overline{\tau}^h_s -s)^2}{(t-s)^{3/2}} \E\left[\frac{1}{N}\sum \limits_{i=1}^N \left|X^{\eta_{\tau^h_s}(i),N,h}_s\right|^{\rho} \right]\,ds \\
    &\phantom{\frac{1}{2\sqrt{N}}\Bigg( \sigma I_{\rho} + \frac{L_{\Lambda}^2}{4 \sigma^2 \sqrt{\pi}} \Bigg\{ \;}+ \frac{2^{\rho -2}\sigma^{\rho}}{\sqrt \pi}\Gamma\left(\frac{\rho + 3}{2} \right) \int_0^t \frac{(t \wedge \overline{\tau}^h_s -s)^2}{(t-s)^{\frac{3-\rho}{2}}}\,ds \Bigg\} \Bigg). 
\end{align*}
With the use of Lemma \ref{controlM}, we obtain:
\begin{align*}
\displaystyle \E \left[ \left\| e_{3}^{N,h}(t,.)\right\|_{L^1} \right] &\le \frac{1}{2\sqrt{N}}\Bigg(\sigma I_{\rho} + \frac{L_{\Lambda}^2}{4 \sigma^2 \sqrt{\pi}} \Bigg\{\left(1+2^{\rho-1}M \right) \int_0^t \frac{(t \wedge \overline{\tau}^h_s -s)^2}{(t-s)^{3/2}}\,ds \\
&\phantom{\frac{1}{2\sqrt{N}}\Bigg(\sigma I_{\rho} + \frac{L_{\Lambda}^2}{4 \sigma^2 \sqrt{\pi}} \Bigg\{\left(1+2^{\rho-1}M \right) \Bigg\{\;} + \frac{2^{\rho -2}\sigma^{\rho}}{\sqrt \pi}\Gamma\left(\frac{\rho + 3}{2} \right) \int_0^t \frac{(t \wedge \overline{\tau}^h_s -s)^2}{(t-s)^{\frac{3-\rho}{2}}}\,ds \Bigg\} \Bigg).
\end{align*}
When $t \ge h$, $\displaystyle\int_0^t \frac{(t \wedge \overline{\tau}^h_s -s)^2}{(t-s)^{3/2}}\,ds \leq h^2\int_0^{t-h}\frac{ds}{(t-s)^{3/2}} + \int_{t-h}^t \sqrt{t-s}\,ds  \leq \frac{8}{3}h^{3/2}$ and $\displaystyle \int_0^t \frac{(t \wedge \overline{\tau}^h_s -s)^2}{(t-s)^{\frac{3-\rho}{2}}}\,ds \le \frac{2}{\rho-1}h^2 t^{\frac{\rho-1}{2}} + \frac{2}{\rho +3}h^{\frac{\rho +3}{2}}\le \frac{2(\rho+2)}{(\rho-1)(\rho+3)}h^{\frac{\rho +3}{2}}$. When $t \le h$, $\displaystyle\int_0^t \frac{(t \wedge \overline{\tau}^h_s -s)^2}{(t-s)^{3/2}}\,ds = \int_0^t \sqrt{t-s}\,ds = \frac{2}{3}t^{3/2} \leq \frac{2}{3}h^{3/2}$ and $\displaystyle \int_0^t \frac{(t \wedge \overline{\tau}^h_s -s)^2}{(t-s)^{\frac{3-\rho}{2}}}\,ds \le \frac{2}{\rho +3}h^{\frac{\rho +3}{2}}$. The term $e^{N,h}_2(t,x)$ can be estimated in the same way and the conclusion holds with  $$\displaystyle C_{2,3} = \frac{1}{2}\left( \sigma I_{\rho} + \frac{L_{\Lambda}^2 T^{3/2}}{4 \sigma^2 \sqrt{\pi}} \left( \frac{8}{3}\left(1+2^{\rho-1}M \right) + \frac{2^{\rho -1}\sigma^{\rho}}{\sqrt \pi}\times\frac{2(\rho+2)}{(\rho-1)(\rho+3)}\Gamma\left(\frac{\rho + 3}{2} \right) T^{\frac{\rho -1}{2}}\right) \right).$$
\end{proof}

Now, we finally treat the term $e^{N,h}_1(t,x)$ in the lemma below:
\begin{lem}\label{ControlGam}
    If $\lambda$ is Lipschitz continuous with constant $L_{\lambda}$ then $\exists C_{1}<\infty, \forall N \in \N^{*}, \forall h \in (0,T], \forall t \in [0,T]$,
    $$\displaystyle \E \left[ \left\| e^{N,h}_{1}(t,.)\right\|_{L^1} \right] \leq C_1\left(h + \sqrt{h}\E \left[ \left\| F^{N,h}(t,.) - F(t,.) \right\|_{L^1} \right] +  {\mathbf 1}_{\{t\ge h\}}\int_0^{t-h} \frac{h}{2(t-s)^{3/2}}\E \left[ \left\| F^{N,h}(s,.) - F(s,.) \right\|_{L^1} \right]\,ds \right).$$ 
\end{lem}
The proof of this assertion relies on the following results.
\begin{lem}\label{AccroissF}
  We have: $$\exists Q < \infty, \forall 0 \le s \le t \le T, \quad \left\| F(t,.) - F(s,.)\right\|_{L^1} \leq Q \left(\sqrt{t}-\sqrt{s}\right) - L_{\Lambda}(t-s)\ln(t-s). $$
\end{lem}
\begin{proof}
Let $ 0 \le s \le t \le T$. We recall that:
$$\displaystyle F(t,x) - F(s,x) = \left(G_t - G_s \right)* F_0(x) - \int_0^t \partial_xG_{t-u} * \Lambda\left(F(u,.) \right)(x)\,du + \int_0^s \partial_xG_{s-u} * \Lambda\left(F(u,.) \right)(x)\,du.$$
Using Equality \eqref{heateq} and the estimates \eqref{FirstDerivG} and \eqref{SecondDerivG} from Lemma \ref{EstimHeatEq}, as well as the fact that $(t-s) \leq 2\sqrt{T}\left(\sqrt{t} - \sqrt{s} \right)$ and $t\ln t-s\ln s=\int_s^t(1+\ln x)dx\le \int_s^txdx=\frac{t+s}{2}(t-s)\le T\times 2\sqrt{T}(\sqrt{t}-\sqrt{s})$, we obtain:
\begin{align*}
    \displaystyle &\left\|F(t,.) - F(s,.) \right\|_{L^1} \\
    &\le \left\|\left(G_t - G_s \right)* F_0 \right\|_{L^1}+ \left\|\int_0^s\left(G_{t-u}-G_{s-u}\right)*\partial_x\Lambda\left(F(u,.)\right)\,du \right\|_{L^1} + \left\|\int_s^t G_{t-u}*\partial_x\Lambda\left(F(u,.)\right)\,du\right\|_{L^1} \\
    &\le \left\|\int_s^t \left(\partial_u G_u*F_0\right)\,du \right\|_{L^1} + \left\|\int_0^s \int_{s-u}^{t-u} \partial_r G_r* \partial_x\Lambda\left(F(u,.)\right)\,dr\,du\right\|_{L^1} + \int_s^t \left\| G_{t-u}*\left(\lambda\left(F(u,.) \right)p(u,.)\right)\right\|_{L^1}\,du  \\
    &\le \frac{\sigma^2}{2}\int_s^t \left\| \left(\partial_{x}G_u*m\right)\right\|_{L^1} \,du + \frac{\sigma^2}{2}\int_0^s \int_{s-u}^{t-u} \left\|  \partial_{xx}G_r* \left(\lambda\left(F(u,.)\right)p(u,.)\right) \right\|_{L^1} \,dr\,du + L_{\Lambda}(t-s) \\
    &\le \frac{\sigma^2}{2}\int_s^t\sqrt{\frac{2}{\pi \sigma^2 u}}\,du + L_{\Lambda}\int_0^s \ln\left(\frac{t-u}{s-u}\right)\,du + L_{\Lambda}(t-s) \\
    &=  \sigma \sqrt{\frac{2}{\pi}} \left(\sqrt{t}-\sqrt{s} \right) + L_{\Lambda}\left( t\ln(t) - s\ln(s) - (t-s)\ln(t-s)\right) + L_{\Lambda}(t-s)\\
    &\leq \left(\sigma \sqrt{\frac{2}{\pi}}+ 2L_{\Lambda} \sqrt{T}\left(1+T\right) \right)\left(\sqrt{t}-\sqrt{s}\right) - L_{\Lambda} (t-s)\ln(t-s).
\end{align*}
The conclusion holds with $Q=\sigma \sqrt{2 / \pi}+ 2 L_{\Lambda}\sqrt{T}\left(1+T\right)$.
\end{proof}
The next lemma provides two different estimations of the term $\displaystyle \E\left[ \int_{s}^{t}  \gamma^{i}_u d|K|_u \right]$. They are both useful to prove Lemma \ref{ControlGam}.
\begin{lem}\label{estimGam}
$\forall N \in \N^{*}, \forall i \in \llbracket2,N\rrbracket,  \forall h \in (0,T], \forall 0 \le s \le t \le T$,
\begin{align}\label{estimGam1}
    &\displaystyle \E\left[\left(\int_{s}^{t}\gamma^{i}_u\,d|K|_u\right)^2 \right] \leq  9N^2 \left(\sigma^2 + L_{\Lambda}^2 T\right)(t-s),
\end{align}
\begin{align}\label{estimGam2}
    &\mbox{and }\displaystyle \E\left[ \int_{s}^{t}\gamma^{i}_u\,d|K|_u \right] \leq N \left(\E\left[ \left\|F^{N,h}(t,.) - F^{N,h}(s,.) \right\|_{L^1} \right] + L_{\Lambda} (t-s)\right).
\end{align}
\end{lem}
\begin{proof}
Let $2 \le i \le N$. Since $\gamma^{N+1}_u = 0$, we have $\displaystyle \int_{s}^{t} \gamma^{i}_u\,d|K|_u   = \int_{s}^{t} \sum \limits_{j=i}^{N} \left(\gamma^{j}_u - \gamma^{j+1}_u \right)\,d|K|_u$ and with the dynamics \eqref{dynamicY} of $Y^{j,N,h}$, we deduce that  
\begin{align*}
   \displaystyle \int_{s}^{t} \gamma^{i}_u\,d|K|_u   &= \int_{s}^{t} \sum \limits_{j=i}^{N} \left(\gamma^{j}_u - \gamma^{j+1}_u \right)\,d|K|_u =  \sum \limits_{j=i}^{N}\left\{ \left( Y^{j,N,h}_t - Y^{j,N,h}_s \right) - \sigma\left( \beta^{j}_t - \beta^{j}_s \right) - \int_s^t \lambda^N \left( \sigma^{-1}_{\tau^h_u}(\eta_{u}(j))\right)\,du \right\}.
\end{align*}  
Let us start by proving the estimation of $\displaystyle \E\left[\left( \int_{s}^{t}\gamma^{i}_u\,d|K|_u \right)^2\right]$. With the use of Jensen's inequality and Lemma \ref{XandY} for $\rho = 2$, we obtain: 
\begin{align*}
    \displaystyle \E\left[\left(\int_{s}^{t}\gamma^{i}_u\,d|K|_u\right)^2 \right] &\leq 3 N \left( \sum \limits_{j=1}^{N} \E\left[\left| Y^{j,N,h}_t - Y^{j,N,h}_s \right|^2\right] + \sum \limits_{j=1}^{N}\left(\int_{s}^{t} \left| \lambda^N\left( \sigma^{-1}_{\tau^h_u}(\eta_{u}(j))\right)\right|\,du \right)^2 + \sigma^2\sum \limits_{j=1}^{N} \E\left[ \left|\beta^{j}_t - \beta^{j}_s \right|^2\right] \right) \\
    &\le 3N \left( \sum \limits_{j=1}^{N}\E\left[\left| X^{j,N,h}_t - X^{j,N,h}_s \right|^2\right] + N L_{\Lambda}^2 (t-s)^2 + N \sigma^2(t-s) \right)\\
  &\le 3N \left( \sum \limits_{j=1}^{N}2\E\left[\sigma^2|W^j_t-W^j_s|^2+L_{\Lambda}^2 (t-s)^2
    \right] + N L_{\Lambda}^2 (t-s)^2 + N \sigma^2(t-s) \right)\\&\leq 9N^2\left(\sigma^2 + L_{\Lambda}^2 T \right)(t-s).
\end{align*}
Notice that because of the latter contribution of $\E\left[|W^j_t-W^j_s|^2\right]$, it was not useful to take advantage of the independence of the Brownian motions $\beta^j$ which ensures $\E\left[ \left|\sum \limits_{j=i}^{N}(\beta^{j}_t - \beta^{j}_s) \right|^2\right]=(N+1-i)(t-s)$.
Let us now prove the second estimation of $\displaystyle \E\left[ \int_{s}^{t}\gamma^{i}_u\,d|K|_u \right]$. To do so, we use that, according to \eqref{Wasserstein} and \eqref{w1fdr}, $\displaystyle \frac{1}{N} \sum \limits_{i=1}^N \left|Y^{i,N,h}_t - Y^{i,N,h}_s \right| = \mathcal{W}_1\left(\mu^{N,h}_t, \mu^{N,h}_s\right) =  \int_{\R}\left|F^{N,h}(t,x) - F^{N,h}(s,x)\right|\,dx = \left\|F^{N,h}(t,.) - F^{N,h}(s,.)\right\|_{L^1}$ to obtain: 
\begin{align*}
    \displaystyle \E\left[\int_{s}^{t}\gamma^{i}_u\,d|K|_u \right] \leq& \sum \limits_{j=1}^{N} \E\left[\left| Y^{j,N,h}_t - Y^{j,N,h}_s \right|\right] + \sum \limits_{j=1}^{N} \int_{s}^{t} \left| \lambda^N\left( \sigma^{-1}_{\tau^h_u}(\eta_{u}(j)) \right)\right|\,du  \\
    \leq& N \E\left[ \left\|F^{N,h}(t,.) - F^{N,h}(s,.) \right\|_{L^1} \right] + N L_{\Lambda} (t-s).
\end{align*}
\end{proof}

Let us now prove Lemma \ref{ControlGam}. 
\begin{proof}
We recall that $\displaystyle e^{N,h}_1(t,x) = \frac{1}{N}\sum \limits_{i=2}^N \int_{0}^{t}(t \wedge \overline{\tau}^h_s -s)\left(  \lambda^N(i) - \lambda^N(i-1) \right)\partial_x G_{t-s}\left(Y^{i,N,h}_s -x \right)\gamma^{i}_s\,d|K|_s$. For $i \in \llbracket 2,N \rrbracket$, we have $\displaystyle \left|\lambda^N(i)-\lambda^N(i-1) \right| = \left| N\int_{\frac{i-1}{N}}^{\frac{i}{N}} \left( \lambda(u) - \lambda\left(u-\frac{1}{N}\right)\right)\,du\right| \le \frac{L_{\lambda}}{N}$. Using the estimate \eqref{FirstDerivG} from Lemma \ref{EstimHeatEq} and the property \eqref{procVarFin}, we have:
\begin{align}\label{here}
    \displaystyle \E\left[\left\|e_{1}^{N,h}(t,.)\right\|_{L^1} \right] &\leq \frac{L_{\lambda}}{\sigma N}\sqrt{\frac{2}{\pi}}\left\{\frac{1}{N}\sum_{i=2}^N  \E\left[ \int_{0}^{t} \frac{(t \wedge \overline{\tau}^h_s -s)}{\sqrt{t-s}}\gamma^{i}_s\,d|K|_s \right]\right\}.
\end{align}
\begin{itemize}
    \item For $t \le h$, since $\frac{(t \wedge \overline{\tau}^h_s -s)}{\sqrt{t-s}} = \sqrt{t-s} \le \sqrt h$, we deduce from \eqref{estimGam1} that $\displaystyle \E\left[\left\|e_{1}^{N,h}(t,.)\right\|_{L^1} \right] \leq \frac{3L_{\lambda}}{\sigma}\sqrt{\frac{2(\sigma^2 + L_{\Lambda}^2T)}{\pi}}\,h$.
    \item For $t \ge h$, we decompose the right-hand side of inequality \eqref{here} onto the sub-intervals $[0,t-h]$ and $[t-h,t]$ for a better control. Therefore,
    \begin{align*}
       \displaystyle \E\left[\left\|e_{1}^{N,h}(t,.)\right\|_{L^1} \right] &\le \frac{L_{\lambda}h}{\sigma N^2}\sqrt{\frac{2}{\pi}} \sum_{i=2}^N \left(\E\left[ \int_{0}^{t-h} \frac{1}{\sqrt{t-s}}\gamma^{i}_s\,d|K|_s \right] +  \frac{1}{\sqrt{h}}\E\left[ \int_{t-h}^{t}\gamma^{i}_s\,d|K|_s \right]\right).
    \end{align*}
    As for the first term of the right-hand side of the above inequality, we introduce $\displaystyle A_s = - \int_s^t \gamma_u^{i}d|K|_u$ and apply Fubini's theorem to obtain:
    \begin{align*}
    \displaystyle \int_0^{t-h}A_s \frac{ds}{2(t-s)^{3/2}} &= \int_0^{t-h} \left(A_0 + \int_0^sdA_r \right)\frac{ds}{2(t-s)^{3/2}} = A_0\left( \frac{1}{\sqrt h} - \frac{1}{\sqrt t}\right) + \int_0^{t-h}\int_0^s\frac{dA_r}{2(t-s)^{3/2}}\,ds \\
    &= A_0\left( \frac{1}{\sqrt h} - \frac{1}{\sqrt t}\right) + \int_0^{t-h}\int_r^{t-h}\frac{ds}{2(t-s)^{3/2}}\,dA_r \\
    &= A_0\left( \frac{1}{\sqrt h} - \frac{1}{\sqrt t}\right) + \int_0^{t-h} \left( \frac{1}{\sqrt h} - \frac{1}{\sqrt{t-r}}\right)\,dA_r \\
    &= -\frac{1}{\sqrt t}A_0 + \frac{1}{\sqrt h}A_{t-h} - \int_0^{t-h}\frac{1}{\sqrt{t-r}}\,dA_r.
    \end{align*}
    Consequently, we obtain that:
    \begin{align*}
      \displaystyle  \E\left[ \int_{0}^{t-h} \frac{1}{\sqrt{t-s}}\gamma^{i}_s\,d|K|_s \right] + \frac{1}{\sqrt{h}} \E\left[ \int_{t-h}^t\gamma_u^{i}\,d|K|_u\right] = \frac{1}{\sqrt{t}}\E\left[\int_0^t\gamma_u^{i}\,d|K|_u\right] + \E\left[\int_{0}^{t-h} \frac{1}{2(t-s)^{3/2}}\int_{s}^{t}\gamma^{i}_u\,d|K|_u\,ds \right].
    \end{align*}

    We shall use the estimate \eqref{estimGam1} and the estimate \eqref{estimGam2} from Lemma \ref{estimGam} for respectively the first term and the second term of the right-hand side of the following inequality:
    \begin{align*}
       \displaystyle &\E\left[\left\|e_{1}^{N,h}(t,.)\right\|_{L^1} \right] \\
       &\le \frac{L_{\lambda}h}{\sigma N^2}\sqrt{\frac{2}{\pi}}\sum_{i=2}^N \left\{ \frac{1}{\sqrt{t}} \E\left[\int_0^t\gamma_u^{i}\,d|K|_u\right] +\int_{0}^{t-h} \frac{1}{2(t-s)^{3/2}}\E\left[ \int_{s}^{t}\gamma^{i}_u\,d|K|_u\right]\,ds \right\} \\
        &\le \frac{L_{\lambda}h}{\sigma N^2}\sqrt{\frac{2}{\pi}}\sum_{i=2}^N \Bigg\{ 3N \sqrt{ \left(\sigma^2 +  L_{\Lambda}^2 T\right)} + N \int_{0}^{t-h} \frac{L_{\Lambda}}{2\sqrt{t-s}}\,ds + N\int_{0}^{t-h} \frac{1}{2(t-s)^{3/2}}\E\left[ \left\|F^{N,h}(t,.) - F^{N,h}(s,.) \right\|_{L^1} \right] \,ds \Bigg\} \\
        &=  \frac{L_{\lambda}}{\sigma }\sqrt{\frac{2}{\pi}}\left(3\sqrt{\left(\sigma^2 + L_{\Lambda}^2 T\right)} + L_{\Lambda}\left(\sqrt{t}-\sqrt{h} \right) \right)h + \frac{L_{\lambda}}{\sigma}\sqrt{\frac{2}{\pi}}\int_{0}^{t-h} \frac{h}{2(t-s)^{3/2}}\E\left[ \left\|F^{N,h}(t,.) - F^{N,h}(s,.) \right\|_{L^1} \right]\,ds \\
       &\le \frac{L_{\lambda}}{\sigma}\sqrt{\frac{2}{\pi}}\left( 3\sqrt{\left(\sigma^2 + L_{\Lambda}^2 T\right)} + L_{\Lambda}\sqrt{T}\right)h  + \frac{L_{\lambda}}{\sigma}\sqrt{\frac{2}{\pi}}\int_{0}^{t-h} \frac{h}{2(t-s)^{3/2}}\E\left[ \left\|F^{N,h}(t,.) - F^{N,h}(s,.) \right\|_{L^1} \right]\,ds.
    \end{align*}
    Since 
    \begin{align*}
       \displaystyle \E\left[ \left\|F^{N,h}(t,.) - F^{N,h}(s,.) \right\|_{L^1} \right] \leq  \E\left[ \left\|F^{N,h}(t,.) - F(t,.) \right\|_{L^1} \right]&+ \left\|F(t,.) - F(s,.) \right\|_{L^1} + \E\left[ \left\|F(s,.) - F^{N,h}(s,.) \right\|_{L^1} \right], 
    \end{align*}
    using Lemma \ref{AccroissF}, we obtain:
    \begin{align*}
    \displaystyle \int_{0}^{t-h} \frac{h}{2(t-s)^{3/2}}&\E\left[ \left\|F^{N,h}(t,.) - F^{N,h}(s,.) \right\|_{L^1} \right]\,ds \le \sqrt{h}\E\left[ \left\|F^{N,h}(t,.) - F(t,.) \right\|_{L^1} \right] \\
    &+ h\int_{0}^{t-h} \frac{\E\left[ \left\|F(s,.) - F^{N,h}(s,.) \right\|_{L^1}\right]}{2(t-s)^{3/2}}\,ds + h\int_{0}^{t} \left(Q\frac{\sqrt{t}-\sqrt{s} }{2(t-s)^{3/2}} - \frac{L_{\Lambda}}{2} \frac{\ln(t-s)}{ \sqrt{t-s}}\right)\,ds.
    \end{align*}
   To treat the last term of the right-hand side of the above inequality, we will use the fact that $\underset{x>0}{\sup}\left\{\sqrt{x}\left(2 - \ln(x)\right)\right\} = 2$.
    \begin{align*}
\displaystyle \int_{0}^{t} \left(Q\frac{\sqrt{t}-\sqrt{s} }{2(t-s)^{3/2}} - \frac{L_{\Lambda}}{2} \frac{\ln(t-s)}{ \sqrt{t-s}}\right)\,ds &= Q \int_0^1 \frac{1 - \sqrt{x}}{2(1-x)^{3/2}}\,dx - L_{\Lambda}\left(\sqrt{t}\ln(t) - 2\sqrt{t} \right) \\
&= Q \left(\left[(1-x)^{-1/2}(1-\sqrt x)\right]_0^1 + \int_0^1 \frac{dx}{2 \sqrt{x} \sqrt{1-x}} \right) + L_{\Lambda} \sqrt{t}\left( 2 - \ln(t)\right) \\
&=Q\left( \frac{\pi}{2}-1\right) + L_{\Lambda} \sqrt{t}\left( 2 - \ln(t)\right) \\
 &\le Q\left( \frac{\pi}{2}-1\right) + 2L_{\Lambda}.
    \end{align*}
Therefore,
\begin{align*}
    \displaystyle \E\left[\left\|e_{1}^{N,h}(t,.)\right\|_{L^1} \right] \leq& C_1 \left( h + \sqrt{h}\E\left[\left\|F^{N,h}(t,.) - F(t,.) \right\|_{L^1} \right] + \int_{0}^{t-h} \frac{h}{2(t-s)^{3/2}}\E\left[\left\|F(s,.) - F^{N,h}(s,.) \right\|_{L^1}\right]\,ds \right)
\end{align*}
where $\displaystyle C_1 =  \frac{L_{\lambda}}{\sigma}\sqrt{\frac{2}{\pi}} \left[1 \vee \left( 3\sqrt{\left(\sigma^2 + L_{\Lambda}^2 T\right)} + L_{\Lambda}(2 +\sqrt{T}) +  Q\left( \frac{\pi}{2}-1\right)\right)\right]$.
\end{itemize}
\end{proof}

Using Lemmas \ref{ControlE}, \ref{ControlE0}, \ref{ControlDetE}, \ref{ControlStochE} and \ref{ControlGam} and the fact that for $s\in[0,t-h]$, $\frac{h}{2(t-s)^{3/2}}\le \frac{1}{2 \sqrt{t-s}} $, we conclude the proof of Proposition \ref{ControlFinalE} for the choice $Z = 2 \max\left(L_{\Lambda}+C_1/2+C_{4,5}, C_{2,3} \right)$.

\begin{remark}
In Lemma \ref{estimGam}, we provide two estimations of $\displaystyle \E\left[\int_s^t \gamma^{i}_u d|K|_u\right]$. If we only use the first estimation \eqref{estimGam1} in the proof of Lemma \ref{ControlGam}, we obtain, using a decomposition that we will detail in Section \ref{SectBiais}, a rough estimation of $\E\left[\left\|e_{1}^{N,h}(t,.)\right\|_{L^1} \right]$ where we lose a $\ln(h)$ factor. 
\end{remark}

\subsection{Estimation of the bias}\label{secproofthm2}

We recall Equation \eqref{eqFh}: 
\begin{align*}
    &F^{N,h}(t,x) - F(t,x) \\
    &= \quad G_t*\left(F^{N,h}_0 - F_0\right)(x) - \displaystyle \int_0^t \partial_x G_{t-s} * \left(\Lambda(F^{N,h}(s,.)) - \Lambda(F(s,.))\right)(x)\,ds + R^{N,h}(t,x) + E^{N,h}(t,x), 
\end{align*}
and we shall use the expression of $E^{N,h}(t,x)$ proved in Lemma \ref{ControlE}. The next lemma provides an upper-bound of $\left\Vert \E\left[E^{N,h}(t,.) \right] \right\Vert_{L^1}$.
\begin{lem}\label{normL1ofE}
Assume that $\lambda$ is Lipschitz continuous and 
the initial positions are \begin{itemize}
\item either i.i.d. according to $m$ and $\displaystyle \int_{\R}\sqrt{F_0(x)(1 - F_0(x))}\,dx < \infty$,
  \item  or optimal deterministic and $\displaystyle \sup_{x \ge 1} x\int_{x}^{+\infty} \left(F_0(-y) + 1 - F_0(y) \right)\,dy<\infty$.
\end{itemize}Then
$$\exists Z_b <\infty, \forall N \in \N^{*}, \forall h \in (0,T], \quad \sup_{t\le T}\left\Vert \E\left[E^{N,h}(t,.) \right] \right\Vert_{L^1} \leq Z_b \left(\sqrt{\frac h N} + h\right).$$
\end{lem}
\begin{proof}
To estimate $\left\Vert \E\left[E^{N,h}(t,.) \right] \right\Vert_{L^1}$, we estimate each $\left\Vert \E\left[e^{N,h}_p(t,.) \right] \right\Vert_{L^1}, p \in \llbracket0,5\rrbracket$. From Lemmas \ref{ControlE0}, \ref{ControlDetE} and \ref{ControlStochE} we have  $\left\Vert \E\left[E^{N,h}(t,.) \right] \right\Vert_{L^1} \le \left\Vert \E\left[e_1^{N,h}(t,.) \right] \right\Vert_{L^1} + 2L_{\Lambda}h+ 2C_{4,5}h$. 
By Lemma \ref{ControlGam} and Corollary \ref{cor1}, we have: 
\begin{align*}
\displaystyle &\left\Vert \E\left[e_1^{N,h}(t,.) \right] \right\Vert_{L^1}\le  \E\left[\left\Vert e_1^{N,h}(t,.) \right\Vert_{L^1}\right]  \\
&\le C_1h + C_1\sqrt{h}\E\left[ \left\|F^{N,h}(t,.) - F(t,.) \right\|_{L^1} \right] + C_1{\mathbf 1}_{\{t\ge h\}}\int_{0}^{t-h} \frac{h}{2(t-s)^{3/2}}\E\left[ \left\|F^{N,h}(s,.) - F(s,.) \right\|_{L^1} \right]\,ds\\
&\le C_1h + C_1\sqrt{h} C \left( \frac{1}{\sqrt N} + h\right) + C_1{\mathbf 1}_{\{t\ge h\}}\int_{0}^{t-h} \frac{h}{2(t-s)^{3/2}}C\left( \frac{1}{\sqrt N} + h\right)\,ds \\
&\le C_1h+C_1C h^{3/2}+ C_1C \sqrt{\frac h N} + C_1C\left( \frac{1}{\sqrt N} + h\right)\sqrt h\\
&\le  \left(C_1 + 2C_1C \sqrt T \right)h +2C_1C\sqrt{\frac h N}.
\end{align*}
The conclusion holds with $Z_b = \max(2L_{\Lambda} + 2C_{4,5} + C_1 + 2C_1C \sqrt T, 2C_1C)$.
\end{proof}

The proof of Theorem \ref{thmBias} relies on the following Proposition that we will prove in Section \ref{SectBiais}.
\begin{prop}\label{normL2}
Assume that $\displaystyle \int_{\R}|x|m(dx) < \infty$ and $\lambda$ is Lipschitz continuous. Then:
  $$\exists M_b< \infty, \forall N \in \N^{*}, \forall h \in (0,T],\quad \sup_{t \leq T}\E\left[\left\Vert F^{N,h}(t,.) - F(t,.) \right\Vert^{2}_{L^2}\right] \le  M_b\left(\frac{1}{N} + h \right). $$
\end{prop}

\subsubsection{Proof of Theorem \ref{thmBias}}

Taking the expectation of Equation \eqref{eqFh} and using Lemma \ref{controlRN}, we obtain that $dx$ a.e.:
\begin{align*}
  \displaystyle &\E\left[ F^{N,h}(t,x)\right] - F(t,x)  \\
  &=  G_t * \E\left[F^N_0(x) - F_0(x) \right] -  \int_0^t \partial_xG_{t-s}*\E\left[ \left(\Lambda(F^{N,h}(s,.)) - \Lambda(F(s,.))\right)(x) \right]\,ds + \E\left[E^{N,h}(t,x) \right].
    \end{align*}
Besides, using Taylor-Young's inequality, we have that:
\begin{align*}
  \left| \Lambda\left(F^{N,h}(s,.)\right) - \Lambda(F(s,.)) - \lambda(F(s,.)) \left[F^{N,h}(s,.) - F(s,.) \right] \right| \le \frac{L_{\lambda}}{2}\left| \left( F^{N,h}(s,.) - F(s,.) \right)^2 \right|
\end{align*}
which implies:
\begin{align*}
  \left\Vert\E\left[ \Lambda(F^{N,h}(s,.))\right] - \Lambda(F(s,.))\right\Vert_{L^1} \le \left\Vert \lambda(F(s,.)) \right\Vert_{L^{\infty}} & \left\Vert \E\left[F^{N,h}(s,.)\right] - F(s,.)  \right\Vert_{L^1} + \frac{L_{\lambda}}{2}\E\left[\left\Vert F^{N,h}(s,.) - F(s,.) \right\Vert^{2}_{L^2}\right].
\end{align*}
Therefore, using the fact that $G_t$ is a probability density and the estimate \eqref{FirstDerivG} from Lemma \ref{EstimHeatEq}, we obtain:
\begin{align*}
  \displaystyle \left\Vert \E\left[ F^{N,h}(t,.)\right] - F(t,.) \right\Vert_{L^1} \le \left\Vert \E\left[ F^{N}_0 \right] - F_0 \right\Vert_{L^1} + &\sqrt{\frac{2}{\pi \sigma^2}}\int_0^t \frac{1}{\sqrt{t-s}} \Bigg\{ L_{\Lambda} \left\Vert \E\left[ F^{N,h}(s,.)\right]  - F(s,.) \right\Vert_{L^1} \\
  &+ \frac{L_{\lambda}}{2}\E\left[\left\Vert F^{N,h}(s,.) - F(s,.) \right\Vert^{2}_{L^2}\right] \Bigg\}\,ds + \left\Vert\E\left[E^{N,h}(t,x) \right]\right\Vert_{L^1} .
\end{align*} 

Using Lemma \ref{normL1ofE} and Proposition \ref{normL2} then Young's inequality, we deduce that:
\begin{align*}
  \displaystyle &\left\Vert \E\left[ F^{N,h}(t,.)\right] - F(t,.) \right\Vert_{L^1} \\
  &\le \left\Vert \E\left[ F^{N}_0 \right] - F_0 \right\Vert_{L^1} + \sqrt{\frac{2}{\pi \sigma^2}}\int_0^t \frac{1}{\sqrt{t-s}} \Bigg\{L_{\Lambda} \left\Vert \E\left[ F^{N,h}(s,.)\right] - F(s,.) \right\Vert_{L^1} +\frac{L_{\lambda}}{2} M_b \left(\frac 1 N + h \right)\Bigg\}\,ds + Z_b \left(\sqrt{\frac h N} + h\right) \\
  &\le \left\Vert \E\left[ F^{N}_0 \right] - F_0 \right\Vert_{L^1} + \left( \frac{1+\sqrt 2}{2}Z_b + \frac{L_{\lambda}M_b}{\sigma} \sqrt{\frac{2t}{\pi}} \right) \left(\frac{1}{N} + h \right) + \frac{L_{\Lambda}}{\sigma} \sqrt{\frac{2}{\pi}}\int_0^t \frac{1}{\sqrt{t-s}} \ \left\Vert \E\left[ F^{N,h}(s,.) \right] - F(s,.) \right\Vert_{L^1}\,ds. 
\end{align*} 
We iterate this inequality to obtain:

\begin{align*}
\displaystyle \left\Vert \E\left[ F^{N,h}(t,.) - F(t,.) \right]\right\Vert_{L^1} &\le \left( \frac{1+\sqrt2}{2} Z_b + \frac{L_{\lambda}M_b + (1 + \sqrt 2)L_{\Lambda}Z_b}{\sigma}\sqrt{\frac{2t}{\pi}}  + \frac{4 L_{\lambda}L_{\Lambda}M_bt}{\sigma^2} \right)\left(\frac{1}{N} +h \right) \\
&+ \left(1 + \frac{L_{\Lambda}}{\sigma}\sqrt{\frac{8t}{\pi}} \right) \left\Vert \E\left[ F^{N}_0 \right] - F_0 \right\Vert_{L^1}  + \frac{2L_{\Lambda}^2}{\sigma^2} \int_0^t \left\Vert \E\left[ F^{N,h}(r,.) - F(r,.) \right]\right\Vert_{L^1}\,dr.
\end{align*}
By Lemma \ref{integrF}, the application $t \mapsto \left\Vert \E\left[ F^{N,h}(t,.)\right] - F(t,.) \right\Vert_{L^1}$ is locally integrable $\forall h \in [0,T], N \in \N^{*}$. Therefore, we can apply Gr\"onwall's lemma and choosing $$ C_b = \max\left(1 + \frac{L_{\Lambda}}{\sigma}\sqrt{\frac{8T}{\pi}}, \frac{1+\sqrt2}{2} Z_b + \frac{L_{\lambda}M_b + (1 + \sqrt 2)L_{\Lambda}Z_b}{\sigma}\sqrt{\frac{2T}{\pi}}  + \frac{4L_{\lambda}L_{\Lambda}M_bT}{\sigma^2} \right) \exp\left(\frac{2 L_{\Lambda}^2}{\sigma^2}T \right)$$ concludes the proof of the theorem. 

\subsubsection{Proof of Proposition \ref{normL2}}\label{SectBiais}
For all $t,h \in [0,T], N \in \N^{*}$,  we use Jensen's inequality upon Equation \eqref{eqFh} and obtain:
\begin{align}\label{eqFhSqu}
   \displaystyle \E\left[\left\Vert F^{N,h}(t,.) - F(t,.) \right\Vert^{2}_{L^2}\right] \le 4 &\int_{\R}  \E\left[G_t*\left(F^{N,h}_0 - F_0\right)^2(x)\right]\,dx + 4 \int_{\R}\E\left[R^{N,h}(t,x)^2\right]\,dx + 4 \int_{\R}\E\left[E^{N,h}(t,x)^2\right]\,dx \\
   &+ 4\int_{\R}\E\left[\left( \int_0^t \partial_xG_{t-s} * \left(\Lambda(F^{N,h}(s,.)) - \Lambda(F(s,.))\right)(x)\,ds \right)^2 \right]\,dx. \notag
\end{align}

On the one hand, we have using the definition \eqref{defrnh} of $R^{N,h}(t,x)$, It\^o's isometry and the estimate \eqref{GSquareEsp} from Lemma \ref{EstimHeatEq} that: $$\displaystyle \int_{\R}\E\left[R^{N,h}(t,x)^2\right]\,dx = \frac{\sigma^2}{N^2} \sum \limits_{i=1}^N \int_0^t \mathbb{E} \left[ \int_{\mathbb{R}} G^2_{t-s}\left(X^{i,N,h}_s -x\right)dx \right] ds  = \frac{\sigma}{N}\sqrt{\frac{t}{\pi}}.$$
On the other hand, using Minkowski's, Young's and Cauchy-Schwarz's inequalities in addition to the estimate \eqref{FirstDerivG} from Lemma \ref{EstimHeatEq}, we get:
\begin{align*}
\displaystyle &\int_{\R}\E\left[\left(\int_0^t \partial_xG_{t-s}*\left(\Lambda\left(F^{N,h}(s,.)\right) - \Lambda(F(s,.))\right)(x)\,ds \right)^2 \right]\,dx =\E\left[\left\Vert \int_0^t \partial_xG_{t-s}*\left(\Lambda(F^{N,h}(s,.)) - \Lambda(F(s,.))\right)\,ds \right\Vert_{L^2}^2 \right] \\
&\le \E\left[\left( \int_0^t \left\Vert \partial_xG_{t-s} \right\Vert_{L^1}\left\Vert \Lambda(F^{N,h}(s,.)) - \Lambda(F(s,.)) \right\Vert_{L^2}\,ds\right)^2\right] \\
&\le \E\left[ \left(\int_0^t \sqrt{\frac{2 L_{\Lambda}^2}{\pi \sigma^2(t-s)}}\left\Vert F^{N,h}(s,.) - F(s,.) \right\Vert_{L^2}\,ds \right)^2 \right]\\
    &\le \frac{2L_{\Lambda}^2}{\pi \sigma^2}\int_0^t \frac{du}{\sqrt{t-u}} \int_0^t \frac{1}{\sqrt{t-s}} \E\left[\left\Vert F^{N,h}(s,.) - F(s,.) \right\Vert^2_{L^2} \right]\,ds \\
    &= \frac{4 L_{\Lambda}^2 \sqrt{t}}{\pi \sigma^2}\int_0^t \frac{1}{\sqrt{t-s}} \E\left[\left\Vert F^{N,h}(s,.) - F(s,.) \right\Vert^2_{L^2} \right]\,ds.
\end{align*}

Therefore, Inequality \eqref{eqFhSqu} becomes:
\begin{align*}
  \displaystyle \E\left[\left\Vert F^{N,h}(t,.) - F(t,.) \right\Vert^{2}_{L^2}\right] \le &4\int_{\R}\E\left[\left(F^{N,h}_0 - F_0\right)^2(x)\right]\,dx + \frac{4 \sigma}{N} \sqrt{\frac t \pi} + 4 \int_{\R}\E\left[E^{N,h}(t,x)^2\right]\,dx \\
  &+ \frac{16L_{\Lambda}^2 \sqrt{t}}{\pi \sigma^2}\int_0^t \frac{1}{\sqrt{t-s}} \E\left[\left\Vert F^{N,h}(s,.) - F(s,.) \right\Vert^2_{L^2} \right]\,ds . \notag
\end{align*}
As for the intialization term, since $|\tilde F_0^N-F_0|$ is not greater than $1/(2N)$, using \eqref{majow1det} for the second inequality, we obtain that
$$ \left\|\tilde F^N_0-F_0 \right\|_{L^2}^2\le \frac{\left\|\tilde F^N_0-F_0 \right\|_{L^1}}{2N}=\frac{ \mathcal{W}_1\left(\tilde \mu_0^N,m\right)}{2N}\le\frac{1}{2N}\displaystyle \int_{\R}|x|m(dx).$$
On the other hand, since $N\hat F^N_0(x)$ is a binomial random variable with parameter $(N,F_0(x))$, 
\begin{align*}
   N\int_{\R}\E\left[\left(\hat F^{N}_0  - F_0\right)^2(x)\right]\,dx=\int_{\R}F_0(x)(1-F_0(x))dx\le \int_{-\infty}^0F_0(x)dx+\int_0^{+\infty}(1-F_0(x))dx=\int_\R|x|m(dx).
\end{align*}
Therefore we have $\displaystyle \int_{\R}\E\left[\left(F^{N,h}_0 - F_0\right)^2(x)\right]\,dx\le\frac{1}{N}\int_\R|x|m(dx)$ for both the random and the deterministic optimal initializations.
With Lemma \ref{estimQuad} below which provides an estimation of the term $\displaystyle \int_{\R} \E\left[E^{N,h}(t,x)^2 \right]\,dx$, we deduce that:
\begin{align*}
    \displaystyle \E\left[\left\Vert F^{N,h}(t,.) - F(t,.) \right\Vert^{2}_{L^2}\right] &\le \frac{4}{N}\int_\R|x|m(dx) + \frac{4\sigma}{N}\sqrt{\frac t \pi}  + 4Q_bh + \frac{16 L_{\Lambda}^2 \sqrt{t}}{\pi \sigma^2}\int_0^t \frac{1}{\sqrt{t-s}} \E\left[\left\Vert F^{N,h}(s,.) - F(s,.) \right\Vert^2_{L^2} \right]\,ds.
\end{align*}
Iterating the previous inequality, we obtain:
\begin{align*}
\displaystyle &\E\left[\left\Vert F^{N,h}(t,.) - F(t,.) \right\Vert^{2}_{L^2}\right] \\
&\le 4\left(1 + \frac{32L_{\Lambda}^2t}{\pi \sigma^2} \right) \left(\frac{1}{N}\int_\R|x|m(dx) +\frac{\sigma}{N}\sqrt{\frac{t}{\pi}}+ Q_bh \right) + \frac{256 L_{\Lambda}^4 t}{\pi \sigma^4}\int_0^t  \E\left[\left\Vert F^{N,h}(r,.) - F(r,.) \right\Vert^2_{L^2} \right]\,dr.
\end{align*}
By Lemma \ref{integrF} and since $\left|F^{N,h}(t,.) - F(t,.)\right| \le 1$, the function $t \mapsto \E\left[\left\Vert F^{N,h}(t,.) - F(t,.) \right\Vert^{2}_{L^2} \right]$ is locally integrable for all $h \in [0,T],\;N \in \N^{*}$. We use Gr\"onwall's lemma once again and conclude for the choice $$ M_b = 4 \max\left\{Q_b\left(1 + \frac{32L_{\Lambda}^2T}{\pi \sigma^2} \right),\left(1 + \frac{32L_{\Lambda}^2t}{\pi \sigma^2} \right)\left(\int_{\R}|x|m(dx) + \sigma\sqrt{\frac{T}{\pi}} \right) \right\}\exp\left(\frac{256 L_{\Lambda}^4T^2}{\pi \sigma^4} \right).$$
\begin{lem}\label{estimQuad}
Assume that $\lambda$ is Lipschitz continuous. Then
$$\exists Q_b, \forall h \in (0,T], \forall t \in [0,T], \quad \displaystyle \int_{\R} \E\left[E^{N,h}(t,x)^2 \right]\,dx \le Q_b \; h .$$
\end{lem}
\begin{proof}
We have that $\displaystyle \int_{\R} \E\left[E^{N,h}(t,x)^2 \right]\,dx \leq 6 \sum \limits_{p=0}^5 \int_{\R} \E\left[e^{N,h}_p(t,x)^2 \right]\,dx $. For this reason, we shall estimate, in what follows, each $\displaystyle \int_{\R} \E\left[e^{N,h}_p(t,x)^2 \right]\,dx, p \in \llbracket0,5\rrbracket$. \\
On the one hand, for $r \in \left\{ 2,3\right\}$, we have using It\^o's isometry and the estimate \eqref{GPrimSqu} from Lemma \ref{EstimHeatEq} that $\forall h \in (0,T]$:
\begin{align*}
  \displaystyle \int_{\R} \E\left[e^{N,h}_r(t,x)^2 \right]\,dx &= \frac{\sigma^2}{N^2} \sum \limits_{i=1}^N \int_0^t (t \wedge \overline{\tau}^h_s -s)^2 \left(\lambda^N(i)\right)^2 \E \left[\int_{\R} \left(\partial_x G_{t-s}\left( Y^{i,N,h}_s -x\right)\right)^2\,dx\right]\,ds  \\
  &\le \frac{L_{\Lambda}^2 }{4 \sigma \sqrt{\pi} N} \int_0^t \frac{(t \wedge \overline{\tau}^h_s -s)^2}{(t-s)^{3/2}} \,ds \le \frac{2L_{\Lambda}^2 }{3 \sigma \sqrt{\pi}} \frac{h^{3/2}}{N},
\end{align*}
where the last inequality has already been derived at the end of the proof of Lemma \ref{ControlStochE}.
The Cauchy-Schwarz inequality then a similar reasoning implies that, for $r \in \{4,5\}$, $\displaystyle \int_{\R} \E\left[e^{N,h}_r(t,x)^2 \right]\,dx \le \frac{2L_{\Lambda}^4T }{3 \sigma^3 \sqrt{\pi}}h^{3/2}$. As for the term $e^{N,h}_0$, we have using the estimate \eqref{GSquare} that $\displaystyle \int_{\R} \E\left[e^{N,h}_0(t,x)^2 \right]\,dx \le \frac{4L_{\Lambda}^2}{\sigma\sqrt{\pi}}h^{3/2}$.\\

On the other hand, using Cauchy-Schwarz's inequality twice, \eqref{GPrimSqu} and the estimation \eqref{estimGam1}, we obtain:

\begin{align*}
    \displaystyle &\int_{\R} \E\left[e^{N,h}_1(t,x)^2 \right]\,dx = \E \left[\int_{\R} \left(\frac{1}{N}\sum \limits_{i=2}^N  \int_{0}^{t}(t \wedge \overline{\tau}^h_s -s)\left(  \lambda^N(i) - \lambda^N(i-1) \right)\partial_x G_{t-s}\left(Y^{i,N,h}_s -x \right)\gamma^{i}_s\,d|K|_s \right)^2\,dx \right] \\
    &\leq \E \left[ \int_{\R}\frac{1}{N}\sum \limits_{i=2}^N\left(\int_{0}^{t}(t \wedge \overline{\tau}^h_s -s)^2\left(\lambda^N(i) - \lambda^N(i-1)\right)^2\left(\partial_x G_{t-s}\right)^2\left(Y^{i,N,h}_s -x \right)\gamma^{i}_s\,d|K|_s\right)\left(\int_0^t \gamma^{i}_r\,d|K|_r\right)\,dx \right] \\
    &\leq \frac{L_{\lambda}^2}{4 \sigma^3 \sqrt\pi }\E \left[\frac{1}{N^3} \sum \limits_{i=2}^N \left(\int_0^t \frac{(t \wedge \overline{\tau}^h_s -s)^2}{(t-s)^{3/2}}\gamma^{i}_s\,d|K|_s\right)\left(\int_0^t \gamma^{i}_r\,d|K|_r\right) \right] \\
    &\leq \frac{L_{\lambda}^2}{4 \sigma^3 \sqrt\pi} \frac{1}{N^3} \sum \limits_{i=2}^N \E^{1/2}\left[\left(\int_0^t \frac{(t \wedge \overline{\tau}^h_s -s)^2}{(t-s)^{3/2}}\gamma^{i}_s\,d|K|_s\right)^2 \right]\E^{1/2}\left[\left(\int_0^t \gamma^{i}_r\,d|K|_r\right)^2 \right]\\
    &\le \frac{3L_{\lambda}^2}{4\sigma^3} \sqrt{\frac{t\left(\sigma^2 + L_{\Lambda}^2T\right)}{\pi}} \frac{1}{N^2} \sum \limits_{i=2}^N \E^{1/2}\left[\left(\int_0^t \frac{(t \wedge \overline{\tau}^h_s -s)^2}{(t-s)^{3/2}}\gamma^{i}_s\,d|K|_s\right)^2 \right].
\end{align*}

We denote by $m=\lceil\log_2\left(t/h \right) \rceil$ and rewrite the integral the following way:
$$\displaystyle \int_{0}^t\frac{(t \wedge \overline{\tau}^h_s -s)^2}{(t-s)^{3/2}}\gamma^{i}_s\,d|K|_s = \sum \limits_{k=0}^{m -1} \int_{t-t/2^{k}}^{t-t/2^{k+1}}\frac{(t \wedge \overline{\tau}^h_s -s)^2}{(t-s)^{3/2}}\gamma^{i}_s\,d|K|_s +  \int_{t-t/2^{m}}^{t}\frac{(t \wedge \overline{\tau}^h_s -s)^2}{(t-s)^{3/2}}\gamma^{i}_s\,d|K|_s.$$
Therefore,
  \begin{align*}
  \displaystyle \E^{1/2}\left[\left(\int_{0}^t \frac{(t \wedge \overline{\tau}^h_s -s)^2}{(t-s)^{3/2}}\gamma^{i}_s\,d|K|_s\right)^2 \right] &\le  \sum \limits_{k=0}^{m -1} \frac{h^2}{\left(\frac{t}{2^{k+1}}\right)^{3/2}} \E^{1/2}\left[\left(\int_{t-t/2^{k}}^{t-t/2^{k+1}} \gamma^{i}_s\,d|K|_s\right)^2 \right] + \sqrt{h} \E^{1/2}\left[\left(\int_{t-t/2^{m}}^{t} \gamma^{i}_s\,d|K|_s\right)^2 \right] \\
  &\le 3N \sqrt{\left(\sigma^2 + L_{\Lambda}^2 T\right)} \left(h^2 \sum \limits_{k=0}^{m -1}\left(\frac{t}{2^{k+1}} \right)^{-1} + h \right) \\
  &= 3N \sqrt{\left(\sigma^2 + L_{\Lambda}^2 T\right)} \left(\frac{2h^2}{t}(2^m -1) + h\right) \\
   &\le 15Nh \sqrt{\left(\sigma^2 + L_{\Lambda}^2 T\right)}. 
  \end{align*}
  We then have $\displaystyle \int_{\R} \E\left[e^{N,h}_1(t,x)^2 \right]\,dx \le \frac{45 L_{\lambda}^2\left(\sigma^2 + L_{\Lambda}^2 T\right)}{4 \sigma^3} \sqrt{\frac T \pi} \; h$. \\
  The conclusion holds for the choice $\displaystyle Q_b =\frac{1}{\sigma }\sqrt{\frac T \pi} \left\{\frac{16L_{\Lambda}^2}{3}\left(1 + \frac{L_{\Lambda}^2 T}{4 \sigma^2} \right) + \frac{45 L_{\lambda}^2}{4}\left(1 + \frac{L_{\Lambda}^2 T}{\sigma^2}\right) \right\}$.
\end{proof}


\subsection{Particle initialization : proof of \eqref{stindet}}\label{proofinitdet}

Since $\displaystyle \int_{\R}y^2m(dy)=2\left(-\int_{-\infty}^0 yF_0(y)\,dy+\int_0^{+\infty}y(1-F_0(y))\,dy\right)$, one has
  \begin{align*}
    \displaystyle \forall x>0,\quad \int_{-\infty}^{-x}F_0(y)\,dy+\int_x^{+\infty}(1-F_0(y))\,dy \le \frac{1}{2x}\int_{\R}y^2m(dy).
  \end{align*}
  Using the monotonicity of $F_0$ for the first inequality then the inequality $F_0(y)+(1-F_0(-y))=m((-\infty,-y]\cup (y,+\infty))\le 1$ valid for $y\ge 0$, we deduce that for all $x>0$,
  \begin{align}
   F_0(-x)+(1-F_0(x))\le \frac{2}{x}\int_{x/2}^x (F_0(-y)+(1-F_0(y))dy&\le \frac{2}{x}\int_{x/2}^{+\infty} (F_0(-y)+(1-F_0(y))dy\label{minointqueue}\\&\le\frac{2}{x^2}\int_{\R}y^2m(dy).\notag
  \end{align}
  Hence $\int_{\R}y^2m(dy)<\infty\Rightarrow \sup_{x\ge 0}x^2(F_0(-x)+(1-F_0(x)))<\infty$.
  Let us now suppose that $\sup_{x\ge 0}x^2(F_0(-x)+(1-F_0(x)))=C<\infty$ and deduce that $\displaystyle \int_{\R}|x|^{2-\varepsilon}m(dx)<\infty$ for all $\varepsilon\in (0,1]$ (and therefore all  $\varepsilon\in (0,2]$) and that $\sup_{N\ge 1}\sqrt{N}{\cal W}_1(\tilde\mu^N_0,m)<\infty$. We have
\begin{equation}
   \forall x>0,\;\int_x^{+\infty}(F_0(-y)+(1-F_0(y)))dy\le \int_x^{+\infty}\frac{C}{y^2}dy=\frac{C}{x}.\label{majointintsurv}
\end{equation}
 For $\varepsilon=1$, we have $$ \displaystyle \int_{\R}|x|m(dx)=\int_0^{+\infty}(F_0(-y)+(1-F_0(y)))\,dy \le 1+\int_1^{+\infty}(F_0(-y)+(1-F_0(y)))\,dy \le 1+C.$$

  Let now $\varepsilon\in (0,1)$. Writing $\frac{x^{2-\varepsilon}}{(2-\varepsilon)(1-\varepsilon)}=\frac{1}{1-\varepsilon}\int_0^x y^{1-\varepsilon}dy=\int_0^x\int_0^yz^{-\varepsilon}dzdy$ and using Fubini's theorem,  we obtain that
  \begin{align*}
  \displaystyle \frac{1}{(2-\varepsilon)(1-\varepsilon)}&\int_0^{+\infty} x^{2-\varepsilon}m(dx)= \int_0^{+\infty}z^{-\varepsilon} \int_z^{+\infty}(1-F_0(y))\,dy\,dz\\
  &\le \int_0^{1}z^{-\varepsilon} dz\int_0^{+\infty}(1-F_0(y))dy+\int_1^{+\infty}z^{-\varepsilon}\int_z^{+\infty}(1-F_0(y))\,dy\,dz.
  \end{align*}

 Combining this inequality with the symmetric one then using the above estimation of $\displaystyle \int_0^{+\infty}(F_0(-y)+(1-F_0(y)))\,dy$ and \eqref{majointintsurv}, we conclude that
  \begin{align*}
   \displaystyle \frac{1}{2-\varepsilon}\int_\R |x|^{2-\varepsilon}m(dx) &\le \int_0^{+\infty}\left(F_0(-y)+(1-F_0(y))\right)\,dy+(1-\varepsilon)C\int_1^{+\infty}z^{-1-\varepsilon}\,dz \le 1+\frac{C}{\varepsilon}.
  \end{align*}
Let $N\ge 1$. Since $1-F_0(x) \le 1/(2N)$ for $x\ge F_0^{-1}(1-1/(2N))$ and $F_0^{-1}(1-1/(2N))\ge F_0^{-1}(1/2)$, we have
\begin{align*}
  \int_{F_0^{-1}\left(\frac{2N-1}{2N}\right)}^{+\infty}(1-F_0(x))\,dx &\le \int_{F_0^{-1}\left(\frac{2N-1}{2N}\right) \wedge \sqrt{N}}^{\sqrt{N}}(1-F_0(x))\,dx + \int_{\sqrt{N}}^{+\infty}(1-F_0(x))\,dx\\&\le\frac{(\sqrt{N}-F_0^{-1}(1/2))^+}{2N}+\int_{\sqrt{N}}^{+\infty}(1-F_0(x))\,dx.\end{align*}
Dealing with $\int^{F_0^{-1}\left(\frac{1}{2N}\right)}_{-\infty}F_0(x)\,dx$ in a symmetric way and using \eqref{majointintsurv},
we deduce that
\begin{align}
   \int^{F_0^{-1}\left(\frac{1}{2N}\right)}_{-\infty}F_0(x)\,dx+\int_{F_0^{-1}\left(\frac{2N-1}{2N}\right)}^{+\infty}(1-F_0(x))\,dx\le \frac{\sqrt{N}-F_0^{-1}(1/2)\vee\left(-\sqrt{N}\right)}{2N}+\frac{C}{\sqrt{N}}.\label{majointqueuesbis}
\end{align}
 Either $F_0^{-1}(1-1/(2N))\le 0$ or $\frac{1}{2N}\le 1-F_0(F_0^{-1}(1-1/(2N))-)\le \frac{C}{(F_0^{-1}(1-1/(2N)))^2}$ so that $F_0^{-1}(1-1/(2N))\le \sqrt{2CN}$. By a symmetric reasoning, $F_0^{-1}(1/(2N))\ge -\sqrt{2CN}$. 
Since $|\tilde F_0^N-F_0|$ is not greater than $1/(2N)$, we deduce that $\int_{F_0^{-1}\left(\frac{1}{2N}\right)}^{F_0^{-1}\left(\frac{2N-1}{2N}\right)} \left|\tilde F_0^N(x)-F_0(x)\right|\,dx\le \sqrt{\frac{2C}{N}}$. With \eqref{w1fdr} and \eqref{majointqueuesbis}, we conclude that
\begin{align}
  \displaystyle {\cal W}_1(\tilde\mu^N_0,m) \le \sqrt{\frac{2C}{N}}+\frac{\sqrt{N}+|F_0^{-1}(1/2)|\vee \sqrt{N}}{2N}+\frac{C}{\sqrt{N}}.
\end{align}
Let us now suppose that $\sup_{N\ge 1}{\cal W}_1(\tilde\mu^N_0,m)<\infty$. Then $\tilde C:=\sup_{N\ge 1}\sqrt{N}\int^{F_0^{-1}\left(\frac{1}{2N}\right)}_{-\infty}F_0(x)\,dx<\infty$ and for $N\ge 1$, since $F_0(x)\ge\frac{1}{4N}$ when $x\ge F_0^{-1}\left(\frac{1}{4N}\right)$,
\begin{align*}
   \frac{F_0^{-1}\left(\frac{1}{2N}\right)-F_0^{-1}\left(\frac{1}{4N}\right)}{4N}\le \int^{F_0^{-1}\left(\frac{1}{2N}\right)}_{F_0^{-1}\left(\frac{1}{4N}\right)}F_0(x)\,dx\le \frac{\tilde C}{\sqrt{N}}.
\end{align*}
For $k\in\N^*$, we deduce that $F_0^{-1}\left(2^{-(k+1)}\right)-F_0^{-1}\left(2^{-k}\right)\ge -\tilde C 2^{\frac{k+3}{2}}$, and after summation that
$$F_0^{-1}\left(2^{-k}\right)\ge F_0^{-1}\left(1/2\right)-\frac{4\tilde C}{\sqrt{2}-1}(2^{\frac{k-1}{2}}-1).$$
With the monotonicity of $F_0^{-1}$, we deduce that
$$\forall u\in(0,1/2],\;F_0^{-1}(u)\ge F_0^{-1}(1/2)-\frac{4\tilde C}{\sqrt{2}-1}(u^{-1/2}-1),$$
an therefore that $\inf_{u\in(0,1/2]}(\sqrt{u}F_0^{-1}(u))>-\infty$. With the inequality $F_0^{-1}(F_0(x))\le x$ valid for $x\in\R$, this implies that
$\inf_{x\in\R:F_0(x)\le 1/2}(x\sqrt{F_0(x)})>-\infty$ and therefore that $\sup_{x\ge 0}x^2F_0(-x)<\infty$. With a symmetric reasoning, we conclude that $\sup_{x\ge 0}x^2(1-F_0(x))<\infty$.

\begin{remark}
  In fact, we have proved that
 \begin{align*}
   \sup_{x\ge 0}x^2\left(F_0(-x)+1-F_0(x)\right)<\infty&\Leftrightarrow\sup_{x\ge 0}x\int_x^{+\infty}\left(F_0(-y)+1-F_0(y)\right)dy<\infty\\&\Leftrightarrow\sup_{N\ge 1}\sqrt{N}{\cal W}_1(\tilde\mu^N_0,m)<\infty\Leftrightarrow\sup_{u\in(0,1/2]}\sqrt{u}\left(F_0^{-1}(1-u)-F_0^{-1}(u)\right)<\infty.
 \end{align*}
 Let us assume that if $\displaystyle \limsup_{x\to\infty}x\int_x^{+\infty}(F_0(-y)+1-F_0(y))dy>0$, which by \eqref{minointqueue} is implied by $\displaystyle \limsup_{x\to\infty}x^2(F_0(-x)+1-F_0(x))>0$ and, by monotonicity of the integral implies that $\limsup_{N\to\infty}x_N\int_{x_N}^{+\infty}(F_0(-y)+1-F_0(y))dy>0$ along any sequence $(x_N)_{N\in\N}$ of positive numbers increasing to $+\infty$ and such that $\limsup_{N\to\infty}\frac{x_{N+1}}{x_N}<\infty$. Then, under the equivalent conditions,
 $$\hat{C}:=\sup_{N\ge 1}\frac{F_0^{-1}(1-\frac{1}{2N})}{\sqrt{N}}\vee\sup_{N\ge 1}\frac{-F_0^{-1}(\frac{1}{2N})}{\sqrt{N}}\in(0,+\infty)$$ and
 $$\sqrt{N}{\cal W}_1(\tilde\mu^N_0,m)\ge \sqrt{N}\left(\int_{-\infty}^{F_0^{-1}(\frac{1}{2N})}F_0(y)dy+\int^{+\infty}_{F_0^{-1}(1-\frac{1}{2N})}(1-F_0(y))dy\right)\ge \sqrt{N}\int_{\hat{C}\sqrt{N}}^{+\infty}(F_0(-y)+1-F_0(y))dy$$
 so that $\limsup_{N\to\infty}\sqrt{N}{\cal W}_1(\tilde\mu^N_0,m)>0$.

\end{remark}

\section{Numerical experiments for the Burgers equation}

In order to confirm our theoretical estimates for the strong and the weak $L^1$-error between $F^{N,h}$ and its limit $F$, we consider, for the choice $\Lambda(u) = -(1-u)^2/2$ and the initial condition $F(0,x) = \mathbf{1}_{\{x \ge 0 \}}$, the following equation:
$$\left\{
  \begin{aligned}
    &\partial_tF(t,x) + \partial_x F(t,x) \Big(1 - F(t,x)\Big) = \frac{\sigma^2}{2}\partial_{xx}F(t,x),\\
    &F(0,x) = \mathbf{1}_{\{x \ge 0 \}}.
  \end{aligned}
\right.$$
We can notice that the function $\left(1-F(t,.)\right)$ is solution of the Burgers equation that was also used in \cite{BossTal}. The Cole-Hopf transformation yields the following closed-form expression of $F$:
\begin{align*}
\displaystyle F(t,x) = 1 - \frac{\mathcal{N}\left(\frac{t-x}{\sigma \sqrt t}\right)}{\mathcal{N}\left(\frac{t-x}{\sigma \sqrt t}\right) + \exp\left(\frac{2x-t}{2\sigma^2}\right)\mathcal{N}\left(\frac{x}{\sigma \sqrt t}\right)},
\end{align*}
where $\displaystyle \mathcal{N}\left(x\right) = \int_{-\infty}^x \frac{\exp(y^2/2)}{\sqrt{2\pi}}\,dy$.

The drift coefficient of the $i^{th}$ particle in the increasing order is then equal to $\lambda^{N}(i) = 1 -\frac{2i -1}{2N} $ and the Euler discretization with step $h \in (0,T]$ of the particle system is: 
\begin{align*}
    \displaystyle dX^{i,N,h}_t &=  \sigma dW^{i}_t + \left(1 + \frac{1}{2N} - \frac{1}{N} \sum \limits_{j=1}^{N} \mathbf{1}_{ \left\{ X^{j,N,h}_{\tau^{h}_{t}} \leq X^{i,N,h}_{\tau^{h}_{t}} \right\}} \right)\,dt,\;1\le i\le N,\;t\in[0,T].
\end{align*}
As $F_0$ is the cumulative function of the Dirac mass centered at zero, we place the $N$ particles at zero for their initialisation.\\

We seek to observe the dependence of the strong $L^1$-error $\E\left[ \mathcal{W}_1\left(\mu^{N,0}_T,\mu_T \right)\right]$ and the weak $L^1$-error $\mathcal{W}_1 \left(\E\left[\mu^{N,h}_T\right],\mu_T\right)$ at time $T$ on the number $N$ of particles and on the time step $h$. We recall \eqref{Wasserstein} and \eqref{w1fdr} where the Wasserstein distance between a probability measure $\nu$ and $\mu_T$ can be expressed either using the quantile functions or the cumulative distribution functions:
\begin{align*}
\mathcal{W}_1\left(\nu,\mu_T \right)&= \int_0^1 \left|F^{-1}_{\nu}(u) - F^{-1}_T(u) \right|\,du \\
&= \int_{\R}\left|F_{\nu}(x) - F(T,x)\right|\,dx.
\end{align*}
We choose to use the second expression because we have an explicit formula for $F(T,.)$ unlike the inverse $F^{-1}_T(.)$ (which can still be numerically estimated but this is costly and induces additional numerical error). \\

When $\nu$ is an empirical measure of the form $\frac{1}{N}\sum\limits_{i=1}^N \delta_{x^{i}}$, we choose to approximate the $\mathcal{W}_1$ distance not using a grid but in the following way. For $(y^{i})_{1 \le i \le N}$ denoting the increasing reordering of $(x^{i})_{1 \le i \le N}$, we have: $\mathcal{W}_1\left( \frac{1}{N}\sum\limits_{i=1}^N \delta_{x^{i}}, \mu_T \right)\simeq\Psi \left( y^1,y^2,\dots, y^N \right) $ where $$\displaystyle \Psi \left( y^1,y^2,\dots, y^N \right) = \sum \limits_{i=1}^{N-1}\frac{1}{2}\left( y^{i+1}-y^{i}\right)\left(\left|F(T,y^{i+1}) - \frac{i}{N}\right| + \left|F(T,y^{i}) - \frac{i}{N} \right|\right).$$

Therefore, for the strong $L^1$-error, $\left(Y^{i,N,h}_{t,r}\right)_{i \in \llbracket1,N\rrbracket}$ being the increasing reordering of the particles positions $\left(X^{i,N,h}_{t,r}\right)_{i\in \llbracket1,N\rrbracket}$, $t>0$ in the $r^{th}$ out of $R$ independent Monte-Carlo runs, we obtain the following approximation:
\begin{align*}
\E\left[ \mathcal{W}_1\left(\mu^{N,h}_T,\mu_T\right)\right] &\simeq \frac{1}{R} \sum \limits_{r =1}^R \Psi\left(Y^{1,N,h}_{T,r}, \dots, Y^{N,N,h}_{T,r}\right).
\end{align*}
We also define the precision of this estimation as half the width of the $95 \%$ confidence interval of the empirical error i.e. $\text{Precision} = 1.96 \times \sqrt{\text{Variance}/R}$ where Variance denotes the empirical variance over the runs of the empirical error over the particles. \\

Concerning the weak $L^1$-error, we approximate $\E\left[ \mu^{N,h}_T\right]$ by $\frac{1}{R \times N} \sum\limits_{r=1}^{R} \sum\limits_{i=1}^{N} \delta_{X^{i,N,h}_{T,r}}$. But as $R \times N$ will be as big as $10^8$ in our simulations, rather than using the previous grid free approximation, we use the grid $\left(F^{-1}_T\left(\frac{k}{K} \right)\right)_{1 \le k \le K-1}$ ($K$ will be chosen equal to $5000$) to compute the $\mathcal{W}_1$ distance. For $k \in \llbracket 0,K-1\rrbracket$ and $x \in \left[F^{-1}_T\left(\frac{k}{K} \right), F^{-1}_T\left(\frac{k+1}{K} \right) \right]$, we make the following approximation $F(T,x) \simeq \frac{2k+1}{2K}$. We also define the function $\varphi$ as:
\begin{align*}
\varphi \left(u_0, u_1, \dots, u_{K-1} \right) =  &\sum \limits_{k=1}^{K-2} \left|u_k - \frac{2k+1}{2K}\right| \left(F^{-1}_T\left(\frac{k+1}{K} \right)- F^{-1}_T\left(\frac{k}{K} \right) \right) + 2\left|u_0-\frac{1}{2K} \right|\left(F^{-1}_T\left(\frac{1}{K} \right)- F^{-1}_T\left(\frac{1}{2K} \right) \right) \\
&+ 2\left|u_{K-1} -\left(1 - \frac{1}{2K}\right) \right|\left(F^{-1}_T\left(1 -\frac{1}{2K} \right)- F^{-1}_T\left(1 - \frac{1}{K} \right) \right).
\end{align*}
Therefore, we can approach the weak $L^1$-error by $\mathcal{W}_1 \left(\E\left[\mu^{N,h}_T\right],\mu_T\right) \simeq \varphi \left(\left(\frac{1}{R}\sum\limits_{r=1}^R F^{N,h}_r\left(T,F^{-1}_T\left(\frac{2k+1}{2K} \right)\right) \right)_{0 \le k \le K-1}\right)$.
We divide the $R$ runs into $B$ batches of $M = R/B$ independent simulations in order to estimate the associated precision. Indeed, we estimate the empirical variance over the batches while estimating the weak error for each independent simulation over the batches. And by the delta method, we may expect, denoting $E=\E\left[ \left(F^{N,h}\left(T,F^{-1}_T\left(\frac{2k+1}{2K} \right) \right)\right)_{0 \le k \le K-1}\right]$, that

\begin{align*}
&\sqrt{\rho}\left[ \varphi \left(\left(\frac{1}{\rho}\sum\limits_{r=1}^\rho F^{N,h}_r\left(T,F^{-1}_T\left(\frac{2k+1}{2K} \right)\right) \right)_{0 \le k \le K-1}\right)- \varphi\left(E\right) \right] \\
  &\phantom{\sqrt{\rho}\left[ \varphi \left(\left(\frac{1}{\rho}\sum\right.\right.\right.}
    \overset{\mathcal L}{\longrightarrow} \mathcal{N}\Bigg(0,\nabla\varphi^{T}\left(E\right)\text{Cov}\left[ \left(F^{N,h}\left(T,F^{-1}_T\left(\frac{2k+1}{2K} \right) \right)\right)_{0 \le k \le K-1}\right]\nabla\varphi\left(E\right) \Bigg)
\end{align*}
when $\rho \to +\infty$. Applying this result with $\rho = M$ and $\rho = R$, one expects that $1/B$ times the empirical variance of $\left(\varphi \left(\left(\frac{1}{M}\sum\limits_{r=(b-1)M+1}^{bM} F^{N,h}_r\left(T,F^{-1}_T\left(\frac{2k+1}{2K} \right)\right) \right)_{0 \le k \le K-1}\right) \right)_{1 \le b \le B}$ over the batches provides an estimator of the variance of $\varphi \left(\left(\frac{1}{R}\sum\limits_{r=1}^R F^{N,h}_r\left(T,F^{-1}_T\left(\frac{2k+1}{2K} \right)\right) \right)_{0 \le k \le K-1}\right)$. So the precision is computed as $1.96$ times the square root of this estimator. \\

For both of the errors, we fix the time horizon $T=1$ and the diffusion coefficient $\sigma^2 = 0.2$. 

\subsection{Strong $L^1$-error behaviour}
We present numerical estimates of $\E\left[ \mathcal{W}_1\left(\mu^{N,h}_T,\mu_T\right)\right]$, computed as described above.  \\

$\blacktriangleright$ \textbf{Dependence on $N$}:\\

 We fix the time-step $h=0.002$ small enough in order to observe the effect of the number $N$ of particles on the error. The simulation is done with $R=100$ Monte-Carlo runs. We obtain the following results for the estimation of the error and the associated precision:\\

\begin{center}
	\begin{tabular}{ |c|c|c|c| }
		\hline
		\multicolumn{4}{|c|}{Evolution of the strong $L^1$-error w.r.t. $N$} \\ \hline
		Number of particles $N$ & Estimation  & Precision & Ratio of decrease \\ \hline
    $250$   & $0.03312361$ & $0.00290442$ & $\times$  \\ \hline
    $1000$  & $0.01598253$ & $0.00133181$ & $2.07$ \\ \hline
    $4000$  & $0.00841976$ & $0.00077491$ & $1.90$ \\ \hline
    $16000$ & $0.00358799$ & $0.00028319$ & $2.35$  \\ \hline
    $64000$ & $0.00193416$ & $0.00016111$ & $1.86$ \\ 
		\hline
	\end{tabular}
\end{center}

We observe that the ratio of successive estimations $\frac{\text{Estimation}(N/4)}{\text{Estimation}(N)}$ is around $2$ when we multiply $N$ by $4$, which means that the strong $L^1$-error is roughly proportional to $N^{-1/2}$.

\begin{remark}
This strong error was computed for the choice $\sigma^2 = 0.2$. When choosing the larger variance $\sigma^2 = 20$ with the same time-step $h = 0.002$, the error is approximately multiplied by $10$ and the $\mathcal{O}\left(N^{-1}\right)$ behaviour is still observed. For the smaller variance $\sigma^2 = 0.002$, we need to choose the smaller time-step $h = 0.001$ in order to recover the $\mathcal{O}\left(N^{-1}\right)$ behaviour in the number of particles since the estimated error is very small. 
 \end{remark}

$\blacktriangleright$ \textbf{Dependence on $h$}:\\

We apply the same strategy to study the dependence of the error on $h$ by choosing a large number $N = 500000$ of particles. The following table presents numerical estimates of the $L^1$-norm of the error and its associated precision for $R=100$ runs.\\

\begin{center}
  \begin{tabular}{ |c|c|c|c|}
    \hline
    \multicolumn{4}{|c|}{Evolution of the strong $L^1$-error w.r.t. $h$} \\ \hline
    Time-step $h$ & Estimation  & Precision & Ratio of decrease \\ \hline
    $1/2$   & $0.07963922$ & $5.8279 \times 10^{-5}$ & $\times$ \\ \hline
    $1/4$   & $0.03550774$ & $5.9151 \times 10^{-5}$ & $2.24$   \\ \hline
    $1/8$   & $0.01682159$ & $5.9598 \times 10^{-5}$ & $2.11$   \\ \hline
    $1/16$  & $0.00817936$ & $6.5275 \times 10^{-5}$ & $2.06$   \\ \hline
    $1/32$  & $0.00409004$ & $5.6482 \times 10^{-5}$ & $2.00$   \\ \hline
    $1/64$  & $0.00206127$ & $5.3539 \times 10^{-5}$ & $1.98$   \\ \hline
    $1/128$ & $0.00116977$ & $4.7468 \times 10^{-5}$ & $1.76$   \\ \hline
    $1/256$ & $0.00084118$ & $5.9673 \times 10^{-5}$ & $1.39$   \\ 
  \hline
  \end{tabular}
\end{center}
          
 We observe that when the time step $h$ is divided by $2$, the ratio of decrease $\frac{\text{Estimation}(h)}{\text{Estimation}(h/2)}$ is approximately equal to $2$.

 \begin{remark}
This strong error was computed for the choice $\sigma^2 = 0.2$. When choosing the larger variance $\sigma^2 = 20$ with the same number of particles $N=500000$, the error is of the same order for the $3$ first time-steps but deteriorates afterwards and becomes larger compared to errors obtained for $\sigma^2 = 0.2$. The ratio also deteriorates quickly and tends to be constant. In order to only observe the effect of the time-step on the weak error, we need to increase the number of particles when $\sigma$ is large. For the smaller variance $\sigma^2 = 0.002$, the error is smaller than the ones obtained for greater $\sigma^2$. The $\mathcal{O}\left(h\right)$-behaviour is recovered when the time-steps are small.
 \end{remark}

\subsection{Weak $L^1$-error behaviour}

We present numerical estimates of $\mathcal{W}_1 \left(\E\left[\mu^{N,h}_T\right],\mu_T\right)$, computed as described above.\\

$\blacktriangleright$ \textbf{Dependence on $N$}:\\

We fix the time-step $h=0.002$ small enough once again to observe the effect of the number $N$ of particles on the weak error. The estimation is done with $B=100$ batches of $M = 200$ independent simulations for a total of $R = 20000$ Monte-Carlo runs and $K=5000$. The results are shown in the following table:

\begin{center}
  \begin{tabular}{ |c|c|c|c| }
    \hline
    \multicolumn{4}{|c|}{Evolution of the weak $L^1$-error w.r.t. $N$} \\ \hline
    Number of particles $N$ & Estimation  & Precision & Ratio of decrease \\ \hline
    $100$  & $0.01018160$ & $5.6947 \times 10^{-4}$ & $\times$ \\ \hline
    $200$  & $0.00483151$ & $3.8455 \times 10^{-4}$ & $2.11$ \\ \hline
    $400$  & $0.00248807$ & $2.0485 \times 10^{-4}$ & $1.94$ \\ \hline
    $800$  & $0.00136491$ & $1.4707 \times 10^{-4}$ & $1.82$ \\ \hline
    $1600$ & $0.00077723$ & $1.0822 \times 10^{-4}$ & $1.76$ \\ \hline
    $3200$ & $0.00038285$ & $4.9747 \times 10^{-5}$ & $2.03$ \\ 
    \hline
  \end{tabular}
\end{center}

We observe that multiplying the number of particles by $2$ implies a division of the error estimation by approximately $2$ which proves that the weak $L^1$-error is roughly proportional to $N^{-1}$.

\begin{remark}
This weak error was computed for the choice $\sigma^2 = 0.2$. When choosing the larger variance $\sigma^2 = 20$ with the time-step $h = 0.002$, the $\mathcal{O}\left(N^{-1}\right)$ behaviour is still observed. When choosing the smaller variance $\sigma^2 = 0.05$ and the smaller time-step $h = 0.001$, the $\mathcal{O}\left(N^{-1}\right)$ behaviour is recovered. The weak error estimates only vary very slightly when varying $\sigma$. However, for smaller values of $\sigma$, the computation of the quantile function becomes difficult. 
 \end{remark}

$\blacktriangleright$ \textbf{Dependence on $h$}:\\

Once again, we do the same to study the dependence of the weak error on $h$ by choosing a large number $N=100000$ of particles, $B=20$ batchs of $M=50$ independent simulations for a total of $R = 1000$ Monte-Carlo runs and $K=5000$. \\

\begin{center}
  \begin{tabular}{ |c|c|c|c|}
    \hline
    \multicolumn{4}{|c|}{Evolution of the weak $L^1$-error w.r.t. $h$} \\ \hline
    Time-step $h$ & Estimation & Precision & Ratio of decrease  \\ \hline
    $1/2$   & $0.07954397$ & $4.7356 \times 10^{-5}$ & $\times$   \\ \hline
    $1/4$   & $0.03546112$ & $4.7932 \times 10^{-5}$ & $2.24$ \\ \hline
    $1/8$   & $0.01681185$ & $4.0437 \times 10^{-5}$ & $2.11$ \\ \hline
    $1/16$  & $0.00816986$ & $4.1616 \times 10^{-5}$ & $2.06$ \\ \hline
    $1/32$  & $0.00407191$ & $3.9306 \times 10^{-6}$ & $2.01$ \\ \hline
    $1/64$  & $0.00199744$ & $3.0719 \times 10^{-5}$ & $2.04$ \\ \hline
    $1/128$ & $0.00096767$ & $5.6043 \times 10^{-5}$ & $2.06$ \\ \hline
    $1/256$ & $0.00048294$ & $3.6172 \times 10^{-5}$ & $2.00$ \\ 
  \hline
  \end{tabular}
\end{center}

We observe that dividing the time step $h$ by $2$ implies a ratio of decrease $\frac{\text{Estimation}(h)}{\text{Estimation}(h/2)}$ greater or equal to $2$ which shows an $L^1$-weak error roughly proportional to $h$.

\begin{remark}
This weak error was computed for the choice $\sigma^2 = 0.2$. When choosing the smaller variance $\sigma^2 = 0.05$ with the same number of particles $N=100000$, the $\mathcal{O}\left(h\right)$ behaviour is recovered. When choosing the larger variance $\sigma^2 = 20$, the $\mathcal{O}\left(h\right)$ behaviour is observed for the first time-steps until $1/32$. However, for smaller time-steps, the error becomes larger, compared to the values obtained for smaller $\sigma$, and the ratio tends to be constant. In order to only observe the effect of the time-step on the weak error, we need to increase the number of particles when $\sigma$ is large.
 \end{remark}

\begin{appendix}
\section{Appendix}
The first lemma gives a condition under which we can interchange a Lebesgue and a stochastic integral. It is called the stochastic Fubini theorem and is a consequence of Theorem 2.2 proved by Veraar in \cite{VERAA}. 
\begin{lem}\label{StochFub}
 Let $V:[0,T] \times \R \times \Omega \to \R$ be a progressively measurable function. If $\displaystyle \int_{\R}\left( \int_0^T |V(t,x)|^2dt \right)^{1/2}dx < \infty$ almost surely then one has:
   $$ \forall t \in [0,T], \quad \text{a.s.}, \quad \displaystyle \int_{\R} \left(\int_0^t V(s,x)\,dW_s \right)\,dx  =  \int_0^t \left(\int_{\R} V(s,x)\,dx \right)\,dW_s. $$
\end{lem}

For $t>0$, let $G_t$ denote the probability density function of the normal law $\mathcal{N}(0,\sigma^2t)$: $$G_t(x) = \exp\left(-\frac{x^2}{2\sigma^2t}\right) \Big/ \sqrt[]{2\pi\sigma^2t} .$$ The following lemma provides a set of estimates that are very useful:
\begin{lem}\label{EstimHeatEq}
The function $G_t(x)$ solves the heat equation:
\begin{align}\label{heateq}
    \partial_t G_t(x) - \frac{\sigma^2}{2} \partial_{xx} G_t(x) = 0, \quad (t,x) \in [0,+\infty)\times \R.
\end{align}
We can express the square of the first order spatial derivative as:
\begin{align}\label{firstDerivSquar}
    \left(\partial_x G_t\right)^2(x) = \frac{x^2}{2 \sigma^5 t^{5/2} \sqrt{\pi}}G_{t/2}(x),
\end{align}
and deduce the $L^1$-norm of $(\partial_xG_t)^2$:
\begin{align}\label{GPrimSqu}
  \displaystyle \|(\partial_xG_t)^2\|_{L^1}=\|\partial_x G_t\|_{L^2}^2 = \frac{1}{4\sigma^3 t^{3/2} \sqrt{\pi}}.
\end{align}
Moreover, we have estimates of the $L^1$-norm of the spatial derivatives of $G$:
\begin{align}\label{FirstDerivG}
  \left\Vert \partial_xG_{t} \right\Vert_{L^1} = \sqrt{\frac{2}{\pi\sigma^2 t}},
\end{align}
\begin{align}\label{SecondDerivG}
    \left\Vert \partial_{xx}G_{t} \right\Vert_{L^1} \leq \frac{2}{\sigma^2 t}.
\end{align}
We may also compute the $L^1$-norm of $G^2_t$:
\begin{align}\label{GSquare}
    \displaystyle \|G_t^2\|_{L^1}=\|G_t\|_{L^2}^2 = \frac{1}{ 2 \sigma \, \sqrt{\pi t}},
\end{align}
which implies that for every measurable function $ \; y:[0,T] \to \R$,
\begin{align}\label{GSquareEsp}
    \displaystyle \int_{\R} \int_{0}^{t} G^2_{t-s}\left(y(s)-x\right)\,ds\,dx = \frac{1}{\sigma} \; \sqrt[]{\frac{t}{\pi}}.
\end{align}
\end{lem}
\begin{proof}
The second estimate is obtained by rewriting $\partial_{xx}G_{t}(x)$ as $\partial_{xx}G_t(x) = - \frac{1}{\sigma^2t}G_t(x) + \frac{1}{\sigma^2t} \left(-x \partial_xG_t(x) \right)$. We apply an integration by parts for the second term and obtain: $\displaystyle \int_{\R}\partial_{xx}G_t(x)dx \leq  \frac{2}{\sigma^2 t}\left\Vert G_{t} \right\Vert_{L^1} = \frac{2}{\sigma^2 t}$. \\
 As for the values of $\left\Vert G_{t}^2 \right\Vert_{L^1}$ and $\displaystyle \int_{\R} \int_{0}^{t} G^2_{t-s}\left(y(s)-x\right)\,ds\,dx $, we use the fact that $G^2_{t}(x) = G_{t/2}(x) / 2 \sigma \, \sqrt{\pi t}$. 
\end{proof}
\end{appendix}

\pagenumbering{gobble}
\bibliographystyle{plain}

\end{document}